\documentclass[10pt, reqno]{amsart}
\usepackage{latexsym}
\usepackage{graphicx}
\usepackage[USenglish]{babel}
\usepackage{amsfonts}
\usepackage{amsmath, amsthm, amssymb, mathtools}
\usepackage[colorlinks=true, citecolor=blue, linkcolor=blue, urlcolor=blue]{hyperref}
\usepackage{tikz}
\usepackage{multirow}
\usepackage{siunitx}
\usepackage{float}
\usepackage{caption}
\usepackage{subcaption}
\usepackage{enumitem}
\usepackage[dvipsnames]{xcolor}
\usepackage{booktabs}
\mathtoolsset{showonlyrefs}
\usepackage{url}
\usepackage[backend=biber, style=numeric, doi=true, url=true, isbn=false]{biblatex}
\DeclareFieldFormat{doi}{\url{https://doi.org/#1}}
\definecolor{nppcol}{rgb}{0.20,0.60,1.00}
\definecolor{gscol}{rgb}{0.85,0.20,0.55}
\definecolor{ruhecol}{rgb}{0.10,0.65,0.35}
\newcommand{\lsnpp}{\raisebox{2pt}{\tikz{\draw[solid,nppcol,line width=1.2pt] (0,0)--(0.40,0);}}}
\newcommand{\lsgs}{\raisebox{2pt}{\tikz{\draw[dashed,gscol,line width=1.2pt] (0,0)--(0.40,0);}}}
\newcommand{\lsruhe}{\raisebox{2pt}{\tikz{\draw[dotted,ruhecol,line width=1.2pt] (0,0)--(0.40,0);}}}

\theoremstyle{plain}
\newtheorem{theorem}{Theorem}[section]
\newtheorem{lemma}[theorem]{Lemma}

\theoremstyle{definition}

\newtheorem{example}[theorem]{Example}
\theoremstyle{remark}
\newtheorem{rmk}[theorem]{Remark}
\allowdisplaybreaks

\newcommand{\M}{\mathsf{M}}

\newcommand{\C}{\mathbb{C}}
\newcommand{\U}{\mathrm{U}}
\newcommand{\F}{\mathbb{F}}

\newcommand{\R}{\mathbb{R}}

\newcommand{\T}{^{\mathsf{T}}}

\newcommand{\rank}{\operatorname{rank}}

\newcommand{\colspace}{\operatorname{col}}

\newcommand{\argmin}{\operatorname{argmin}}
\newcommand{\argmax}{\operatorname{argmax}}

\title[The Normal Procrustes Problem]{The Normal Procrustes Problem: A Riemannian Optimization Approach}
\author{Kyle Bierly}
\thanks{Corresponding author: \texttt{kylebierly@math.ubc.ca}.\\
\indent Department of Mathematics, The University of British Columbia,
1984 Mathematics Road, Vancouver, BC V6T 1Z2, Canada.\\
\indent ORCID: 0009-0004-3173-264X}

\begin{document}
\begin{abstract}
    For given $m\times n$ data matrices $X,Y,$ we investigate the Normal Procrustes Problem---the least squares optimization problem that aims to minimize $\|AX-Y\|_F^2$, where $A$ is constrained to be a normal $m\times m$ matrix. As far as the author of this article is aware, no other method that attempts to solve the Normal Procrustes Problem exists in the literature; we thus propose what is, to our knowledge, the first such method. We, furthermore, adapt our approach to address the Real Normal Procrustes Problem, where $A$ must be real. In our treatment of these problems, we first reduce our complex and real objective functions to be purely optimizable over the Riemannian manifolds of the unitary and real orthogonal matrices, respectively. This reduction enables us to apply techniques in Riemannian manifold optimization to approximate solutions to both. The Closest Normal Matrix and Real Closest Normal Matrix Problems are both special cases of their respective Procrustes Problems (by setting $m=n$ and $X=I_m$) and have been previously studied in the literature. Our approach thus recovers a novel Riemannian optimization method for approximating solutions to both these problems. We further numerically test the performance of our method across all such problems (including against previously developed algorithms on the Closest Normal Matrix Problems) and obtain competitive residuals and favorable scaling in wall-clock time. 
    \end{abstract}
    \maketitle

\section{Introduction}\label{sec:intro}
Let the matrices $0\neq X,Y\in \M_{m\times n}(\F)$ (where $\F=\C$ or $\R$) be given. The $\mathcal{P}$-\textit{Procrustes Problem} aims to determine an optimal matrix $A\in \mathcal{P} \subseteq \M_{m}(\F)$ that minimizes
\begin{equation}\label{initialProblem}
    \|AX-Y\|_F^2,
\end{equation}
where $\|\cdot\|_F$ denotes the Frobenius norm. We assume $m \geq n \geq 1$ throughout without loss of generality; if $m<n$, right-multiplying by the
right singular vectors of $X$ reduces the problem to one with $\rank(X) \leq m$ columns,
up to an additive constant.

In this paper, we consider when $A$ is \textit{normal}, or \linebreak $A\in \mathcal{N}=\{B\in \M_m(\C)\mid B^*B=BB^*\}$; for the real normal case, we write \linebreak $A\in \mathcal{N}_{\R}=\{B\in \M_m(\R)\mid B\T B=BB\T\}$, such that the conjugate transpose, $^*,$ is replaced with the transpose, $\T$. 

The Normal Procrustes Problem (NPP) can be stated as determining
\begin{equation}\label{eq:procrustes}
   A^\star\in \underset{\substack{A \in \mathcal{N}}}{\argmin} \|AX - Y\|_F^2,
\end{equation}
where $X,Y\in \M_{m\times n}(\C)$. Note that this set may be empty; we do not attempt to determine whether the infimum is always attained by some $A \in \mathcal{N}$. 

The motivation for investigating the above problem arose in \cite{ltp}, in which necessary and sufficient conditions on $X$ and $Y$ are determined such that \eqref{initialProblem} is equal to zero; when $A\in \mathcal{P}$, this particular problem is referred to as the $\mathcal{P}$-\textit{Targeting Problem}. Such conditions were determined for various classes of matrices, such as when $A$ is unitary, Hermitian, an orthogonal projection matrix, and more. The case in which $A$ is normal was found to be particularly difficult. Indeed, smaller results were found, such as when the spectrum of a normal $A$ is constrained to lie in $\{\lambda,\mu\},$ where $\lambda,\mu \in \C$, but the general Normal Targeting Problem remains unsolved. 

By leveraging the singular value decomposition of $X$, one can often reduce the $\mathcal{P}$-Targeting Problem to a $\mathcal{P}$-\textit{Matrix Completion Problem} (see \cite[Theorem 2.2]{ltp}), in which the first $\rank(X)= r$ columns of $A$ are predetermined, and one must then establish whether there exist remaining $m-r$ columns such that some $A\in \mathcal{P}$ can be constructed. The Normal Matrix Completion Problem is a well-studied, difficult problem~\cite{scalone, jiang}. To the best of the author's knowledge, there does not currently exist a closed-form solution to the Normal Matrix Completion Problem, wherein lies the difficulty of solving the Normal Targeting Problem.

Hence, we take an alternative, more tractable approach to tackling this problem. Instead of determining the exact conditions when $AX=Y$, we turn to the more general NPP and approximate $A^\star$ via Riemannian manifold optimization techniques. Note that when $X=I_m$ (the identity matrix) and $Y\in \M_m(\C)$, the NPP recovers the \textit{Closest Normal Matrix Problem} (CNP). The CNP is studied in \cite{chu,scalone} (the real version, where $A\in \mathcal{N}_{\R}$ and $Y\in \M_m(\R)$) and \cite{kovac, gabriel, ruhe} (the complex version); we will later compare the optimization capabilities of this paper with those of the aforementioned papers, which are seemingly the most related prior research to the findings of this paper. Although various types of Procrustes Problems have been thoroughly investigated since 1962 \cite{gower, hurley}, the author was unable to find any prior literature specifically on the NPP, so the findings of this article are believed to be novel, and they offer a more general framework than the CNP literature by allowing $X$ to be a variable.  

Our approach employs the spectral decomposition property of normal matrices, through which every normal matrix ${A \in \M_m(\C)}$ admits a decomposition
\begin{equation}\label{eq:normalDecomp}
    A = UDU^*,
\end{equation} where $U \in \M_m(\C)$ is unitary, and
$D = \mathrm{diag}(d_1, \ldots, d_m) \in \M_m(\C)$ is diagonal and comprises the eigenvalues of $A$. An optimal $D$ can be determined in closed form in terms of $U$ and the given matrices $X$ and $Y$ in \eqref{eq:procrustes}. In Section \ref{sec:complexRed}, we do exactly this and reduce our objective \eqref{eq:procrustes} to an equivalent optimization problem over the matrix Riemannian manifold of the unitary matrices, $\mathrm{U}(m)$.

This strategy closely follows the approach of \cite{noferini}, in which the authors approximate the solution to the Closest $\Omega$-stable Matrix Problem (i.e., determining the nearest matrix with all its eigenvalues in a prescribed closed set $\Omega$) by following a similar reduction process over $\mathrm{U}(m)$. This manifold---viewed as a Riemannian submanifold of a Euclidean space---provides us with tools such as the Riemannian gradient and Hessian (Section \ref{sec:complexGradient}), enabling us to run second-order solvers over it within the MATLAB Riemannian manifold optimization package Manopt (version 8.0) \cite{manopt} via its built-in \texttt{unitaryfactory}. For further reading on the theory and applications related to optimization over Riemannian manifolds, see Boumal's textbook \cite{boumal}. 

Since our reduced, negated objective function is not geodesically concave over $\mathrm{U}(m)$ (demonstrated in Section~\ref{sec:complexRed}), Riemannian gradient methods cannot guarantee convergence toward a global maximum. However, our numerical findings (Section~\ref{sec:complexNum}) remain encouraging. For this complex case, our solver demonstrates strong consistency across random starts. In every trial the author conducted for which the global minimum residual is given to be zero beforehand, our algorithm attains a final residual of less than $10^{-10}$. Furthermore, we compare the performance of our model with two established algorithms for the CNP (again when $X = I_m$):
\begin{itemize}
    \item Ruhe's seminal Jacobi-type algorithm in \cite{ruhe}.
    \item Guglielmi and Scalone's more recent two-level gradient system algorithm in \cite{scalone}, which, although motivated by the Real CNP, is nonetheless generalizable to the complex case.
\end{itemize}
In both comparisons, our method attains essentially equivalent final residuals while scaling more favorably in wall-clock time as the amount of data increases.

In Section \ref{sec:realRed}, we consider the Real NPP, where we rewrite \eqref{eq:procrustes} such that $A\in \mathcal{N}_\R$ and $X,Y \in \M_{m\times n}(\R)$. We similarly, with a little more nuance than the complex version, reduce the objective to be over the orthogonal matrix Riemannian manifold $\mathrm{O}(m)$. In Section \ref{sec:realGradient}, we again compute the Riemannian gradient and Hessian of our objective function. Finally, in Section~\ref{sec:realNum}, we demonstrate, as in the complex case, numerical results for the Real NPP. Our solver terminates at suboptimal normal matrices more frequently than in the complex case, likely reflecting a greater nonconcavity of the real objective. We further observe that our method scales more favorably in solve time than Guglielmi and Scalone's algorithm as the amount of data increases. 

Our code for our numerical experiments is made publicly available
via the repository: \href{https://github.com/kbierly/NPP}{https://github.com/kbierly/NPP}.

\section*{Acknowledgements}
While this article was already being finalized, the author was made aware that Matvei Zhukov presented a poster titled ``Computing Nearest Normal Matrix Using Riemannian Optimization'' at Foundations of Computational Mathematics 2026 \cite{zhukov}, which the author did not attend. Vanni Noferini and Zhukov's preprint \cite{noferini2} was developed independently of ours, and became available on the arXiv eight days after the first version of this article was posted. While \cite{noferini2} features some overlap with the findings presented here, it is specific to the CNP and the Real CNP, which are subcases of the NPP and Real NPP (which our article treats in their full generality). However, \cite{noferini2} offers insights into the CNP that this article does not, such as a reduction of the problem to the flag manifold $\mathrm{U}(m)/\mathrm{U}(1)^m$ and a theoretical analysis of its extrema; we leave the extension of such ideas to the NPP as future work. 
The appeal to Rademacher's Theorem in Section~\ref{sec:realGradient} is drawn from \cite{noferini2}.

Nicolas Gillis suggested investigating whether Riemannian manifold optimization techniques could be implemented on the reduced objective \eqref{eq:fU} in Theorem~\ref{thm:reduced}. Gillis also made the author aware of Zhukov's presentation. Nicola Guglielmi and Carmela Scalone generously shared their implementation of their algorithm. Stephan Ramon Garcia gave helpful comments on an initial draft of the paper. Elina Robeva provided mentorship throughout the development of this paper. 

\section{Reduction of the NPP}\label{sec:complexRed}
Let $X,Y\in \M_{m\times n}(\C)$, and let $A\in \mathcal{N}$. In the following section, we work to reduce the NPP to an equivalent optimization problem over $\mathrm{U}(m)$.
Via the spectral decomposition of $A$ in \eqref{eq:normalDecomp} and under unitary invariance of the Frobenius norm, we can rewrite
\begin{equation}\label{eq:reduced}
    \|AX - Y\|_F^2 = \|UDU^*X - Y\|_F^2 = \|DU^*X - U^*Y\|_F^2=\sum_{i=1}^m \|d_i (U^*X)_i - (U^*Y)_i\|_F^2,
\end{equation}
where $(U^*X)_i$ and $(U^*Y)_i$ denote the $i$th rows of $U^*X$ and $U^*Y$ respectively. Since the terms in~\eqref{eq:reduced} are decoupled across $i$, we minimize over each $d_i$ independently.
Before we do so, we define:
\begin{align*}
\phi_i := \|(U^*X)_i\|_F^2, \quad \text{and}\quad \gamma_i := (U^*Y)_i(U^*X)_i^*.
\end{align*}
Note that the following lemma (essentially the solution to the Diagonal Procrustes Problem) is already well-known in the Procrustes literature---as noted in \cite[Section 4.2.2]{gillis}. We state the proof here for completeness.

\begin{lemma}\label{lem:optD}
  For fixed $U \in \mathrm{U}(m)$, a global minimizer of~\eqref{eq:reduced}
  over \linebreak $D = \mathrm{diag}(d_1,\ldots,d_m)$ is given by (for $i=1,\ldots,m$):
  \begin{equation}\label{eq:optdi}
      d_i^\star =  
      \begin{cases}
          \dfrac{\gamma_i}
                {\phi_i}
          & \text{if } (U^*X)_i \neq 0, \\[8pt]
          0 & \text{if } (U^*X)_i = 0.
      \end{cases}
  \end{equation}
\end{lemma}

\begin{proof}
  If $(U^*X)_i = 0$, the term $\|d_i (U^*X)_i - (U^*Y)_i\|_F^2 = \|(U^*Y)_i\|_F^2$ is
  independent of $d_i$; set $d_i^\star = 0$. Otherwise, $0<\phi_i$, and we complete the square:
  \begin{align*}
      \|d_i (U^*X)_i - (U^*Y)_i\|_F^2
      &= \phi_i\,|d_i|^2
        - 2\,\mathrm{Re}(\bar{\gamma}_i\, d_i)
        + \|(U^*Y)_i\|_F^2 \\
      &= \phi_i\left|d_i - \frac{\gamma_i}{\phi_i}\right|^2
        - \frac{|\gamma_i|^2}{\phi_i}
        + \|(U^*Y)_i\|_F^2.
  \end{align*}
  Since the first term is nonnegative and the remaining terms are independent of $d_i$,
  this is minimized if and only if
  \[
      d_i^\star = \frac{\gamma_i}{\phi_i}. \qedhere
  \]
\end{proof}

We write $D^\star_U$ to emphasize that the optimal diagonal depends on $U$; the scalars $\phi_i$, $\gamma_i$, and $d_i^\star$ likewise depend on $U$, though we suppress this in the notation. We can now state the main reduction:

\begin{theorem}\label{thm:reduced}
  The Normal Procrustes Problem~\eqref{eq:procrustes} is equivalent to determining
  \begin{equation}\label{eq:Uminimize}
      U^\star\in \underset{{U \in \mathrm{U}(m)}}{\argmax}\,f(U),
  \end{equation}
  where
  \begin{equation}\label{eq:fU}
      f(U) = \sum_{i=1}^m \frac{|\gamma_i|^2}{\phi_i}=\sum_{i=1}^m |d^\star_i|^2\phi_i=\|D^\star_U U^* X \|_F^2
  \end{equation}
  with the convention that the $i$th term equals zero when $(U^*X)_i = 0$. An optimal $A^\star = U^\star D^\star_{U^\star} (U^{\star})^*$
  is normal, where ${D^\star_{U^\star} = \mathrm{diag}(d_1^\star, \ldots, d_m^\star)}$
  is given by Lemma~\ref{lem:optD}.
\end{theorem}

\begin{proof}
  Suppose $(U^*X)_i \neq 0$. Substituting $d_i^\star = \gamma_i/\phi_i$
  into $\|d_i^\star (U^*X)_i - (U^*Y)_i\|_F^2$, the residual for the $i$th row is:
  \begin{align*}
      \left\|\frac{\gamma_i}{\phi_i}(U^*X)_i - (U^*Y)_i\right\|_F^2
      &= \frac{|\gamma_i|^2}{\phi_i}
        - \frac{2|\gamma_i|^2}{\phi_i}
        + \|(U^*Y)_i\|_F^2 \\
      &= \|(U^*Y)_i\|_F^2 - \frac{|\gamma_i|^2}{\phi_i}.
  \end{align*}
  Summing over $i$ yields
  \begin{align*}
      \sum_{i=1}^m \left(\|(U^*Y)_i\|_F^2 - \frac{|\gamma_i|^2}{\phi_i}\right)
      = \|U^*Y\|_F^2 - \sum_{i=1}^m\frac{|\gamma_i|^2}{\phi_i}
      = \|Y\|_F^2 - \sum_{i=1}^m\frac{|\gamma_i|^2}{\phi_i},
  \end{align*}
  where we used unitary invariance of the Frobenius norm. Since $\|Y\|_F^2$ is
  constant, minimizing over $\mathrm{U}(m)$ is equivalent to determining~\eqref{eq:Uminimize}.
\end{proof}

\begin{rmk}
    The objective $f$ admits a natural geometric interpretation. Let \linebreak ${\theta_i(U) \in [0, \pi/2]}$
be the Hermitian angle between the $i$th rows of $U^*X$ and $U^*Y$, defined by
\[
    \cos\theta_i(U)
    = \frac{|(U^*Y)_i(U^*X)_i^*|}
           {\|(U^*X)_i\|_F\,\|(U^*Y)_i\|_F}
\]
when both rows are nonzero. Then~\eqref{eq:fU} becomes
\begin{equation}\label{eq:fgeom}
    f(U) = \sum_{i=1}^m \|(U^*Y)_i\|_F^2\cos^2\theta_i(U),
\end{equation}
a weighted sum of squared cosines of Hermitian angles between corresponding rows
of $U^*X$ and $U^*Y$, weighted by the squared norms of the rows of $U^*Y$.
Thus, maximizing $f$ is equivalent to finding a unitary transformation that
simultaneously aligns each row $(U^*X)_i$ with the corresponding $(U^*Y)_i$
as well as possible (weighted by $\|(U^*Y)_i\|_F^2$).
\end{rmk}

Before we proceed to optimizing this function via Riemannian gradient techniques, 
we note that $f$ is not geodesically concave in general, meaning there exists a 
geodesic (a generalization of a straight line to the manifold, see \cite[Definition 5.38]{boumal}) along which the concavity 
inequality fails.

\begin{example}\label{ex:geodesicComplex}Take $m=n=2$, $X=I_2$, $Y=e_1e_1^*$, and consider the geodesic
\[
    c(t) = \begin{bmatrix}\cos(\tfrac{t\pi}{2}) & -\sin(\tfrac{t\pi}{2})\\
    \sin(\tfrac{t\pi}{2}) & \cos(\tfrac{t\pi}{2})\end{bmatrix}, \quad t\in[0,1].
\]
Since $X = I_2$, we have $\phi_i = 1$ for $i=1,2$. Computing directly via Theorem \ref{thm:reduced},
\begin{align*}
    f(c(t)) &= |\gamma_1|^2 + |\gamma_2|^2
             = \cos^4\!\left(\tfrac{t\pi}{2}\right) +
               \sin^4\!\left(\tfrac{t\pi}{2}\right).
\end{align*}
Thus, $f(c(0))=f(c(1))=1$ and $f\!\left(c\!\left(\tfrac12\right)\right)=\tfrac{1}{2}$. So,
\[
     \frac{1}{2}f(c(0))+\frac{1}{2}f(c(1))
     =1 > \frac{1}{2}
     = f\!\left(c\!\left(\tfrac{1}{2}\right)\right).
\]
\end{example}
We cannot guarantee that $f$ attains
its supremum on $\U(m)$, since $f$ may not be continuous. Furthermore, $f$ is not geodesically concave, so a Riemannian
gradient method may converge to a stationary point, which
need not be a global maximizer.

\section{Computing the Gradient and Hessian: Complex Case}\label{sec:complexGradient}

On the open dense subset of $\mathrm{U}(m)$ where $\phi_i \neq 0$ for all $i$, $f$ is smooth, so its Riemannian gradient $\mathrm{grad}_U f$ is well-defined; when $\rank(X)=m$---in particular for the CNP, where $X = I_m$---this is all of $\mathrm{U}(m)$. Where some $\phi_i = 0$---which can occur only when $\rank(X)<m$, since then some column of $U$ may lie in $\colspace(X)^\perp$---$f$ need not be smooth, as the $i$th term transitions to the constant $0$ by the convention of Theorem~\ref{thm:reduced}. For $X \neq 0$, each map $U \mapsto (U^*X)_i$ is real-analytic and not identically zero on $\U(m)$, so the locus has measure zero, and a random initialization avoids it with probability $1$ \cite{mityagin}. Although our solver may drive some $\phi_i$ toward $0$, the Cauchy--Schwarz inequality gives $|\gamma_i|^2/\phi_i \leq \|(U^*Y)_i\|_F^2$ while $\phi_i >0$, so the corresponding term of $f$ remains bounded. 

Viewing $\mathrm{U}(m)$ as a Riemannian submanifold of the Euclidean space \linebreak ${\R^{2m^2}\cong \M_m(\C)}$, equipped with the inner product ${\langle A, B\rangle = \operatorname{Re}(\operatorname{tr}(A^*B))}$, the Riemannian gradient is obtained by projecting the Euclidean gradient $\nabla_U f$---such that $f$ is considered as a function in the ambient Euclidean space---onto the tangent space $T_U\mathrm{U}(m)$, ensuring that optimization iterates remain on $\mathrm{U}(m)$.

Thus, we proceed by computing $\nabla_U f$. The Euclidean gradient and Hessian are defined for arbitrary ambient directions $V\in\M_m(\C)$. For the Riemannian gradient and Hessian, however, we restrict $V \in T_U\mathrm{U}(m)$. Let $\epsilon\in \R$. We write $\delta(h) := \frac{d}{d\epsilon}h(U+\epsilon V)\big|_{\epsilon=0}$, and record for the $i$th term of~\eqref{eq:fU}:
\begin{equation}\label{eq:deltas}
    \delta(\gamma_i) = (V^*Y)_i(U^*X)_i^* + (U^*Y)_i(V^*X)_i^*,
    \qquad
    \delta(\phi_i) = 2\,\mathrm{Re}\bigl[(U^*X)_i(V^*X)_i^*\bigr],
\end{equation}
from which the quotient rule gives
\begin{equation}\label{eq:ddstar}
    \delta(d_i^\star) = \frac{\delta(\gamma_i)}{\phi_i} 
    - d_i^\star\frac{\delta(\phi_i)}{\phi_i}.
\end{equation}
Applying the product rule to $f(U) = \sum_i |d_i^\star|^2\phi_i$, using $\delta(|d_i^\star|^2) = 2\,\mathrm{Re}(\bar{d_i^\star}\delta(d_i^\star))$, substituting~\eqref{eq:deltas} into~\eqref{eq:ddstar}, and summing over $i$:
\begin{align*}
    \delta\!\left(f\right)
    &= 2\sum_{i=1}^m \Big( \underbrace{\,\mathrm{Re}\bigl[\bar{d_i^\star}(U^*Y)_i(V^*X)_i^*\bigr]}_{\text{(A)}}
     + \underbrace{\,\mathrm{Re}\bigl[\bar{d_i^\star}(V^*Y)_i(U^*X)_i^*\bigr]}_{\text{(B)}}\\
     &~~~- \underbrace{|d_i^\star|^2\,\mathrm{Re}\bigl[(U^*X)_i(V^*X)_i^*\bigr]}_{\text{(C)}}\Big).
\end{align*}
Further applying the identity $\sum_i r_i s_i t_i^* = \mathrm{tr}(RST^*)$ for $R = \mathrm{diag}(r_i)$ and matrices $S$, $T$ with respective $i$th rows
$s_i$, $t_i$, together with cyclicity of the trace and
${\mathrm{Re}(\mathrm{tr}(S)) = \mathrm{Re}(\mathrm{tr}(S^*))}$:
\begin{align*}
&\text{(A):}\quad \sum_{i=1}^m \mathrm{Re}\bigl[\bar{d_i^\star}(U^*Y)_i(V^*X)_i^*\bigr]
= \mathrm{Re}\bigl[\mathrm{tr}\!\left((D^\star_U)^* U^*YX^*V\right)\bigr]\\
&\phantom{\text{(A):}\quad  \sum_{i=1}^m \mathrm{Re}\bigl[\bar{d_i^\star}(U^*Y)_i(V^*X)_i^*\bigr]}
= \mathrm{Re}\bigl[\mathrm{tr}\!\left((XY^*UD^\star_U)^*V\right)\bigr],\\
&\text{(B):}\quad \sum_{i=1}^m \mathrm{Re}\bigl[\bar{d_i^\star}(V^*Y)_i(U^*X)_i^*\bigr]
= \mathrm{Re}\bigl[\mathrm{tr}\!\left((D^\star_U)^* V^*YX^*U\right)\bigr]\\
&\phantom{\text{(B):}\quad \sum_{i=1}^m \mathrm{Re}\bigl[\bar{d_i^\star}(V^*Y)_i(U^*X)_i^*\bigr]}
= \mathrm{Re}\bigl[\mathrm{tr}\!\left((YX^*U(D^\star_U)^*)^*V\right)\bigr],\\
&\text{(C):}\quad \sum_{i=1}^m |d_i^\star|^2\,\mathrm{Re}\bigl[(U^*X)_i(V^*X)_i^*\bigr]
= \mathrm{Re}\bigl[\mathrm{tr}\!\left(|D^\star_U|^2 U^*XX^*V\right)\bigr]\\
&\phantom{\text{(C):}\quad \sum_{i=1}^m |d_i^\star|^2\,\mathrm{Re}\bigl[(U^*X)_i(V^*X)_i^*\bigr]}
= \mathrm{Re}\bigl[\mathrm{tr}\!\left((XX^*U|D^\star_U|^2)^*V\right)\bigr].
\end{align*}
Combining these terms yields
\begin{align*}
    \delta(f) &= \mathrm{Re}\!\left[\mathrm{tr}\!\left(\bigl(2(XY^*UD^\star_U
+ YX^*U(D^\star_U)^* - XX^*U|D^\star_U|^2)\bigr)^*V\right)\right]\\
&= \mathrm{Re}\!\left(\mathrm{tr}\!\left((\nabla_U f)^* V\right)\right),
\end{align*}
from which we identify
\begin{equation}\label{eq:euclidean_grad}
    \nabla_U f = 2\bigl(XY^*UD^\star_U + YX^*U(D^\star_U)^* - XX^*U|D^\star_U|^2\bigr).
\end{equation}
We now project $\nabla_U f$ onto the tangent space $T_U\mathrm{U}(m)$ 
to obtain the Riemannian gradient. Since the Lie algebra of $\mathrm{U}(m)$ 
is the space of skew-Hermitian matrices, \cite[Lemma~3.2]{bahar} gives $T_U\mathrm{U}(m) = \{US : S^* = -S\}$.
Projecting $\nabla_U f$ onto $T_U\mathrm{U}(m)$ via the skew-Hermitian 
part operator ${\operatorname{skew}(Z) = \frac{1}{2}(Z - Z^*)}$ yields 
the Riemannian gradient
\[
\operatorname{grad}_U f = U\operatorname{skew}(U^*\nabla_U f).
\]

We next compute the Riemannian Hessian $\operatorname{Hess}_U f[V]$. Note that determining the exact Hessian is not strictly necessary even for a second-order Riemannian manifold optimizer like \texttt{trustregions} within Manopt, which approximates the Hessian well using finite differences on the gradient. However, as Boumal states in \cite[Section 6.4.1]{boumal}, if it is practical to compute the Hessian, the resulting ``enhanced accuracy of the model is a strong incentive to do so.'' Thus, for improved computational efficiency and completeness, we determine the Hessian here. 

To do so, we follow a 
similar process (outlined in \cite[Corollary~5.16]{boumal}) to that of the 
Riemannian gradient; we evaluate ${\delta(\nabla_U\, f)}$, then 
project the result back onto $T_U\mathrm{U}(m)$. Boumal carries out exactly 
this computation for a general smooth function in 
\cite[Equation~(7.39)]{boumal} over the real orthogonal group $\mathrm{O}(n)$. The unitary 
case is analogous and gives
\begin{equation}\label{eq:genHessian}
    \operatorname{Hess}_U f[V] 
= U\operatorname{skew}\!\left(U^* H_U f[V] 
- U^*V\operatorname{herm}(U^*\nabla_U f)\right),
\end{equation}
where $\operatorname{herm}(Z) = \tfrac{1}{2}(Z + Z^*)$, and $\delta(\nabla_U f)=H_U f[V]$ 
is the Euclidean Hessian. It remains to compute $H_U f[V]$.

Consider the perturbation of~\eqref{eq:euclidean_grad} in the direction 
$V\in T_U\U(m)$,
\begin{align*}
\nabla_{U+\epsilon V} f &= 2\bigl(XY^*(U+\epsilon V)D^\star_{U+\epsilon V} 
+ YX^*(U+\epsilon V)(D^\star_{U+\epsilon V})^* \\
&~~~- XX^*(U+\epsilon V)|D^\star_{U+\epsilon V}|^2\bigr).
\end{align*}
Applying the product rule to each of the three terms and 
evaluating at $\epsilon = 0$ gives
\begin{equation}\label{eq:hess}
\begin{aligned}
    H_U f[V] = 2\big(
    &XY^*VD^\star_U + XY^*U\,\delta(D^\star_U) \\
    &+ YX^*V(D^\star_U)^* + YX^*U\,\delta(D^\star_U)^* \\
    &- XX^*V|D^\star_U|^2 - XX^*U\,\delta(|D^\star_U|^2)\big),
\end{aligned}
\end{equation}
where $\delta(D^\star_U)$ and $\delta(|D^\star_U|^2)$ are computed as we did for the gradient. Substituting \eqref{eq:hess} into~\eqref{eq:genHessian} yields the desired result.
\section{Numerical Experiments: Complex Case}\label{sec:complexNum}
We now evaluate how effective our findings are at optimizing \eqref{eq:Uminimize}. We employ the Riemannian trust-region solver first introduced in \cite{rtr}, outlined in \cite[Algorithm~6.3]{boumal}, and available as \texttt{trustregions} within Manopt. This algorithm fits a quadratic model to $f$ at the iterate $U_k \in \mathrm{U}(m)$ using the gradient and Hessian, and then optimizes this model over a disk (the \textit{trust region}) in the 
tangent space $T_{U_k}\mathrm{U}(m)$. The optimal step is retracted back onto $\mathrm{U}(m)$ to produce a tentative iterate $\tilde{U}_k$, whose actual cost $f(\tilde{U}_k)$ is compared against the value predicted by the model. If the model's prediction is sufficiently accurate, the step is accepted ($U_{k+1} = \tilde{U}_k$); otherwise it is rejected ($U_{k+1} = U_k$). The trust-region radius is then reduced, maintained, or enlarged according to the accuracy of the model, and the process repeats. In practice, \texttt{trustregions} terminates once the Riemannian gradient norm $\|\operatorname{grad}_U f\|_F$ falls below a prescribed tolerance (we use the Manopt default of $10^{-6}$), at which point the returned iterate approximately satisfies the first-order criticality condition in \cite[Definition 4.4]{boumal}.

\subsection{Optimization Performance on Random Matrices}\label{ssec:complexPerform}~\\[8pt]
\indent We evaluate the solver's performance over given $X,Y\in \M_{m\times n}(\C)$ with \linebreak
${m \in \{10, 20, 50\}}$, $n \in \{0.2m,\, 0.5m,\,0.8m,\, m\}$, and either complex random Gaussian entries (MATLAB's \texttt{randn(m,n) + 1i*randn(m,n)}) or complex random uniform entries (MATLAB's \texttt{rand(m,n) + 1i*rand(m,n)}). For each instance, we run three random starts on $\U(m)$. The author also independently tested an initialization built from the Schur decomposition of $YX^{\dag}$ (where $^\dag$ denotes the Moore--Penrose pseudoinverse), intended as a least-squares guess at $U^\star$; however, such starts produced degeneracies---most notably for the CNP with nilpotent $Y$, where the start lands on a global minimum with $f(U_0) = 0$ and the solver terminates immediately---while offering no advantage in final residuals or timing. All reported results therefore use random initialization, for which we encountered no such degeneracy. 
Note that for brevity, we only include the residual plots when $n \in \{0.5m, m\}$ in this paper, but all of our plots---including plots of the gradient norm---are available for download at \url{https://github.com/kbierly/NPP}. Additionally, the output of Manopt's \texttt{checkgradient} and \texttt{checkhessian} tools, which we employ to verify our numerical implementation, can be found in our repository. All experiments (as in future sections) were performed on an Intel Core Ultra 9 285H CPU 2.90 GHz with 16 GB of RAM, running MATLAB R2026a (LAPACK 3.11.0 and MKL 2025.0.1) on a Windows system. All reported times are wall-clock times.
\begin{figure}[H]
    \centering
    \begin{minipage}[t]{0.45\textwidth}\vspace*{0pt}
    \includegraphics[width=\textwidth]{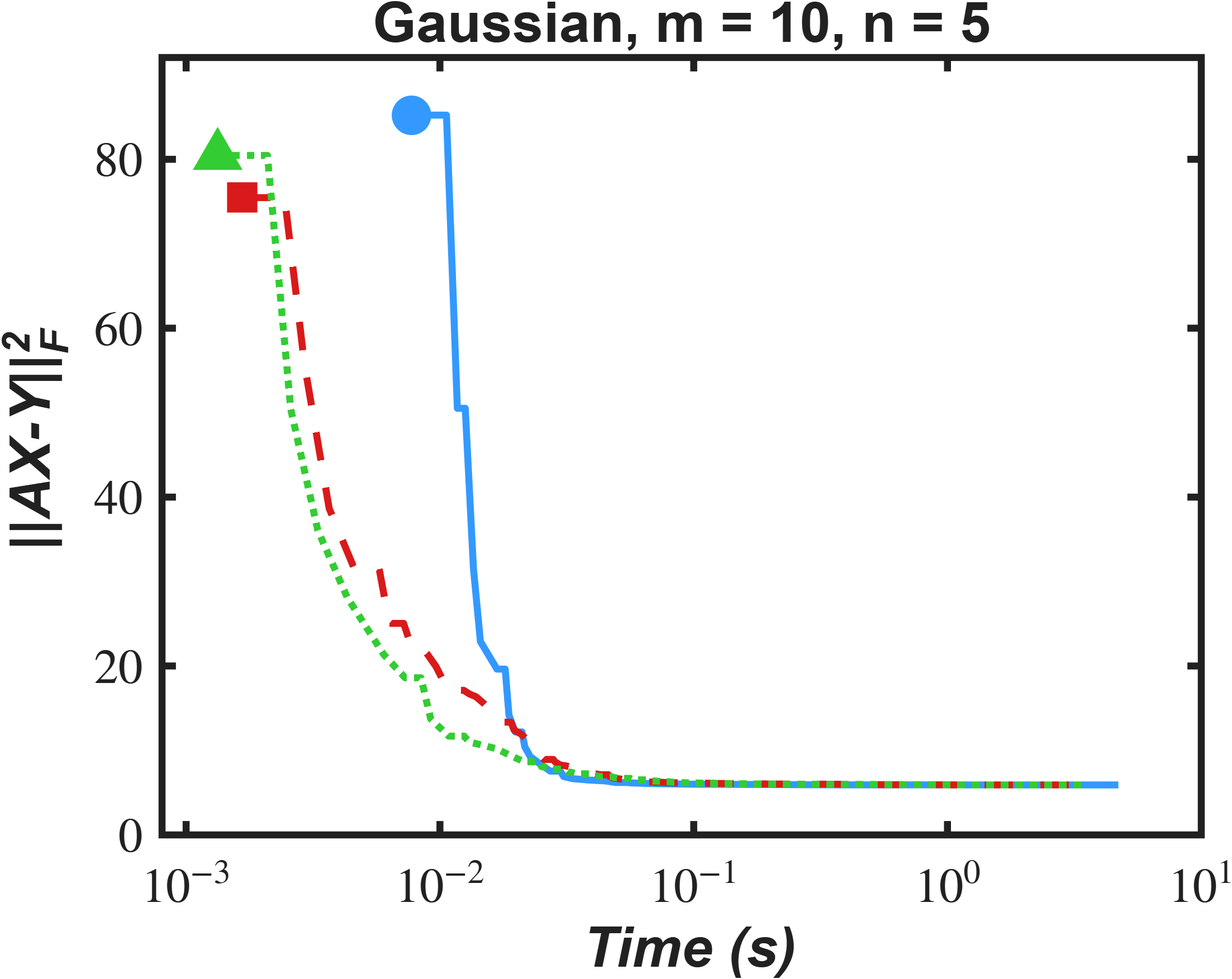}
\end{minipage}\hfill
\begin{minipage}[t]{0.45\textwidth}\vspace*{0pt}
    \includegraphics[width=\textwidth]{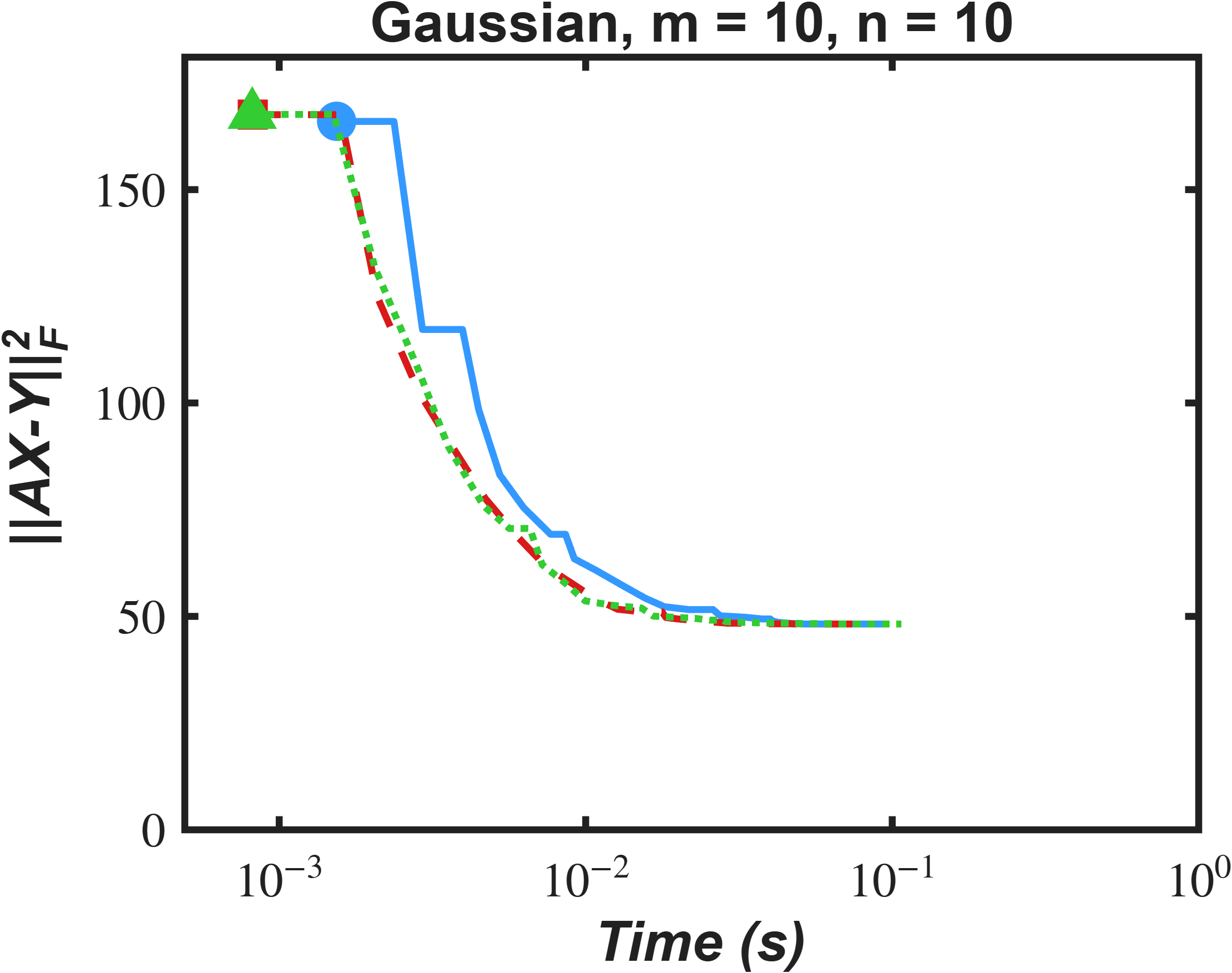}
\end{minipage}\hfill
\begin{minipage}[t]{0.10\textwidth}\vspace*{6.5pt}
    \includegraphics[width=\textwidth]{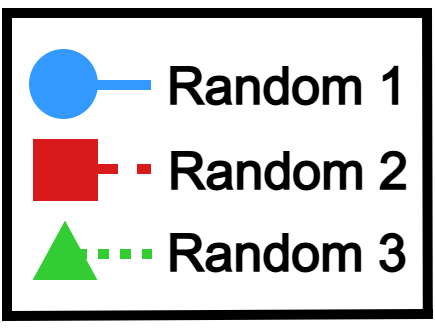}
\end{minipage}\\[5pt]
\begin{minipage}[t]{0.45\textwidth}\vspace*{0pt}
    \includegraphics[width=\textwidth]{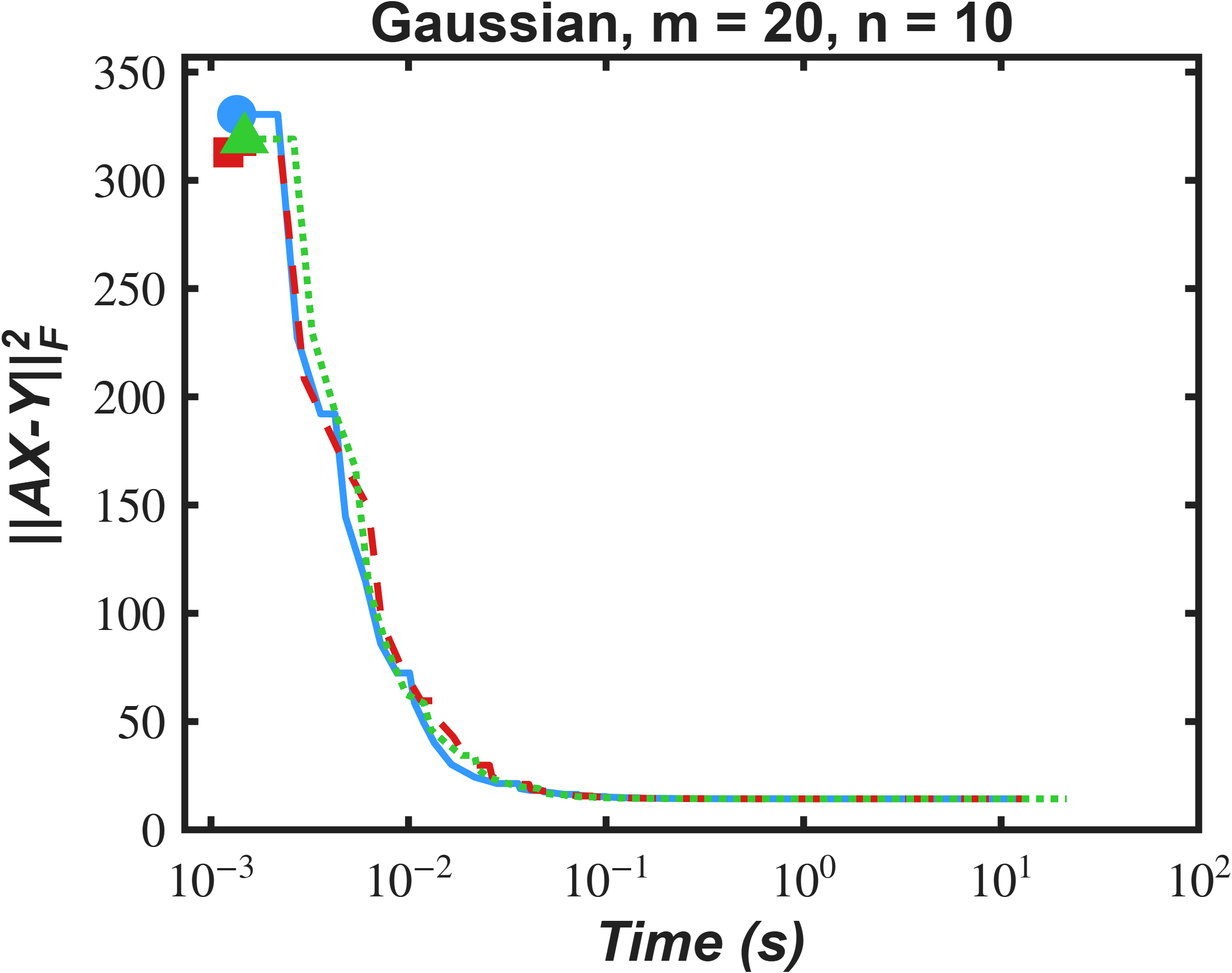}
\end{minipage}\hfill
\begin{minipage}[t]{0.45\textwidth}\vspace*{0pt}
    \includegraphics[width=\textwidth]{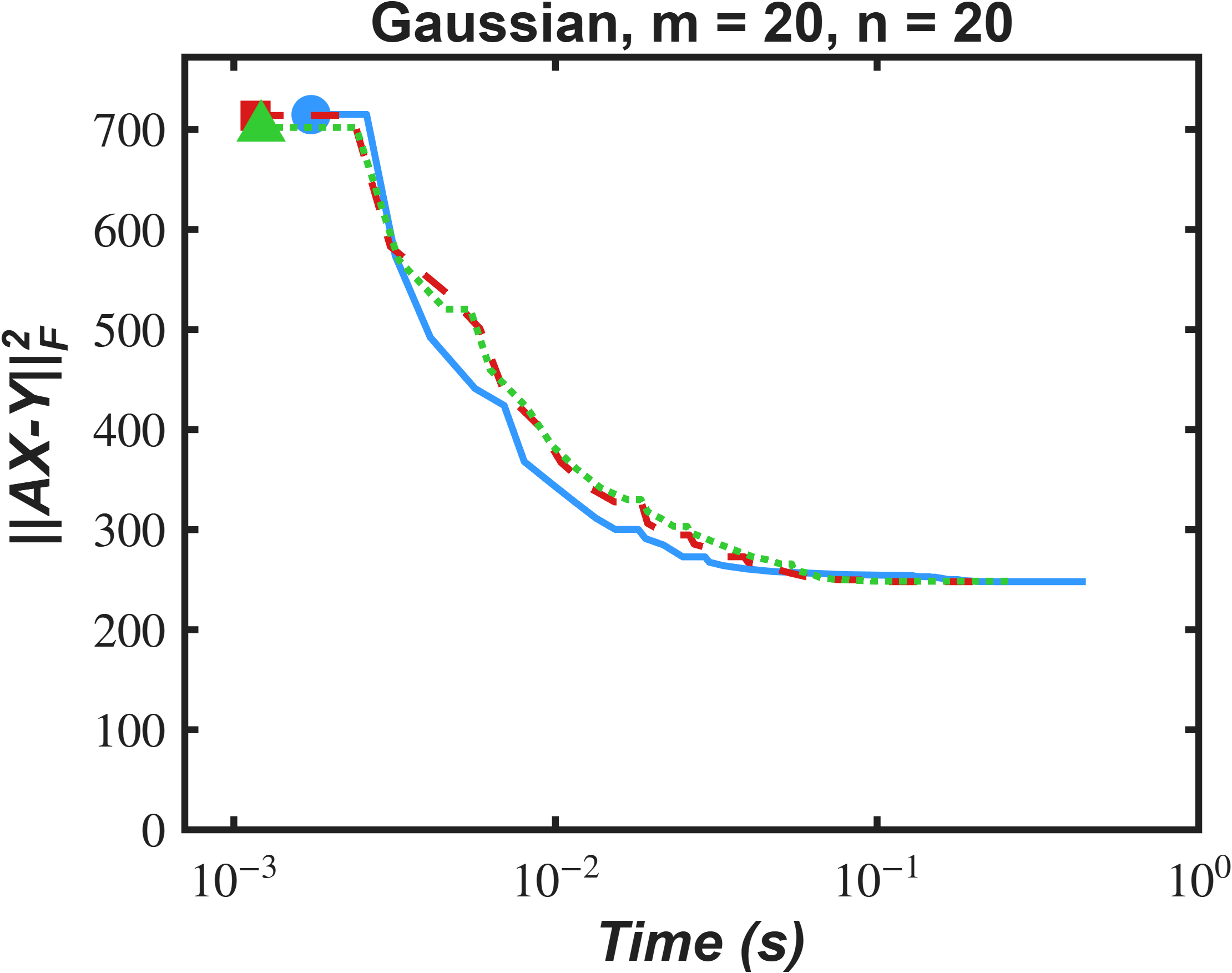}
\end{minipage}\hfill
\begin{minipage}[t]{0.10\textwidth}\vspace*{0pt}
    \phantom{\includegraphics[width=\textwidth]{Fig1c.png}}
\end{minipage}\\[5pt]
\begin{minipage}[t]{0.45\textwidth}\vspace*{0pt}
    \includegraphics[width=\textwidth]{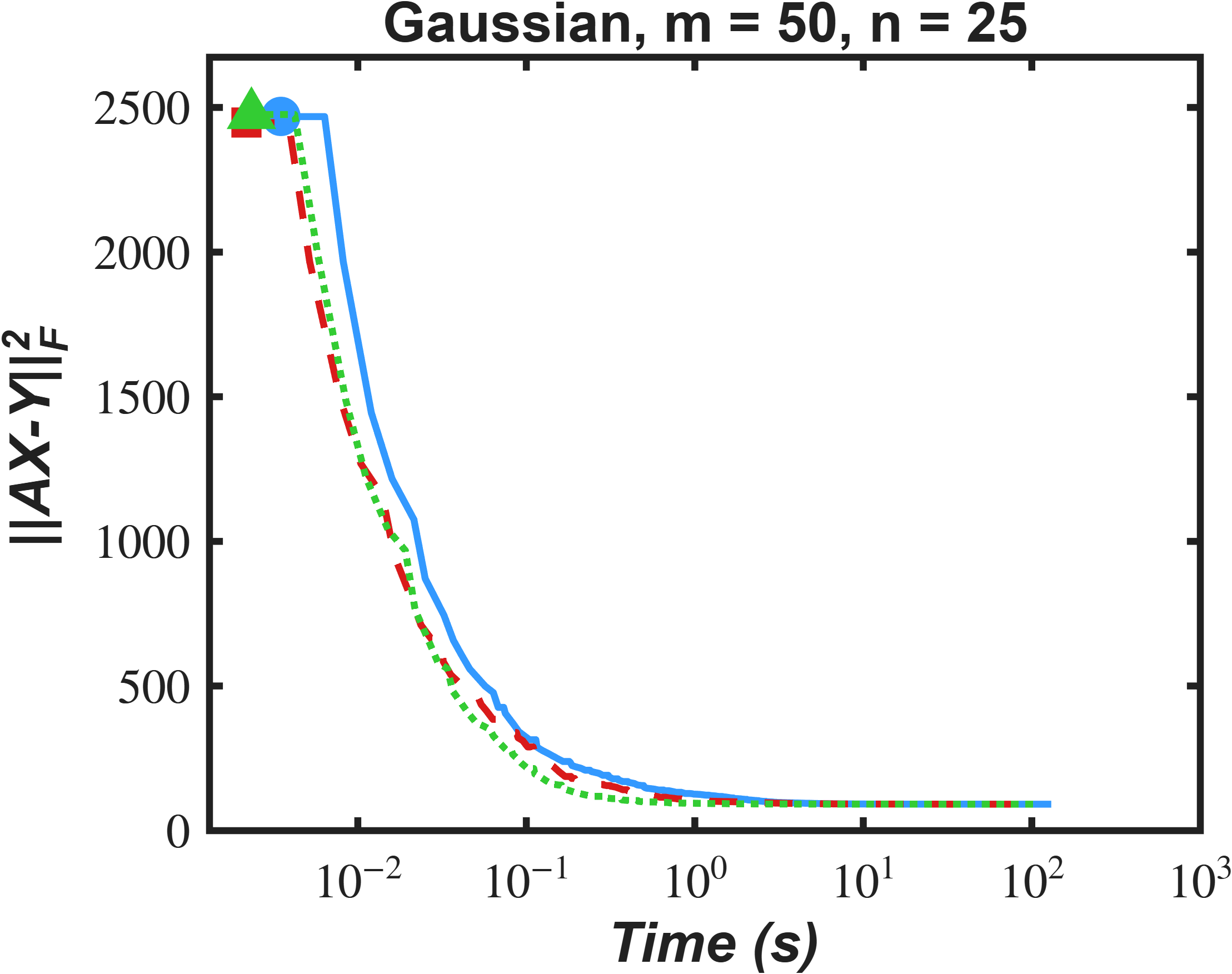}
\end{minipage}\hfill
\begin{minipage}[t]{0.45\textwidth}\vspace*{0pt}
    \includegraphics[width=\textwidth]{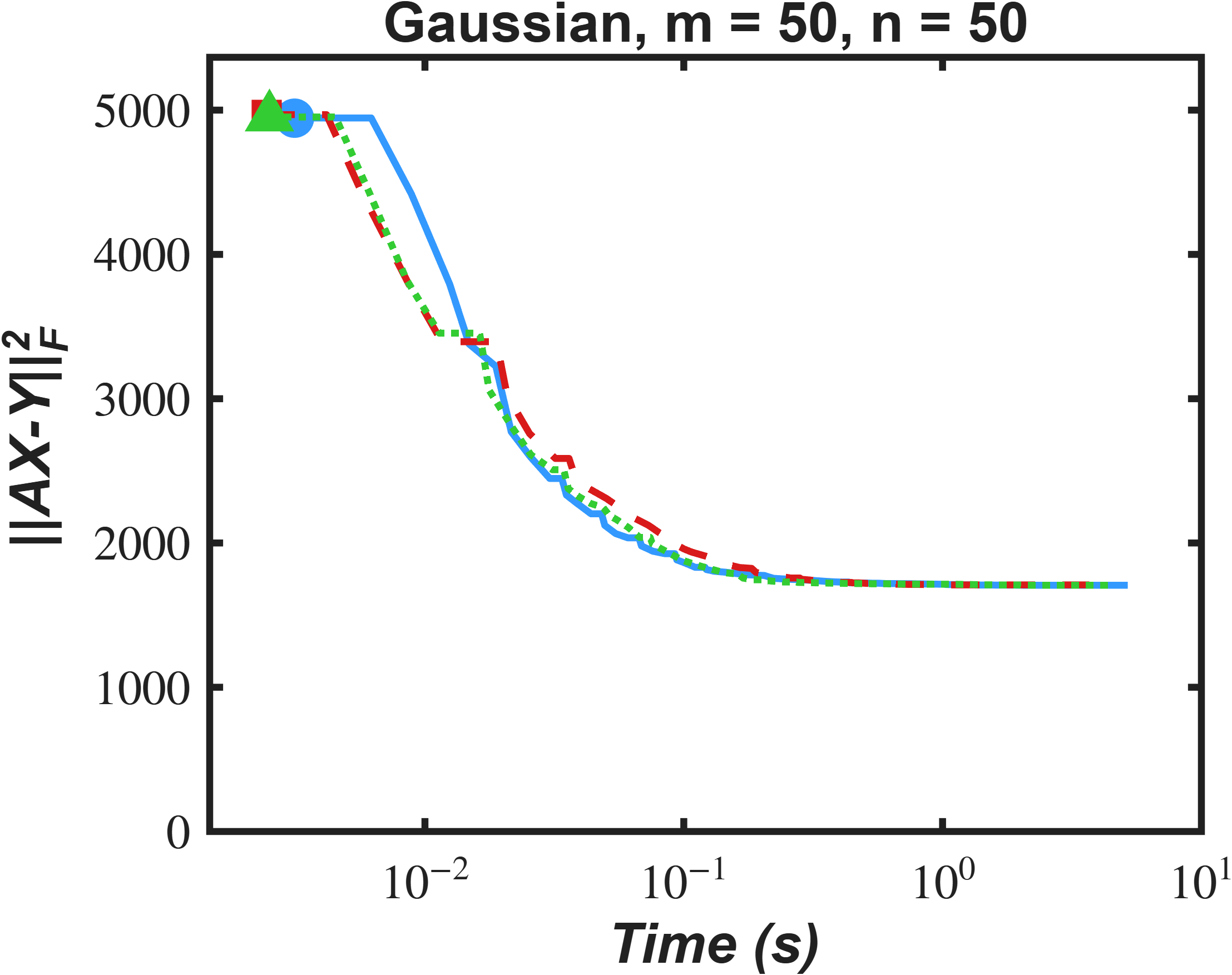}
\end{minipage}\hfill
\begin{minipage}[t]{0.10\textwidth}\vspace*{0pt}
    \phantom{\includegraphics[width=\textwidth]{Fig1c.png}}
\end{minipage}
    \caption{Residual $\|AX-Y\|_F^2$ (with $A \in \mathcal{N}$) vs.~time to solve, over three random starts with Gaussian inputs}\label{fig:convgGaus}
\end{figure}

\begin{figure}[H]
    \centering
    \begin{minipage}[t]{0.45\textwidth}\vspace*{0pt}
    \includegraphics[width=\textwidth]{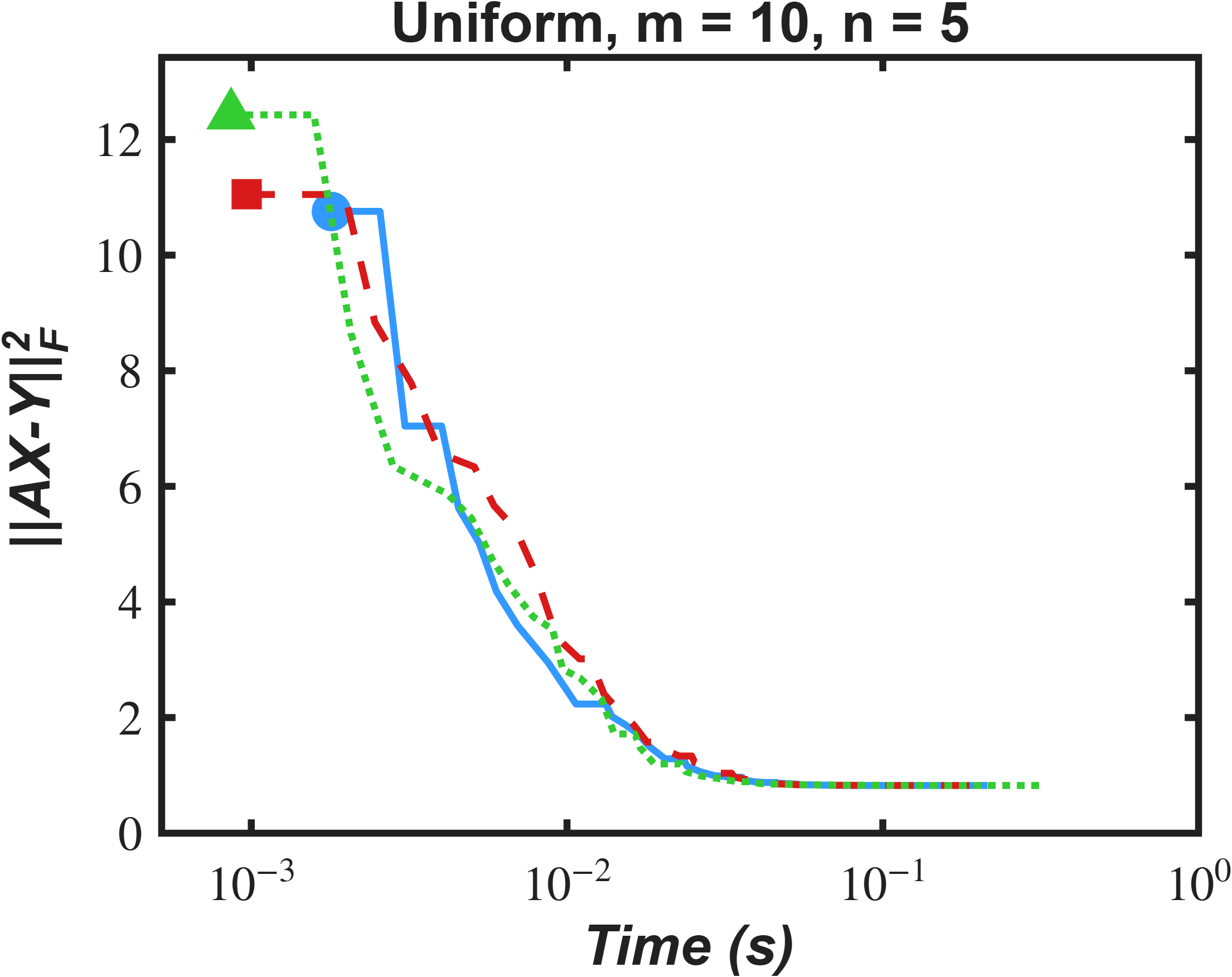}
\end{minipage}\hfill
\begin{minipage}[t]{0.45\textwidth}\vspace*{0pt}
    \includegraphics[width=\textwidth]{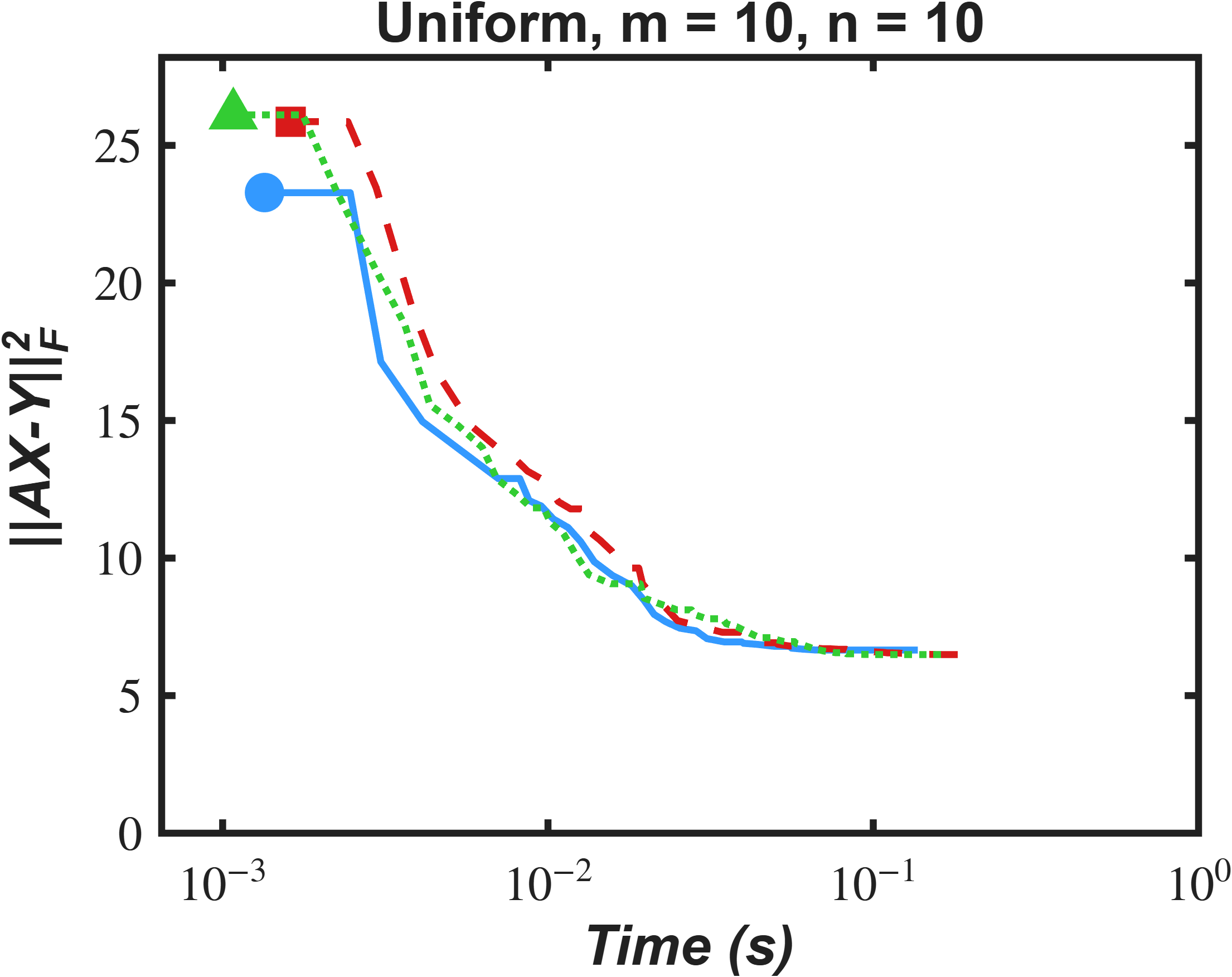}
\end{minipage}\hfill
\begin{minipage}[t]{0.10\textwidth}\vspace*{6.5pt}
    \includegraphics[width=\textwidth]{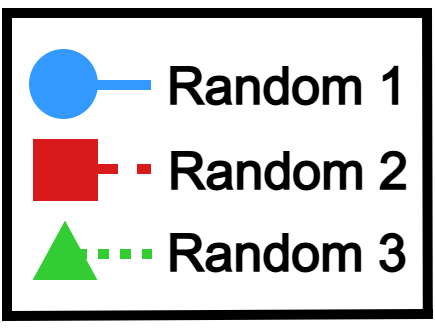}
\end{minipage}\end{figure}

\begin{figure}
    \centering
\begin{minipage}[t]{0.45\textwidth}\vspace*{0pt}
    \includegraphics[width=\textwidth]{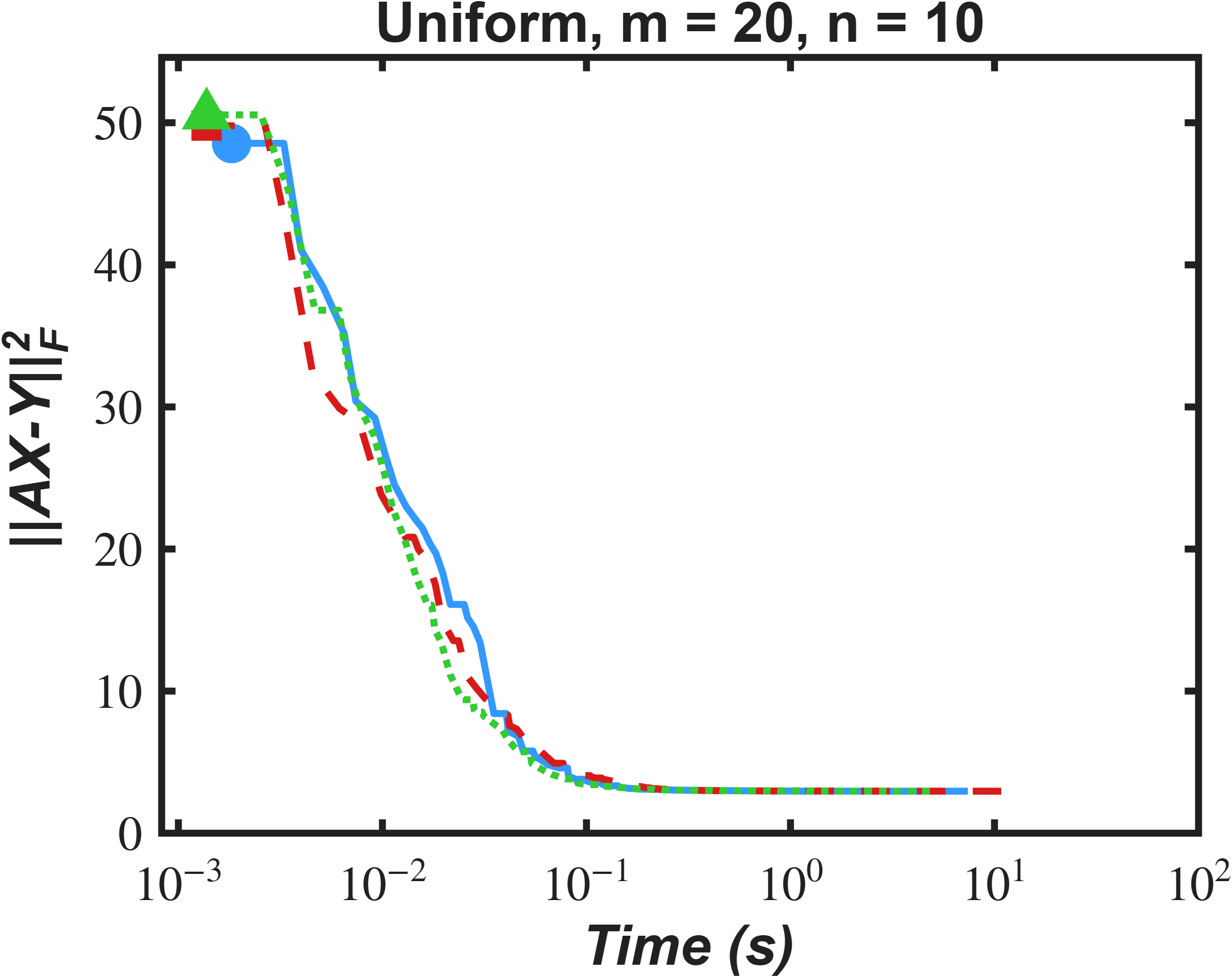}
\end{minipage}\hfill
\begin{minipage}[t]{0.45\textwidth}\vspace*{0pt}
    \includegraphics[width=\textwidth]{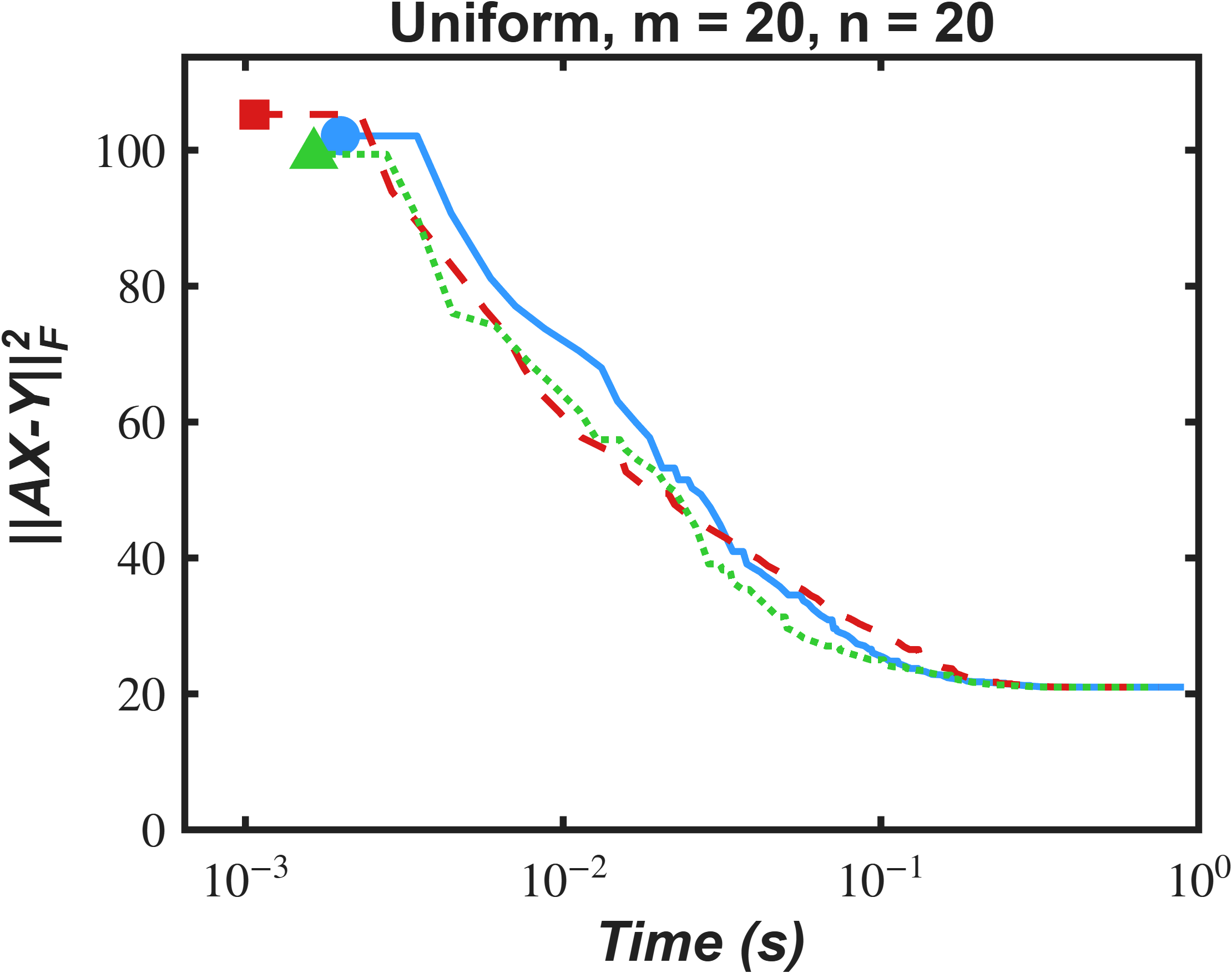}
\end{minipage}\hfill
\begin{minipage}[t]{0.10\textwidth}\vspace*{0pt}
    \phantom{\includegraphics[width=\textwidth]{Fig2c.png}}
\end{minipage}\\[5pt]
\begin{minipage}[t]{0.45\textwidth}\vspace*{0pt}
    \includegraphics[width=\textwidth]{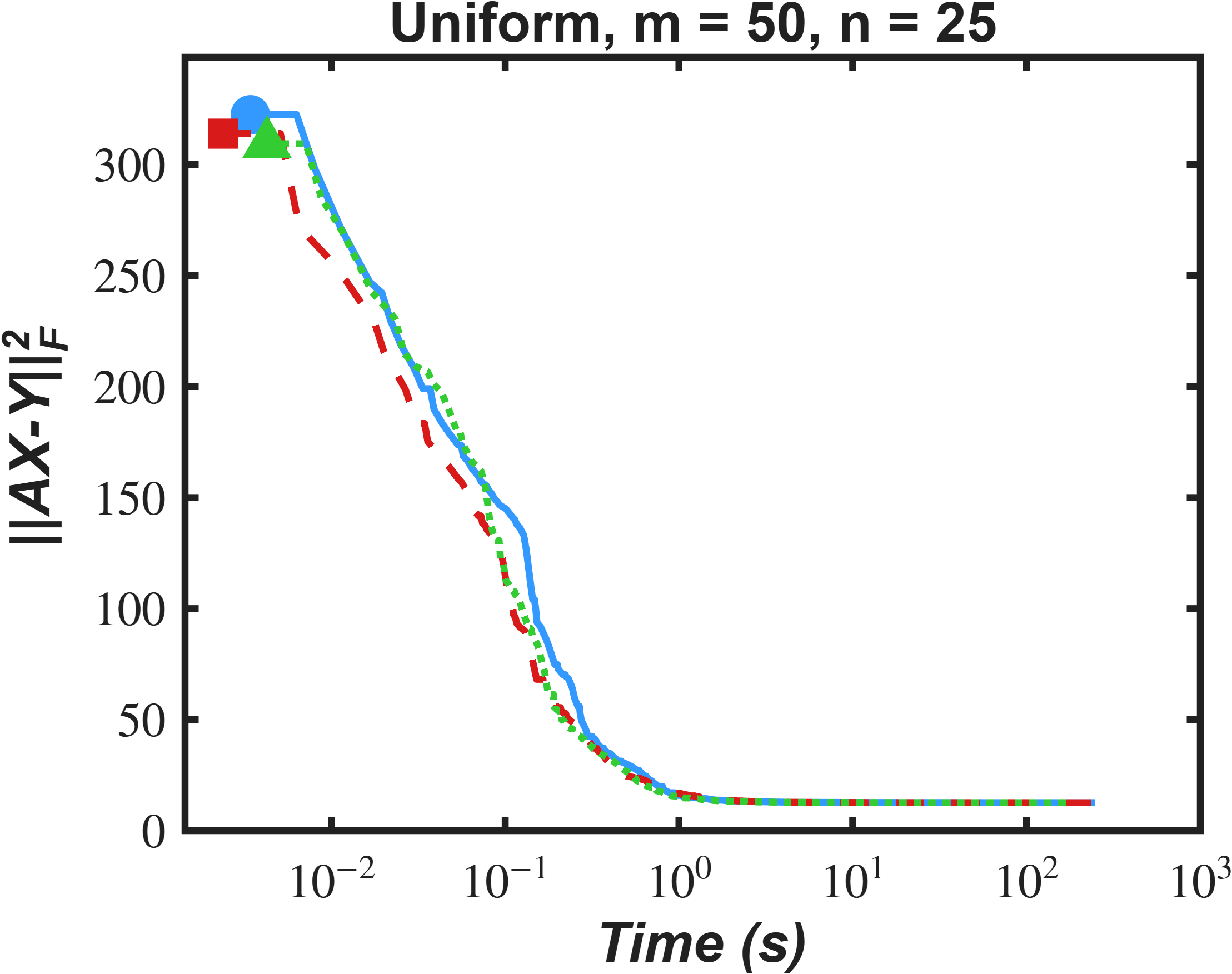}
\end{minipage}\hfill
\begin{minipage}[t]{0.45\textwidth}\vspace*{0pt}
    \includegraphics[width=\textwidth]{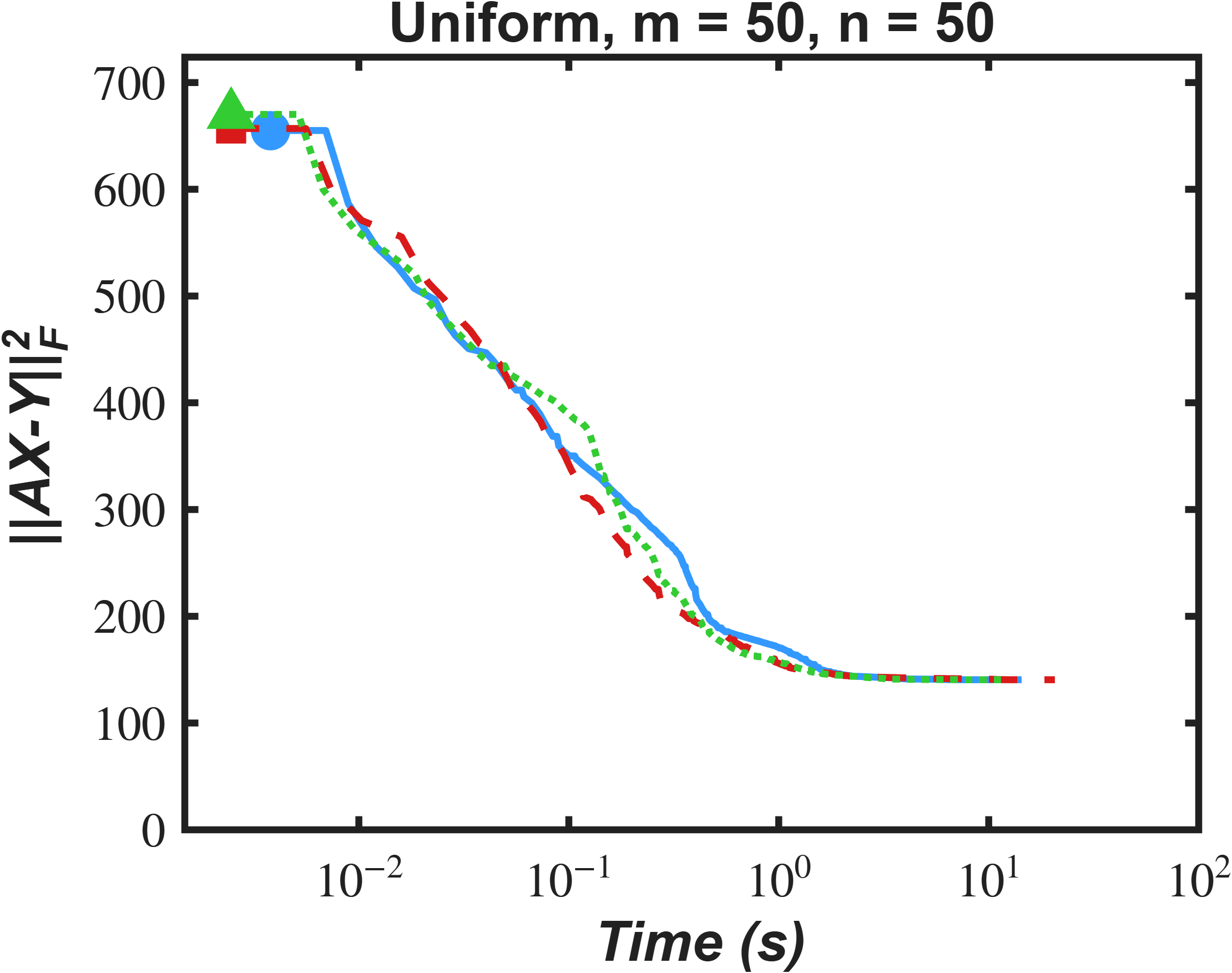}
\end{minipage}\hfill
\begin{minipage}[t]{0.10\textwidth}\vspace*{0pt}
    \phantom{\includegraphics[width=\textwidth]{Fig2c.png}}
\end{minipage}
    \caption{Residual $\|AX-Y\|_F^2$ (with $A \in \mathcal{N}$) vs.~time to solve, over three random starts with uniform inputs}\label{fig:convgUni}
\end{figure}
We also record the results of these experiments in the following tables:
\begin{table}[H]
\centering
\caption{Performance on complex \textit{Gaussian} inputs, best of three random starts}
\begin{tabular}{cc|ccc}
\hline
$m$ & $n$ & Init Res & Final Res & Time (s) \\
\hline
10 &  2 &     9.36 &    0.000 &   0.03 \\
10 &  5 &    85.23 &    5.916 &   4.71 \\
10 &  8 &   131.31 &   32.207 &   0.17 \\
10 & 10 &   165.98 &   48.219 &   0.10 \\
20 &  4 &   101.82 &    0.000 &   0.05 \\
20 & 10 &   312.83 &   14.297 &  13.50 \\
20 & 16 &   566.84 &  101.916 &   0.44 \\
20 & 20 &   715.10 &  247.993 &   0.45 \\
50 & 10 &   943.43 &    0.092 &   0.96 \\
50 & 25 &  2447.13 &   91.508 & 113.47 \\
50 & 40 &  3861.55 &  797.191 &   9.19 \\
50 & 50 &  4952.57 & 1706.676 &   4.54 \\
\hline
\end{tabular}
\label{tab:perf_gaussian}
\end{table}

\newpage

\begin{table}[H]
\centering
\caption{Performance on complex \textit{uniform} inputs, best of three random starts}
\begin{tabular}{cc|ccc}
\hline
$m$ & $n$ & Init Res & Final Res & Time (s) \\
\hline
10 &  2 &    0.90 &    0.000 &   0.03 \\
10 &  5 &   11.05 &    0.825 &   0.20 \\
10 &  8 &   21.93 &    6.154 &   0.17 \\
10 & 10 &   26.11 &    6.493 &   0.16 \\
20 &  4 &   14.78 &    0.640 &   0.16 \\
20 & 10 &   48.55 &    2.956 &   7.40 \\
20 & 16 &   80.16 &   14.320 &   1.78 \\
20 & 20 &  105.26 &   20.986 &   0.75 \\
50 & 10 &  129.79 &    2.508 &   1.21 \\
50 & 25 &  313.98 &   12.588 & 261.87 \\
50 & 40 &  504.74 &   76.858 &  25.35 \\
50 & 50 &  655.05 &  140.566 &  14.34 \\
\hline
\end{tabular}
\label{tab:perf_uniform}
\end{table}

Further note that in every trial $\|\operatorname{grad}_U f\|_F$ attains the desired threshold of less than $10^{-6}$, thus reaching approximate criticality. Future work could investigate whether there exists an initialization strategy that exploits the geometry of the objective function, while still avoiding degeneracies for rank-deficient $X$, to reduce compute time.

Another apparent trend is the relationship between matrix dimensions and solve time. The solver takes considerably more time to complete for $n$ near $0.5m$ than when $n$ lies near either end of the dimensional spectrum. When $n$ is small relative to $m$ (e.g.\ $n = 0.2m$), a reasonable explanation for the reduced solve time lies in how $\rank(X)$ constrains the problem---recall the analogous role of $\rank(X)$ in the Matrix Completion and Targeting Problems mentioned in Section~\ref{sec:intro}. Since ${\rank(X)\leq n}$, the problem is heavily underdetermined; $A$ is constrained only on $\colspace(X)$ and remains free on its orthogonal complement. Thus, this large degree of freedom significantly reduces the number of directions the solver must resolve, as the residual is unchanged along those left free, leaving an effectively lower-dimensional problem that the solver optimizes quickly. As $n$ grows larger, this freedom disappears. Why solve times when $n$ is proximal to $0.5m$ are markedly higher than when $n$ is proximal to $m$ remains unclear to the author, and we leave a finer analysis of this tendency to future work. Figures \ref{fig:convgGaus} and \ref{fig:convgUni}, however, demonstrate that the primary reduction in residual occurs within the first second on the tested examples.

The freedom present in the small $n$ case has a further, observable consequence for the recovered solutions. Consider Table \ref{tab:diagnostic}:
\newpage
\begin{table}[H]
\centering
\caption{Solution diagnostics across the three random starts. \textit{Res.} is the range of final residuals as a percentage of the largest final residual; \textit{Spread} is the maximum pairwise $\|A_i - A_j\|_F$ as a percentage of the largest $\|A_i\|_F$}
\begin{tabular}{cc|cc|cc}
\hline
& & \multicolumn{2}{c|}{Gaussian} & \multicolumn{2}{c}{Uniform} \\
$m$ & $n$ & Res.\ (\%) & Spread (\%) & Res.\ (\%) & Spread (\%) \\
\hline
10 &  2 & 0.00 & 120.21 & 0.00 & 127.24 \\
10 &  5 & 0.00 &   0.01 & 0.00 &  88.12 \\
10 &  8 & 0.43 &  70.85 & 0.00 &   0.00 \\
10 & 10 & 0.00 &   0.00 & 2.47 &  68.69 \\
20 &  4 & 0.00 & 126.25 & 0.00 & 106.99 \\
20 & 10 & 0.00 & 105.86 & 0.04 &  84.98 \\
20 & 16 & 1.02 &  69.45 & 0.76 &  76.33 \\
20 & 20 & 0.20 &  27.26 & 0.00 &   0.00 \\
50 & 10 & 0.00 & 126.35 & 0.00 & 123.90 \\
50 & 25 & 0.04 & 107.34 & 0.07 & 104.03 \\
50 & 40 & 0.17 &  86.82 & 0.11 &  40.35 \\
50 & 50 & 0.15 &  39.03 & 0.02 &  23.55 \\
\hline
\end{tabular}
\label{tab:diagnostic}
\end{table}

 Although the residual percent difference across starts is zero for all $n = 0.2m$ cases in Table~\ref{tab:diagnostic}, the recovered matrix $A$ is not the same across all starts, as reflected in the nonzero spread values. This non-uniqueness property again results from $A$ being unconstrained on $\colspace(X)^\perp$, since infinitely many distinct normal matrices may achieve the same optimal residual, and different starts may recover different matrices.

A separate phenomenon appears in a handful of cases, where the starts recover different final residuals. For example, in the uniform, $m = n = 10$ instance, the recovered residuals differ by $2.47$\% across starts in Table~\ref{tab:diagnostic}. Such cases are direct evidence of the failure of geodesic concavity of $f$ demonstrated in Section~\ref{sec:complexRed}, since the solver recovers distinct, suboptimal stationary points depending on its starting point. When the author allowed the solver to run to a smaller gradient-norm threshold of $10^{-15}$, this relative spread persisted rather than collapsing to a single value, further confirming that the starts recover genuinely distinct stationary points.

Finally, the primary difference we observe between the Gaussian and uniform setups is the average size of the residuals, but this is attributable to the differing scale of the input entries---uniform entries are smaller in magnitude than Gaussian entries on average, yielding correspondingly smaller residuals. However, we observe no meaningful effect of distribution type on solver
performance.

\subsection{Known Global Optimum Recovery}\label{ssec:complexKnown}~\\[8pt]
\indent In this subsection we revisit the Normal Targeting Problem by asking the following question: If a normal matrix $A^\star \in \mathcal{N}$ is given ahead of time such that $A^\star X - Y = 0$, can our solver recover $A^\star$ or some equally sufficient normal matrix? 

If we could guarantee that the solver always returns this result whenever such an $A^\star$ exists, then our model could serve as a heuristic approach to the Normal Targeting Problem, provided one is willing to accept the solver's computational cost. The nonconcavity of our problem (demonstrated in Section~\ref{sec:complexRed} and Subsection~\ref{ssec:complexPerform}), however, makes such a guarantee likely impossible. Thus, we instead numerically test how \emph{frequently} a residual of zero is recovered and how \emph{close} the final residual comes to it. This experiment offers us a glimpse of our model's global optimality accuracy, which the previous subsection could not provide, since we could not be certain what the global optimum was in our prior experiments.

We set up this experiment by first generating random $A^\star$ and $X$. We then set $Y:=A^\star X$ and plug both $X$ and $Y$ into our solver and observe the recovered residual. We only consider random $X$ with uniform entries, since the previous subsection yielded negligible difference between the two distributions, and we run two random starts. We generate $A^\star = U^\star D^\star (U^\star)^*$ by drawing $U^\star$ from the same Haar measure (the uniform measure on $\mathrm{U}(m)$) as the random starts and taking the diagonal entries of $D^\star$ (independent of $U^\star$) to have real and imaginary parts independently uniform on $[-1,1]$. Throughout, we write $A_{\mathrm{rec}}$ for the final normal matrix recovered by our algorithm. We obtain the following results (residual and gradient norm plots are available via our \href{https://github.com/kbierly/NPP}{GitHub repository}):

\begin{table}[H]
\centering
\caption{Global optimum recovery with uniform $X$ on two random starts. \textit{Rel.\ $A^\star$} is the relative difference $\|A_{\mathrm{rec}} - A^\star\|_F$ as a percentage of $\|A^\star\|_F$}
\resizebox{\textwidth}{!}{%
\begin{tabular}{cc|cc|cc|cc|cc}
\hline
& & \multicolumn{2}{c|}{Init Res} & \multicolumn{2}{c|}{Final Res} & \multicolumn{2}{c|}{Time (s)} & \multicolumn{2}{c}{Rel.\ $A^\star$ (\%)} \\
$m$ & $n$ & Start 1 & Start 2 & Start 1 & Start 2 & Start 1 & Start 2 & Start 1 & Start 2 \\
\hline
10 &  2 &    1.8 &    1.4 & 1.2e-22 & 2.7e-16 &  0.66 & 0.04 &  148.0 &  117.2 \\
10 &  5 &    4.6 &    7.2 & 3.4e-12 & 4.0e-12 &  0.44 & 0.52 &   67.6 &   62.9 \\
10 &  8 &   11.6 &    8.8 & 1.6e-12 & 3.6e-12 &  0.68 & 0.50 &    0.0 &    0.0 \\
10 & 10 &   18.7 &   15.4 & 3.7e-14 & 1.4e-14 &  0.29 & 0.29 &    0.0 &    0.0 \\
20 &  4 &   18.7 &   16.3 & 2.8e-14 & 3.6e-16 &  0.07 & 0.09 &  142.8 &  137.8 \\
20 & 10 &   36.2 &   34.9 & 5.7e-11 & 4.5e-11 &  7.42 & 7.42 &  111.7 &   94.7 \\
20 & 16 &   85.9 &   77.5 & 4.8e-17 & 6.2e-15 &  0.92 & 0.73 &    0.0 &    0.0 \\
20 & 20 &  100.1 &   93.4 & 6.3e-18 & 3.3e-14 &  1.31 & 0.93 &    0.0 &    0.0 \\
50 & 10 &  111.9 &  110.4 & 5.3e-21 & 1.9e-23 &  0.76 & 1.50 &  130.2 &  314.9 \\
50 & 25 &  215.7 &  223.5 & 9.0e-13 & 8.8e-13 & 71.99 & 97.81 &   93.9 &   94.6 \\
50 & 40 &  404.1 &  432.3 & 5.7e-17 & 2.0e-15 & 11.80 & 13.39 &    0.0 &    0.0 \\
50 & 50 &  500.8 &  518.4 & 2.0e-17 & 3.0e-22 &  7.91 & 12.14 &    0.0 &    0.0 \\
\hline
\end{tabular}%
}
\label{tab:known_min}
\end{table}

From the above, we observe that a sufficient $A_{\mathrm{rec}}$ was recovered such that the final residual was less than $10^{-10}$ in every trial. Both random starts recover $A^\star$ (visible in the Rel.\ $A^\star$ column) in each $n\geq 0.8m$ trial, whereas when $n \leq 0.5m$, equally sufficient but distinct normal matrices were recovered, reflecting the freedom on $\colspace(X)^{\perp}$.

Although these results are specific to their constructed instances and do not constitute a guarantee, the consistency observed across trials in our previous experiments in Subsection \ref{ssec:complexPerform}, together with the final residual precision attained here, supports the accuracy of our model.

\subsection{Comparison with Established Algorithms over the CNP}\label{ssec:complexComp}~\\[8pt]
\indent As we touched on in Section~\ref{sec:intro}, our model has no direct competitors for the NPP, but when we let $X = I_m$ and $Y \in \M_m(\C)$, the NPP reduces to exactly the CNP, allowing us to numerically compare our method with established techniques. Setting $X=I_m$ simplifies several steps of the process we follow in Sections~\ref{sec:complexRed} and~\ref{sec:complexGradient}. Since the rows of $U^*X = U^*$ have unit norm, we have $\phi_i = 1$ for all $i=1,\ldots, m$, and $D^\star_U$'s optimal diagonal entries become
\begin{align*}
   d_i^\star=\frac{\gamma_i}{\phi_i}=(U^*Y)_i(U^*)_i^*=e_i^* (U^*Y U)e_i=(U^*YU)_{i,i},
\end{align*}
and the objective~\eqref{eq:fU} reduces (under unitary invariance of the Frobenius norm) to
\begin{equation}\label{eq:objCNP}
   f(U) = \|D^\star_U\|_F^2 = \|\operatorname{diag}(U^*YU)\|_F^2.
\end{equation}
Substituting $X = I_m$ into the Euclidean gradient~\eqref{eq:euclidean_grad} yields
\begin{equation}\label{eq:gradCNP}
   \nabla_U f = 2\bigl(Y^*UD^\star_U + YU(D^\star_U)^* - U|D^\star_U|^2\bigr),
\end{equation}
and into the Euclidean Hessian~\eqref{eq:hess} yields
\begin{align*}
   H_U f[V] = 2\bigl(
   &Y^*VD^\star_U + Y^*U\,\delta(D^\star_U) \\
   &+ YV(D^\star_U)^* + YU\,\delta(D^\star_U)^* \\
   &- V|D^\star_U|^2 - U\,\delta(|D^\star_U|^2)\bigr),
\end{align*}
where, since $\delta(\phi_i) = 0$,
\begin{align*}
   \delta(D^\star_U) = \operatorname{diag}(V^*YU + U^*YV).
\end{align*}
The Riemannian gradient and Hessian are then obtained through the same methods in Section~\ref{sec:complexGradient}.

Gabriel \cite{gabriel} and Ruhe \cite{ruhe} somewhat independently developed the first tractable approaches for approximating the CNP in 1987. However, we only compare our model against Ruhe's work, since he provided a complete algorithm that works for the general CNP, whereas Gabriel characterized the structure of its solutions for specific classes of matrices. 

Ruhe's algorithm relies on what he and Gabriel refer to as a $\Delta H$-matrix, see \cite[Definition 1]{ruhe}. Ruhe proves that every square matrix $Y$ has a closest normal matrix $A$, and in the coordinates of the eigenvectors of $A$, $Y$ is a $\Delta H$-matrix. He furthermore proves that there is a convergent algorithm \cite[Algorithm J]{ruhe} that transforms an arbitrary square $Y$ into a $\Delta H$-matrix by means of a sequence of unitary similarities. As Noferini and  Zhukov write, this technique can be viewed as a variant of ``coordinate descent on the manifold of unitary matrices'' \cite[Section 1]{noferini2}. Guglielmi and Scalone state in \cite[Section 2.1]{scalone} that Ruhe's algorithm ``is proved to determine the global optimum'' to the CNP, but Ruhe only proves that the algorithm converges to a stationary point, and he provides only local optimality conditions (\cite[Theorems 3 \& 4]{ruhe}), not global ones. In all instances that Ruhe numerically tested, his algorithm converged to a normal matrix satisfying these conditions (i.e.\ not saddle points), but he does not explicitly demonstrate that the recovered normal matrices are globally optimal. 

Guglielmi and Scalone further state that ``since this problem is solved globally by Ruhe's algorithm, our interest is mainly directed to the computation of the closest real normal matrix and to the normal completion problem'' \cite[Section 2.1]{scalone}. Thus, rather than developing an algorithm to compete with Ruhe's on the CNP, they concentrated on developing the first algorithms that address these alternate problems. They demonstrate, however, that their algorithm \cite[Algorithm 1]{scalone} is not specific to only these problems and is generalizable over the complex numbers. They further report that in ``all examples [when implementing the complex version of their algorithm], we have found the same matrix obtained by applying Ruhe's algorithm'' \cite[Section 2.1]{scalone}. They do not benchmark how their algorithm performs against Ruhe's on the general CNP in terms of computational efficiency, however. We evaluate and compare such performance in this subsection, including against our solver.

Guglielmi and Scalone's approach differs significantly from Ruhe's. They rely on the fact that a matrix is normal if and only if the square of its Frobenius norm is equal to the sum of the squares of the moduli of its eigenvalues. To objectively quantify this property, they employ Henrici's (squared) \textit{departure from normality} \cite{henrici}:
\begin{equation}\label{eq:Henrici}
    \nu(Y) = \|Y\|^2_F - \|\Lambda\|^2_F = \|Y\|^2_F -\sum_{i=1}^m|\lambda_i(Y)|^2,
\end{equation}
such that $Y$ is normal if and only if \eqref{eq:Henrici} is equal to $0$. Introducing $\varepsilon > 0$ and a perturbation direction $E \in \M_m(\C)$ with $\|E\|_F = 1$, they minimize \[\|(Y+\varepsilon E)-Y\|_F^2 = \varepsilon^2\] subject to the constraint
\[\nu(Y+\varepsilon E)=0,\] such that the normal $A^\star = Y+\varepsilon^\star E^\star$ is optimal. Note that this result is what the authors aspire to obtain, and determining a global optimum is not necessarily feasible or guaranteed in practice, as noted in \cite[Section 4.1]{scalone}.

In pursuit of this result, Guglielmi and Scalone implement a two-level gradient system. The ``inner level'' first optimizes $E$ by fixing $\varepsilon$, computing the gradient of $\nu(Y+\varepsilon E)$, updating $E$ by steepest descent, normalizing to preserve $\|E\|_F=1$, and iterating until $E$ settles at a negative multiple of the gradient, which is a local optimum (denoted $E(\varepsilon)$); see \cite[Section 3.2]{scalone} for additional details on stationary points. Once this process is complete, their algorithm commences the ``outer level,'' where the direction is fixed at $E(\varepsilon)$ and the derivative of $\nu(Y+\varepsilon E(\varepsilon))$ with respect to $\varepsilon$ is computed. If $\dot{\varepsilon}$ is such that $\nu(Y+\dot{\varepsilon} E(\varepsilon))=0$, then this derivative is equal to zero as well by \cite[Theorem 5]{scalone}. Notably, for the CNP, if $\varepsilon>\dot{\varepsilon}$, then $\nu(Y+\varepsilon E(\varepsilon))=0$  by \cite[Theorem 6]{scalone}, and if $\varepsilon\leq \dot{\varepsilon}$, then the function is locally convex. Consequently, Newton's method can be applied from the left to determine a locally minimal $\dot{\varepsilon}$, where $\nu(Y+\dot{\varepsilon} E(\varepsilon))=0$. Once the outer level terminates, their algorithm loops until improvements on $\varepsilon$ are negligible within a prescribed tolerance. Guglielmi and Scalone further incorporate a bisection correction as a safeguard against overshooting $\dot{\varepsilon}$ before the iterates reach the locally convex regime.

Observe that the description above reduces their algorithm to only the components relevant to the CNP; an advantage that Guglielmi and Scalone's algorithm enjoys over ours is that one can implement it to address the Normal Matrix Completion Problem by fixing entries of $Y$ via a given pattern, which involves a slightly more complicated algorithm. Whether our Riemannian manifold optimization techniques could be modified to similarly work over this problem, we leave to future work.

Neither Ruhe's nor Guglielmi and Scalone's MATLAB code is publicly available. However, the author of this paper corresponded with the latter authors, who graciously offered their implementation for numerical testing; at the authors' request, their implementation is not included in our public repository. Readers wishing to reproduce our comparisons against their algorithm may contact them directly. Ruhe's original code could not be obtained due to his passing in 2015. We therefore faithfully reproduce Ruhe's algorithm in MATLAB from his pseudocode in \cite[Algorithm J]{ruhe}, incorporating the determinant sign correction noted in \cite[Section 5]{higham}. Any comparison with Ruhe's algorithm therefore comes with the caveat that it reflects our own implementation, which may not match the performance of a tuned original. However, we test all three methods---Ruhe's, Guglielmi and Scalone's, and our Riemannian trust-region solver on a random start---against several examples from the prior two papers to ensure that we accurately recover and validate previous results. Specifically, we examine:
\begin{itemize}
    \item From Ruhe \cite{ruhe}:
    \begin{itemize}
        \item[$\circ$] A specific $2\times2$ complex matrix (Ruhe's unnamed first example in \cite[Section 4]{ruhe}), whose putative closest normal matrix is also recovered in \cite[Example 8]{scalone}.
        \item[$\circ$] $B_7$, a $7\times7$ complex matrix with random, Gaussian entries.
        \item[$\circ$] $J_7$, a $7\times7$ Jordan block with zero eigenvalue.
        \item[$\circ$] $T_7$, a $7\times7$ nilpotent matrix (random strictly upper triangular).
        \item[$\circ$] $F_{12}$, the $12\times12$ Frank matrix, whose entries are
        \[(F_{12})_{jk}=\begin{cases}
            \min(j,k)&\text{if }k \geq j-1\\
            0 & \text{otherwise},
        \end{cases}\]
        and, as Ruhe notes, is a ``favorite stumbling block for eigenvalue algorithm developers.''
    \end{itemize}
    \item From Guglielmi and Scalone \cite{scalone}:
    \begin{itemize}
        \item[$\circ$] Example 1, a $3\times 3$ real matrix.
        \item[$\circ$] Example 9, a $7\times7$ complex matrix arising in \cite[Example 2]{barl}, for which they report a $2$-norm comparison.
    \end{itemize}
\end{itemize}

To better understand how these algorithms scale performance-wise as the amount of data increases, we also test the three methods across random, uniformly distributed $Y\in \M_m(\C)$ with complex entries and 
${m\in \{10,20,50,100,200,250\}}$. We record the results here (more extensive data can again be found at our \href{https://github.com/kbierly/NPP}{GitHub repository}):
\newpage

\begin{table}[H]
\centering
\caption{Performance on the test matrices of \cite{ruhe} and \cite{scalone}. Line styles as in Figure~\ref{fig:complexLit}. \textit{Iterations} reports a different metric for each method: the number of iterations for NPP, total outer\,/\,inner iterations for G\&S, and sweeps for Ruhe}
\label{tab:literature}
\begin{tabular}{c|c|c|c|c|c}
\hline
\rule[-1.2ex]{0pt}{4.2ex}Matrix & Method & Final Residual & $\|AA^*-A^*A\|_F$ & Iterations & Time (s) \\
\hline
\multirow{3}{*}{$2\times2$}
  & \lsnpp\ NPP  & 1.933 & 4.7e-16 & 9 & 0.024 \\
  & \lsgs\ G\&S  & 1.933 & 7.7e-07 & 47\,/\,318 & 0.012 \\
  & \lsruhe\ Ruhe & 1.933 & 2.8e-15 & 2 & 0.003 \\
\hline
\multirow{3}{*}{$B_7$}
  & \lsnpp\ NPP  & 29.752 & 3.8e-14 & 15 & 0.044 \\
  & \lsgs\ G\&S  & 29.752 & 1.1e-06 & 15\,/\,17562 & 0.680 \\
  & \lsruhe\ Ruhe & 29.752 & 1.9e-13 & 123 & 0.005 \\
\hline
\multirow{3}{*}{$J_7$}
  & \lsnpp\ NPP  & 0.857 & 1.1e-15 & 13 & 0.027 \\
  & \lsgs\ G\&S  & 0.857 & 2.7e-07 & 21\,/\,1226 & 0.072 \\
  & \lsruhe\ Ruhe & 0.857 & 2.8e-15 & 17 & 0.011 \\
\hline
\multirow{3}{*}{$T_7$}
  & \lsnpp\ NPP  & 12.943 & 6.8e-15 & 18 & 0.047 \\
  & \lsgs\ G\&S  & 12.943 & 7.0e-07 & 14\,/\,35124 & 1.381 \\
  & \lsruhe\ Ruhe & 12.943 & 5.4e-14 & 278 & 0.012 \\
\hline
\multirow{3}{*}{$F_{12}$}
  & \lsnpp\ NPP  & 509.922 & 4.8e-13 & 36 & 0.217 \\
  & \lsgs\ G\&S  & 509.922 & 4.8e-06 & 21\,/\,10745 & 0.963 \\
  & \lsruhe\ Ruhe & 510.155 & 5.5e-12 & 358 & 0.033 \\
\hline
\multirow{3}{*}{Ex.\ 1}
  & \lsnpp\ NPP  & 0.709 & 5.6e-16 & 11 & 0.024 \\
  & \lsgs\ G\&S  & 0.709 & 1.6e-07 & 11\,/\,874 & 0.030 \\
  & \lsruhe\ Ruhe & 0.709 & 1.0e-15 & 14 & 0.012 \\
\hline
\multirow{3}{*}{Ex.\ 9}
  & \lsnpp\ NPP  & 495.631 & 2.2e-11 & 19 & 0.047 \\
  & \lsgs\ G\&S  & 495.631 & 1.7e-03 & 20\,/\,1055 & 0.045 \\
  & \lsruhe\ Ruhe & 495.631 & 4.6e-11 & 10 & 0.002 \\
\hline
\end{tabular}
\end{table}

The final recovered residuals for the matrices in Table~\ref{tab:literature} differ only on $F_{12}$, where our method and Guglielmi and Scalone's both attain a marginally lower residual than Ruhe's. However, we observe that $\|AA^* - A^*A\|_F$, our measure of normality of the recovered matrix $A_{\mathrm{rec}}$, exceeds $10^{-7}$ in every Guglielmi--Scalone trial, whereas both our method and Ruhe's remain below $10^{-10}$ throughout. This gap is intuitively explainable through each method's structure, since our method and Ruhe's inherently work over $\mathcal{N}$ while Guglielmi and Scalone's reaches normality only in the limit.

We furthermore observe that in terms of computation time, Ruhe's algorithm terminated quickest (rather significantly) on each example in Table~\ref{tab:literature}. We validate
our implementation of Ruhe's algorithm on the two deterministic test matrices,
$J_7$ and $F_{12}$, obtaining exact agreement with \cite[Table~1]{ruhe} on Ruhe's measure of
normalized residual distance, $d_{\mathcal{N}}(Y)/\|Y\|_F$. We similarly recover sweep counts in strong agreement with Ruhe's on these two matrices; both comparisons are recorded in the MATLAB \texttt{diary} available in our repository.

Note that for the $2\times 2$ example, or \cite[Example 8]{scalone}, all three methods recovered the same putative closest normal matrix reported in \cite{ruhe} and \cite{scalone}. Specifically, for
\[Y=\begin{bmatrix}
    0.7616 + 1.2296i & -1.4740-0.4577i\\
    -1.6290 -2.6378i & 0.1885 - 0.8575i 
\end{bmatrix},
\]
we obtain
\[
A_{\mathrm{rec}}=\begin{bmatrix}1.1449 + 0.8324i &-2.0841 - 0.9957i \\
-1.0695 - 2.0473i &-0.1948 - 0.4603i
\end{bmatrix}.
\]
\begin{figure}[H]
    \centering
    \begin{minipage}[t]{0.45\textwidth}\vspace*{0pt}
    \includegraphics[width=\textwidth]{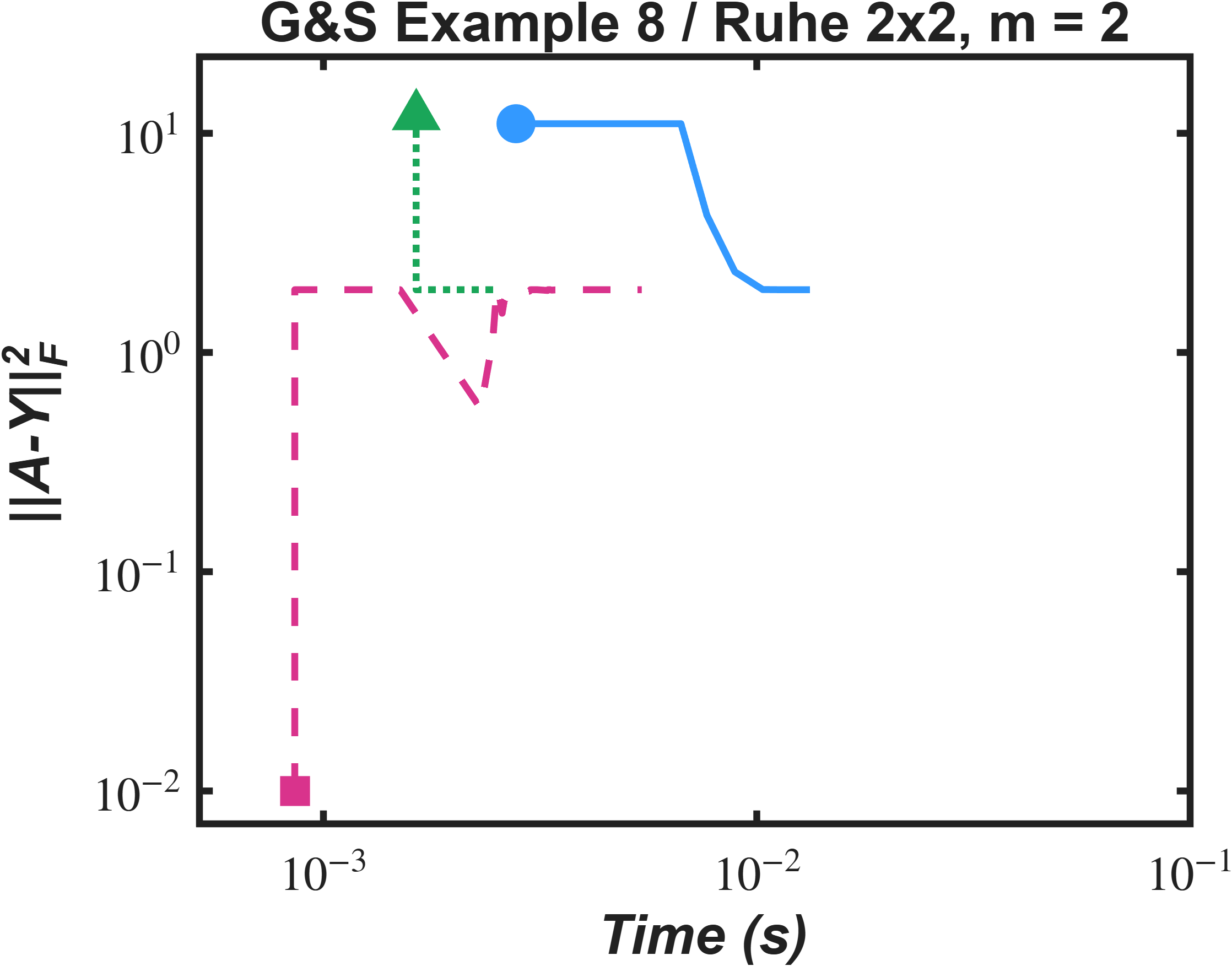}
\end{minipage}\hfill
\begin{minipage}[t]{0.45\textwidth}\vspace*{0pt}
    \includegraphics[width=\textwidth]{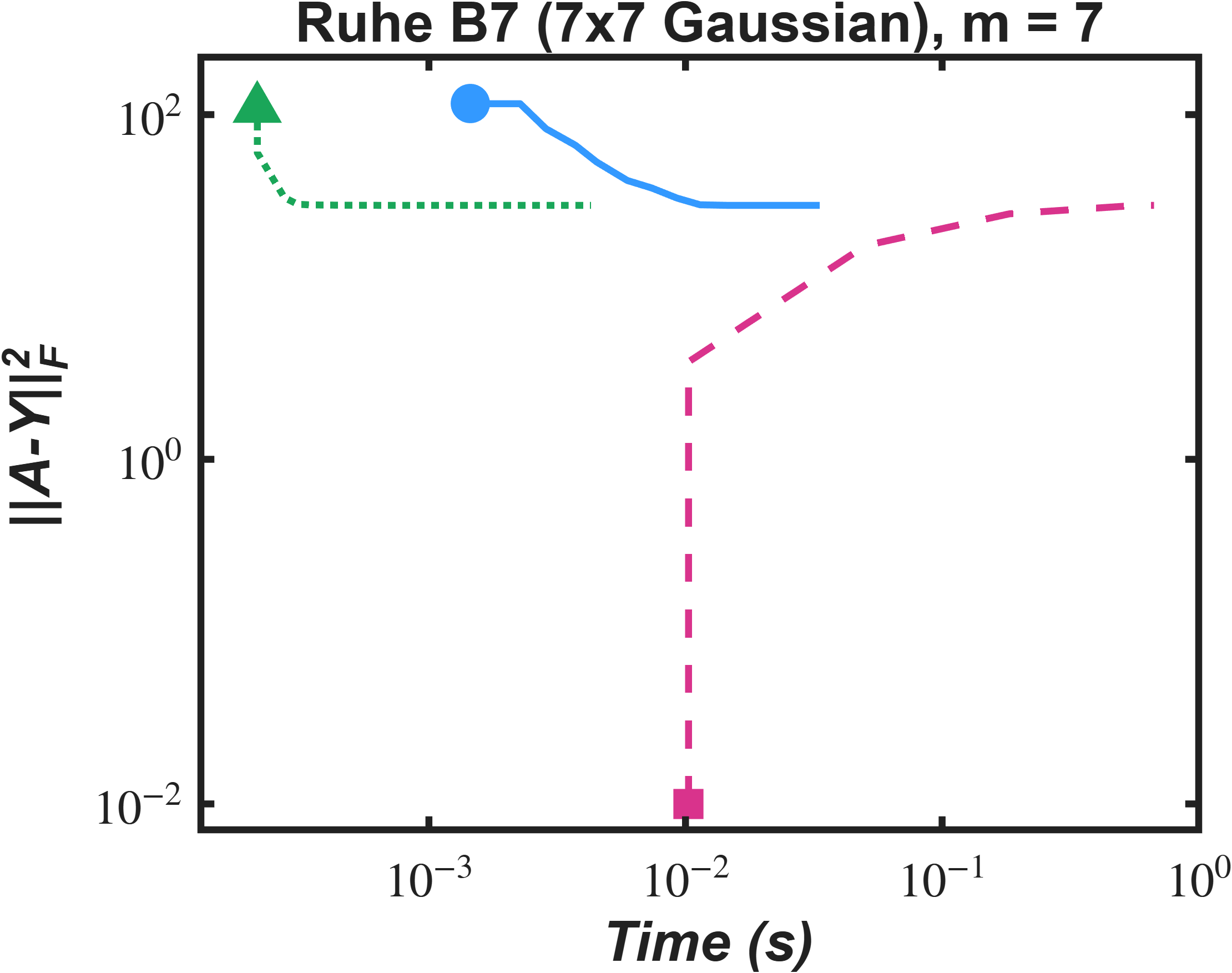}
\end{minipage}\hfill
\begin{minipage}[t]{0.10\textwidth}\vspace*{6.5pt}
    \includegraphics[width=\textwidth]{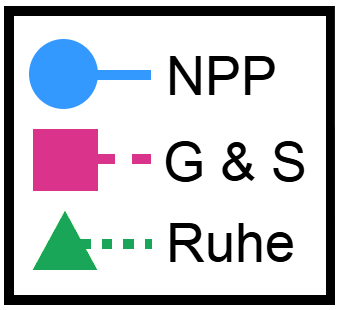}
\end{minipage}
    \\[5pt]
\begin{minipage}[t]{0.45\textwidth}\vspace*{0pt}
    \includegraphics[width=\textwidth]{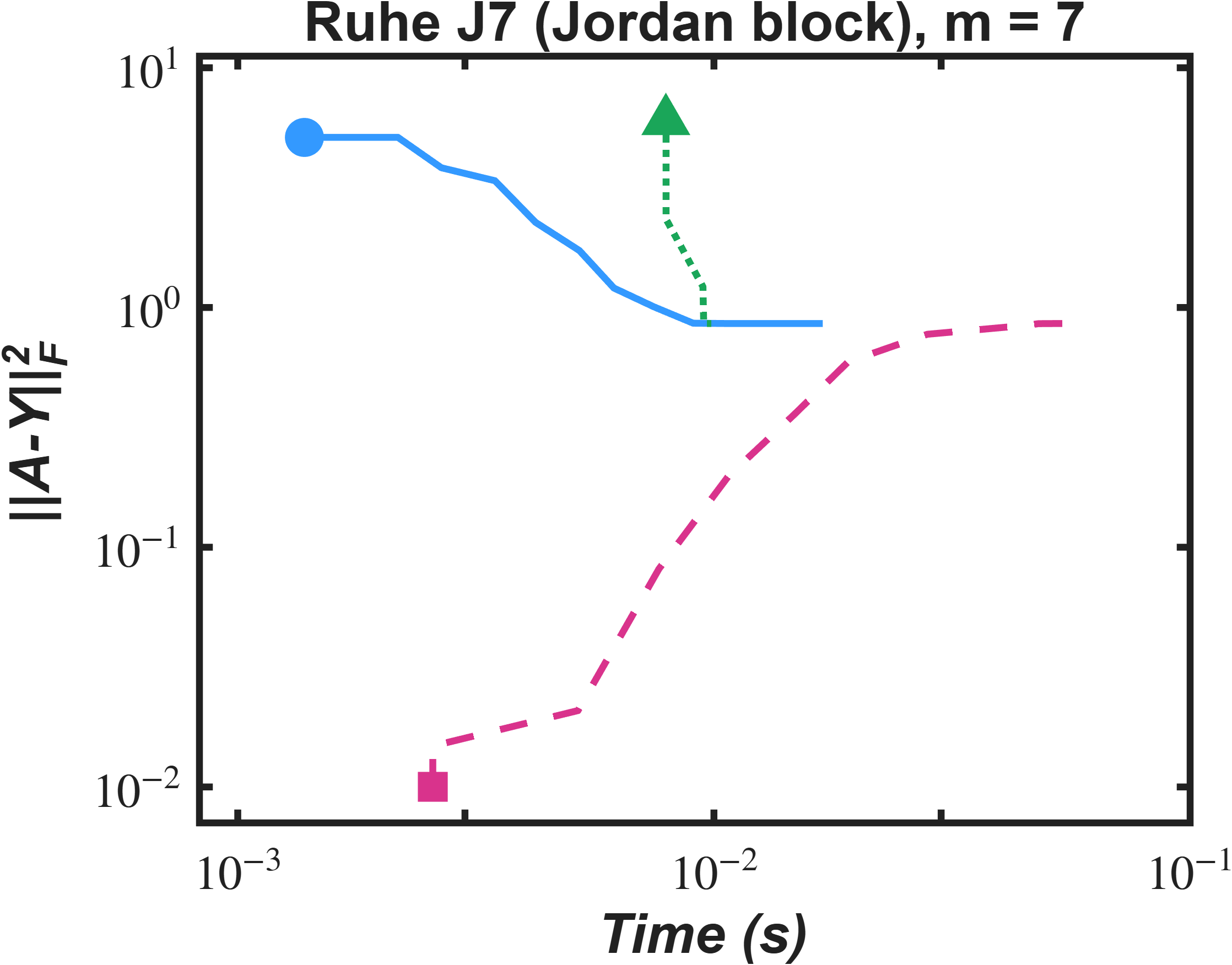}
\end{minipage}\hfill
\begin{minipage}[t]{0.45\textwidth}\vspace*{0pt}
    \includegraphics[width=\textwidth]{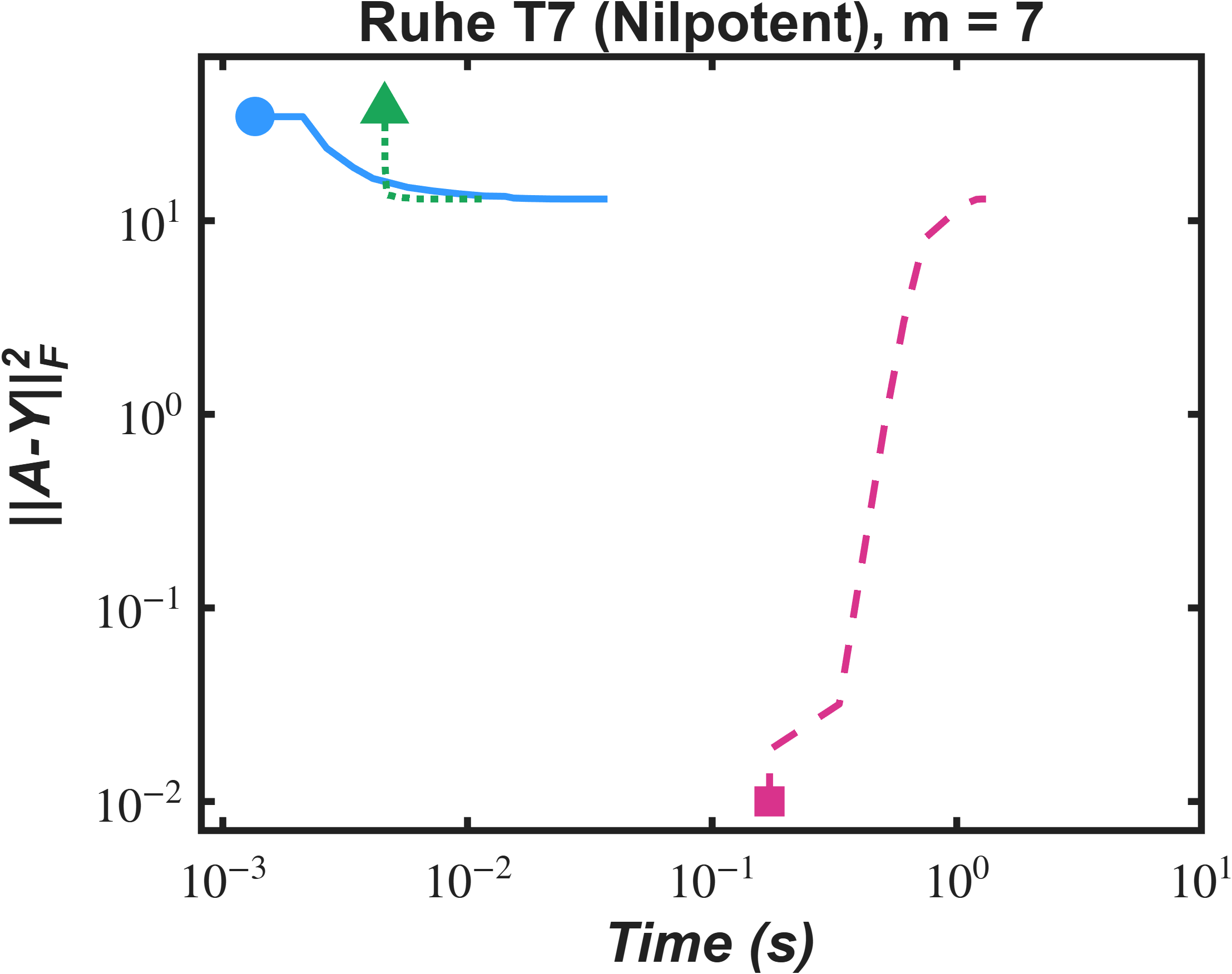}
\end{minipage}\hfill
\begin{minipage}[t]{0.10\textwidth}\vspace*{0pt}
    \phantom{\includegraphics[width=\textwidth]{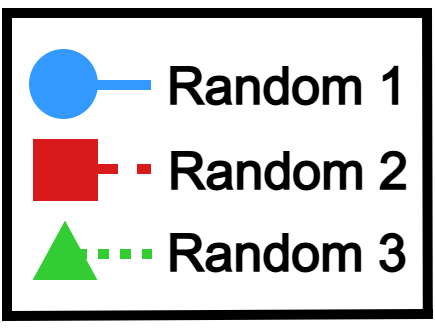}}
\end{minipage}\\[5pt]
\begin{minipage}[t]{0.45\textwidth}\vspace*{0pt}
    \includegraphics[width=\textwidth]{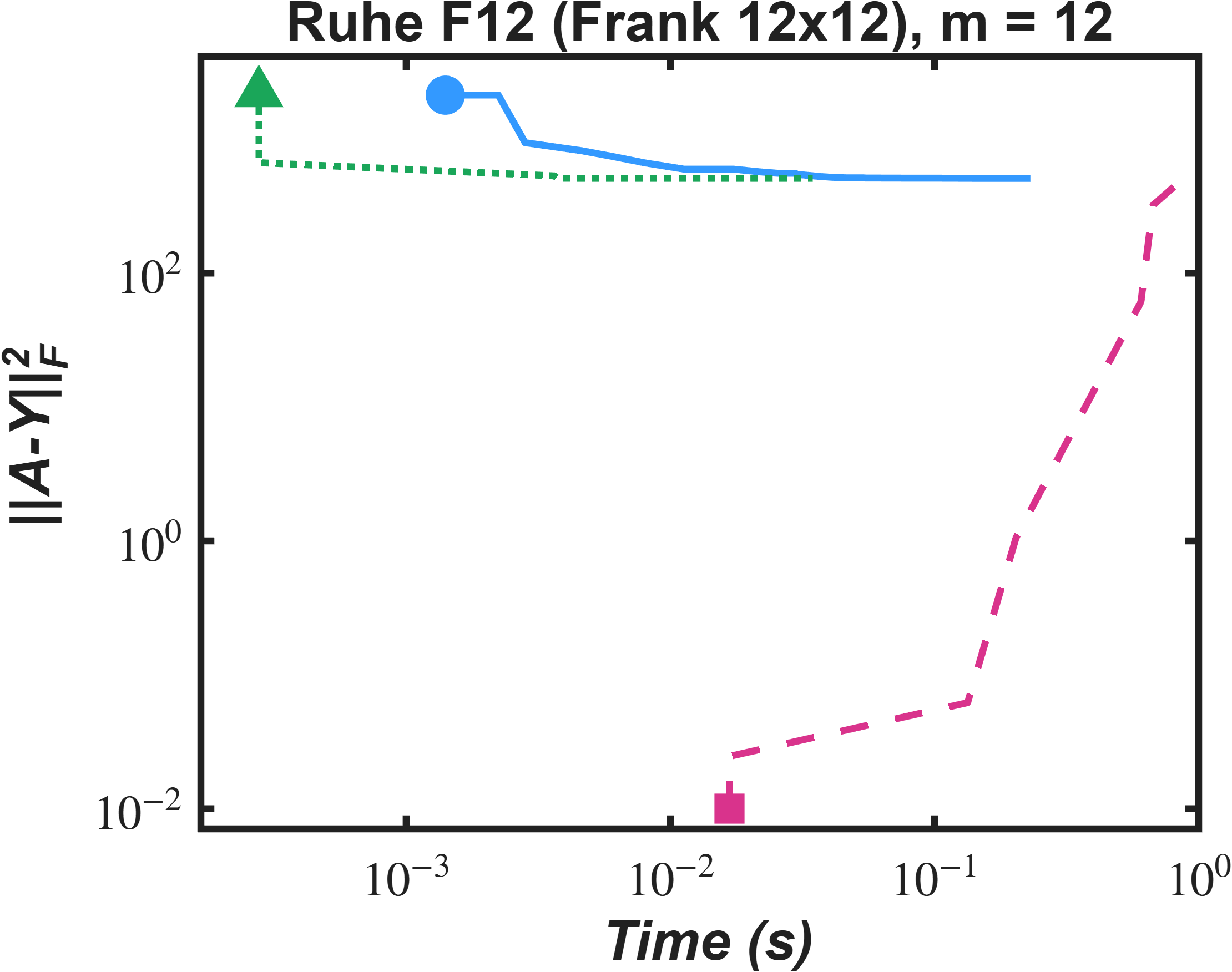}
\end{minipage}\hfill
\begin{minipage}[t]{0.45\textwidth}\vspace*{0pt}
    \includegraphics[width=\textwidth]{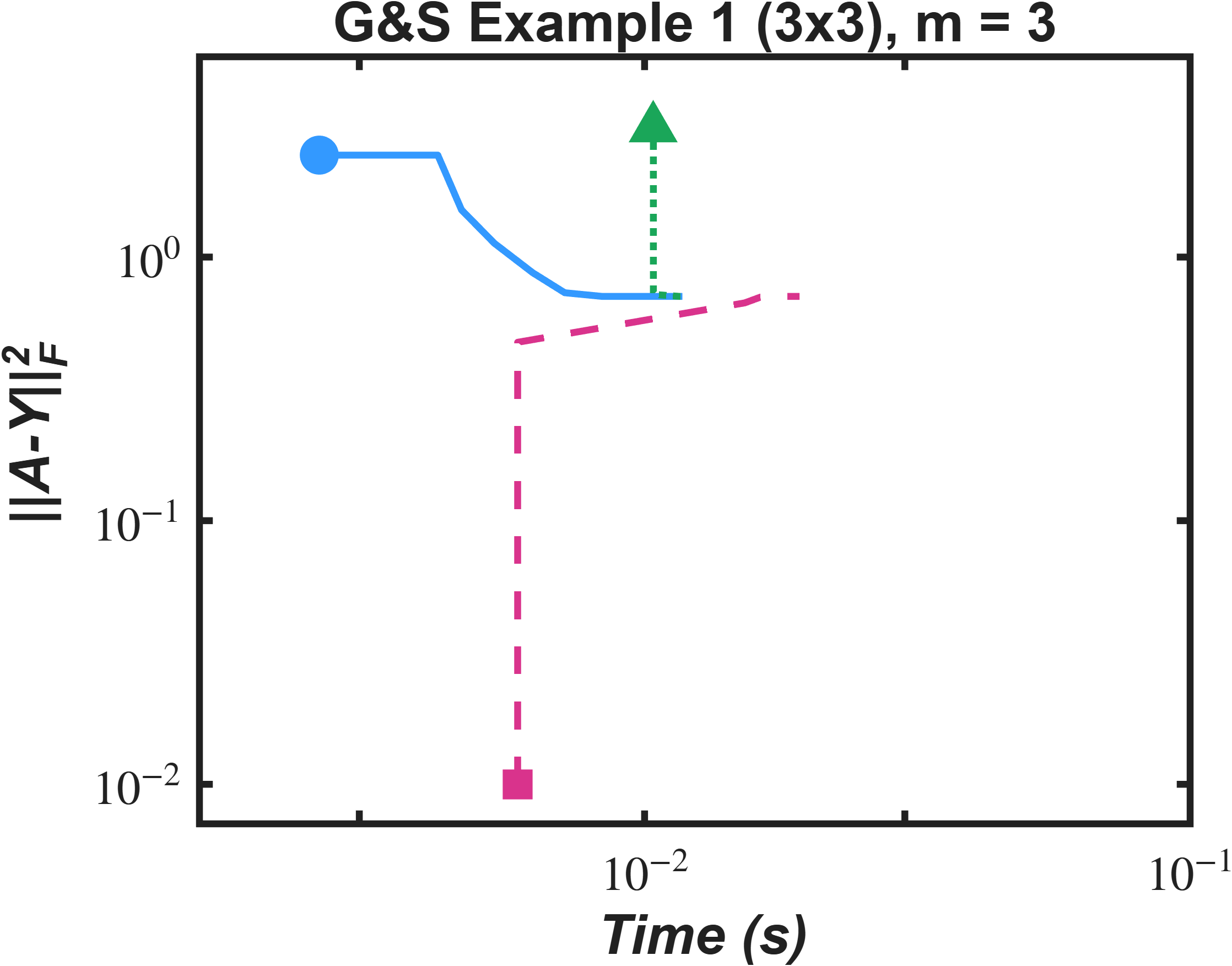}
\end{minipage}\hfill
\begin{minipage}[t]{0.10\textwidth}\vspace*{0pt}
    \phantom{\includegraphics[width=\textwidth]{Fig3f.png}}
\end{minipage} 
\end{figure}
\newpage 
\begin{figure}[H]
    \centering
\begin{minipage}[t]{0.45\textwidth}\vspace*{0pt}
    \includegraphics[width=\textwidth]{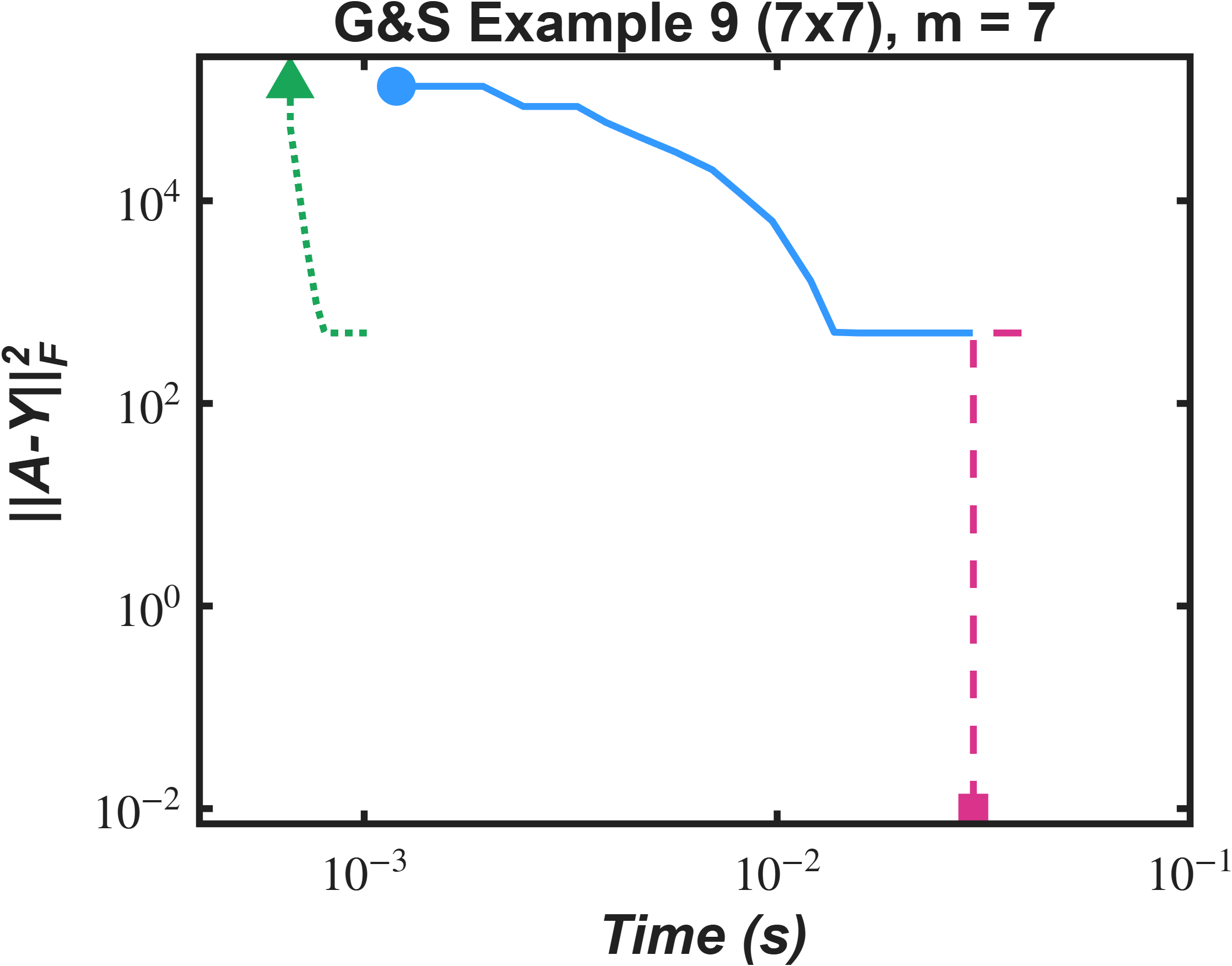}
\end{minipage}
    \caption{Residual $\|A-Y\|_F^2$ (note $A$ is not strictly in $\mathcal{N}$) vs.~time to solve, for the matrices in Table~\ref{tab:literature}}\label{fig:complexLit}
\end{figure}

A cursory glance at Figure~\ref{fig:complexLit} may be slightly misleading, as Guglielmi and Scalone's algorithm seemingly starts at an excellent initial guess and moves further away until it settles at a much worse one, whereas both our method and Ruhe's start further away and converge toward a better residual. However, this difference in curves is attributable to how these algorithms fundamentally differ in structure. As mentioned previously, Guglielmi and Scalone's algorithm only approaches the normal matrices in the limit (its prior iterates being nonnormal). Thus, if we were to plot their residuals strictly over $A\in \mathcal{N}$, then their residual curve would only appear as the single point where the algorithm terminates within a satisfactory tolerance. Instead, we plot the residuals over $A\in \M_m(\C)$ to visualize the progress of
Guglielmi and Scalone's algorithm, which for $Y\notin \mathcal{N}$ begins at essentially $Y$ itself ($A = Y + 0.1 E(0.1)$) and simultaneously approaches normality while attempting to minimize distance to $Y$; $A$ remains strictly in $\mathcal{N}$ for our method and Ruhe's.

For Example~1, all three methods recover $\|A-Y\|_F^2 = 0.709$, and likewise for
Example~9 all three recover the same $2$-norm distance of ${\|A-Y\|_2 = 12.813}$ to three decimal places; both are less than the values reported in \cite{scalone}
(${\|X_c-Y\|_F^2 = 0.919}$ and ${\|A-Y\|_2 = 12.8569}$, respectively). Guglielmi and Scalone recover the matrix $X_c$ via their own implementation of Ruhe's algorithm for Example 1, and they obtain Example 9's $2$-norm value from their algorithm. We were unable to reproduce either value from these examples.
\newpage

\begin{table}[H]
\centering
\caption{Scaling on random $Y$ with uniformly distributed complex entries. Columns are as in Table~\ref{tab:literature}}
\label{tab:scaling}
\begin{tabular}{c|c|c|c|c|c}
\hline
\rule[-1.2ex]{0pt}{4.2ex}$m$ & Method & Final Residual & $\|AA^*-A^*A\|_F$ & Iterations & Time (s) \\
\hline
\multirow{3}{*}{10}
  & \lsnpp\ NPP  & 4.202 & 1.3e-14 & 23 & 0.088 \\
  & \lsgs\ G\&S  & 4.202 & 2.8e-06 & 25\,/\,3973 & 0.260 \\
  & \lsruhe\ Ruhe & 4.202 & 3.8e-14 & 57 & 0.004 \\
\hline
\multirow{3}{*}{20}
  & \lsnpp\ NPP  & 16.302 & 5.1e-14 & 26 & 0.318 \\
  & \lsgs\ G\&S  & 16.325 & 2.5e-06 & 28\,/\,13294 & 3.20 \\
  & \lsruhe\ Ruhe & 16.325 & 4.8e-13 & 380 & 0.109 \\
\hline
\multirow{3}{*}{50}
  & \lsnpp\ NPP  & 97.700 & 4.3e-13 & 30 & 1.25 \\
  & \lsgs\ G\&S  & 97.591 & 2.3e-05 & 35\,/\,13084 & 25.41 \\
  & \lsruhe\ Ruhe & 97.693 & 1.7e-11 & 1033 & 2.76 \\
\hline
\multirow{3}{*}{100}
  & \lsnpp\ NPP  & 389.361 & 2.2e-12 & 95 & 18.24 \\
  & \lsgs\ G\&S  & 389.586 & 6.6e-06 & 16\,/\,50876 & 532 \\
  & \lsruhe\ Ruhe & 389.504 & 1.6e-10 & 2034 & 33.05 \\
\hline
\multirow{3}{*}{200}
  & \lsnpp\ NPP  & 1560.426 & 1.2e-11 & 135 & 61.61 \\
  & \lsgs\ G\&S  & 1560.821 & 2.6e-06 & 15\,/\,49319 & 2600 \\
  & \lsruhe\ Ruhe & 1560.412 & 9.5e-10 & 2254 & 255 \\
\hline
\multirow{3}{*}{250}
  & \lsnpp\ NPP  & 2428.283 & 1.8e-11 & 193 & 352 \\
  & \lsgs\ G\&S  & 2428.891 & 9.1e-05 & 30\,/\,68833 & 5290 \\
  & \lsruhe\ Ruhe & 2428.565 & 9.0e-09 & 6516 & 1387 \\
\hline
\end{tabular}
\end{table}

In Table~\ref{tab:scaling}, we analyze how the methods compare as the amount of
data increases. For each instance, we again find that each method settles at a
normal $A$ with a final residual that does not differ significantly from the
others, and there exists an instance where each method attains a marginally
lower residual than the other two, so there is no consistent winner.

As in Table~\ref{tab:literature}, Guglielmi and Scalone's recovered $A$ has the
disadvantage that $\|AA^*-A^*A\|_F$ is larger in each instance than for the other
two methods. While we could tighten the stopping criteria for their algorithm so
that $A$ better satisfies this normality condition, we observe that their
algorithm, at least on the above examples, already requires more compute time to terminate
than the other two, and this
disparity only widens as the matrix size increases.

As in Table~\ref{tab:literature}, Ruhe's algorithm remains the fastest method on
the smaller matrices ($m = 10$ and $m = 20$), but is overtaken by our method as the
matrix size grows large. The crossover in compute time occurs at $m = 50$, and
beyond this point our method is fastest in each instance.
\newpage

\begin{figure}[H]
    \centering
    \begin{minipage}[t]{0.45\textwidth}\vspace*{0pt}
    \includegraphics[width=\textwidth]{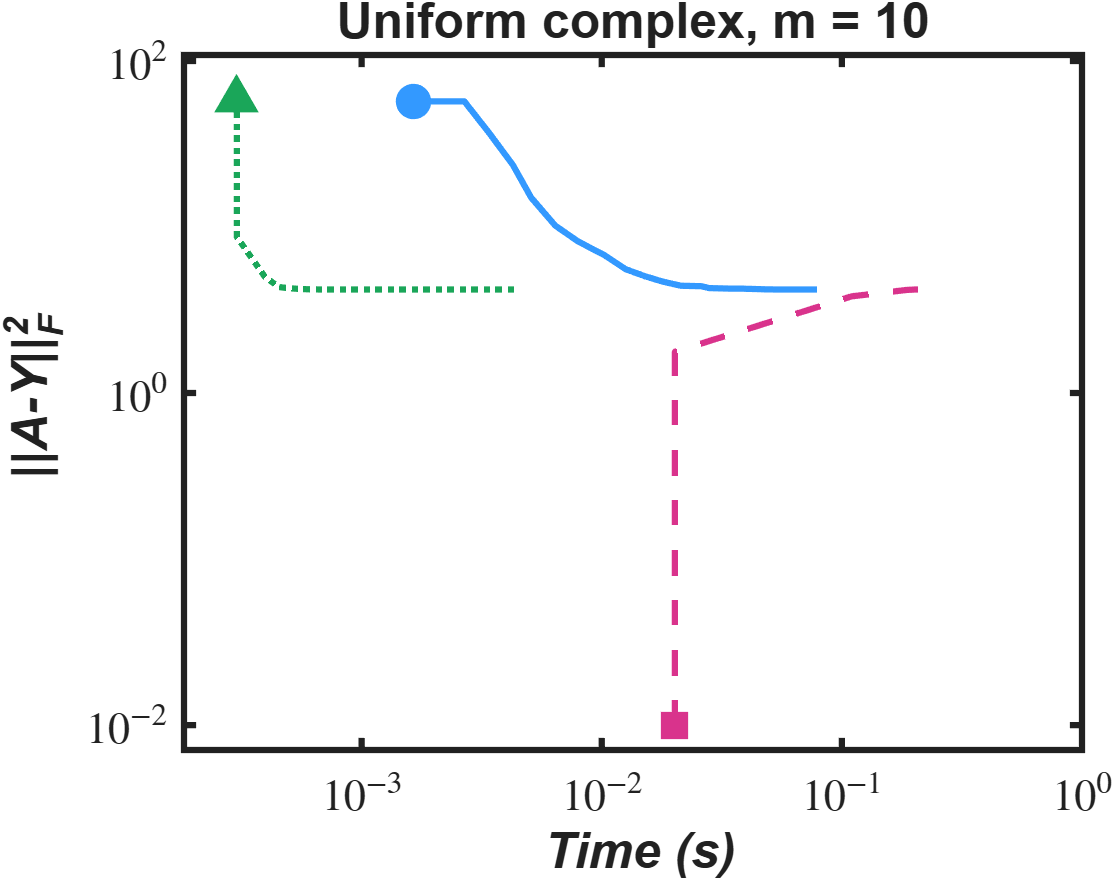}
\end{minipage}\hfill
\begin{minipage}[t]{0.45\textwidth}\vspace*{0pt}
    \includegraphics[width=\textwidth]{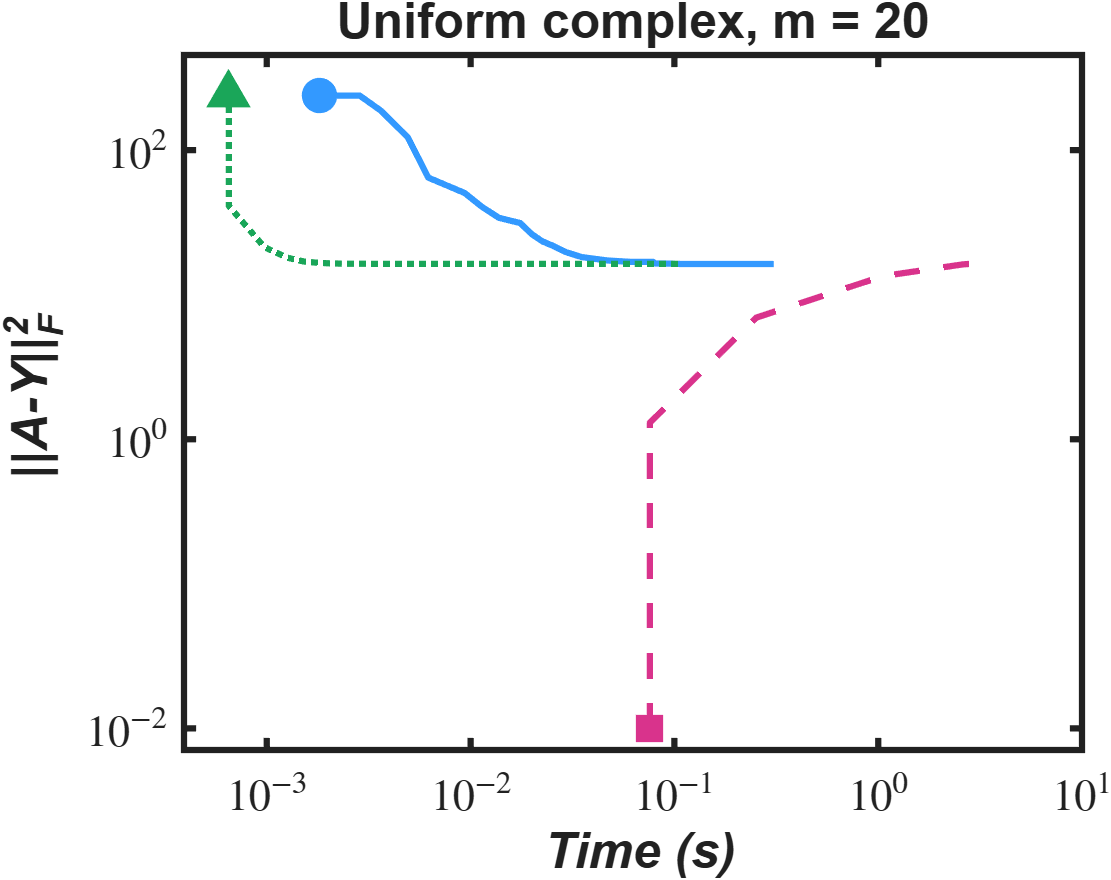}
\end{minipage}\hfill
\begin{minipage}[t]{0.10\textwidth}\vspace*{6.8pt}
    \includegraphics[width=\textwidth]{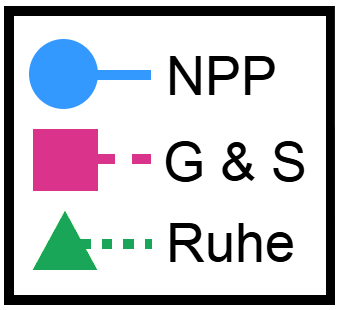}
\end{minipage}
    \\[5pt]
\begin{minipage}[t]{0.45\textwidth}\vspace*{0pt}
    \includegraphics[width=\textwidth]{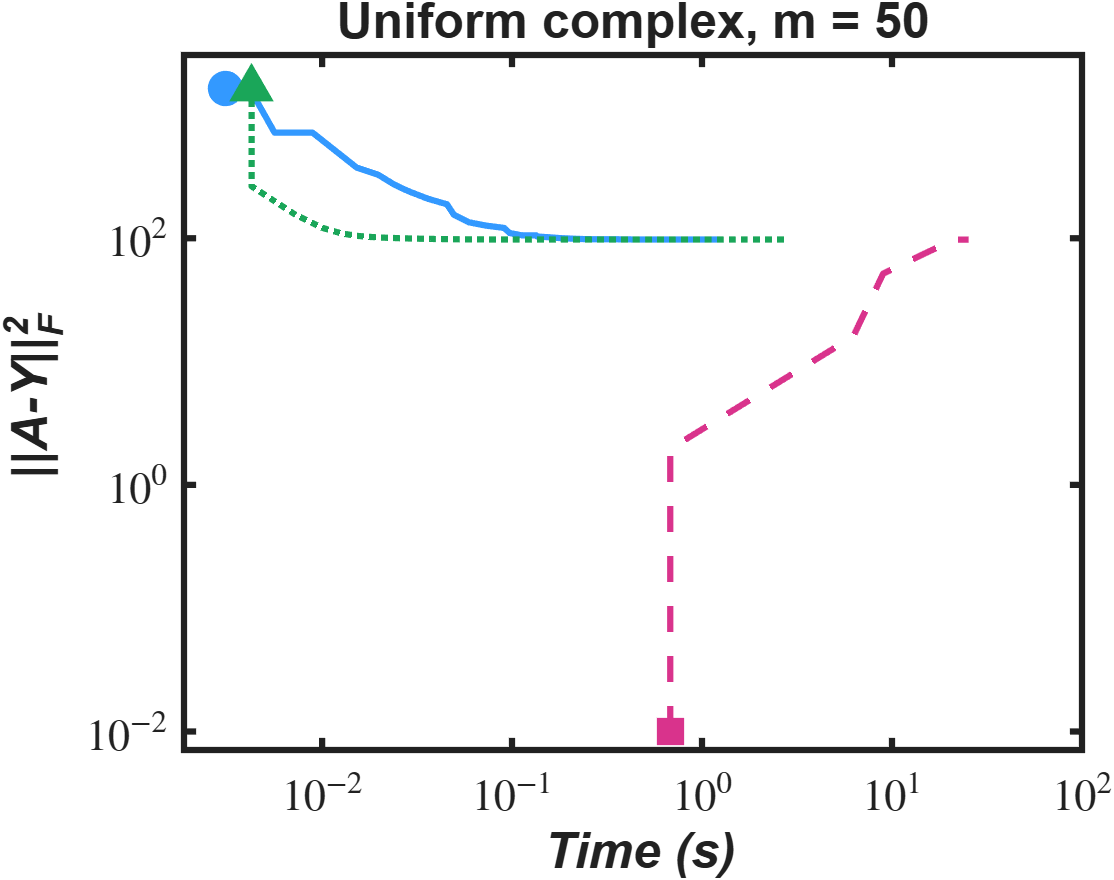}
\end{minipage}\hfill
\begin{minipage}[t]{0.45\textwidth}\vspace*{0pt}
    \includegraphics[width=\textwidth]{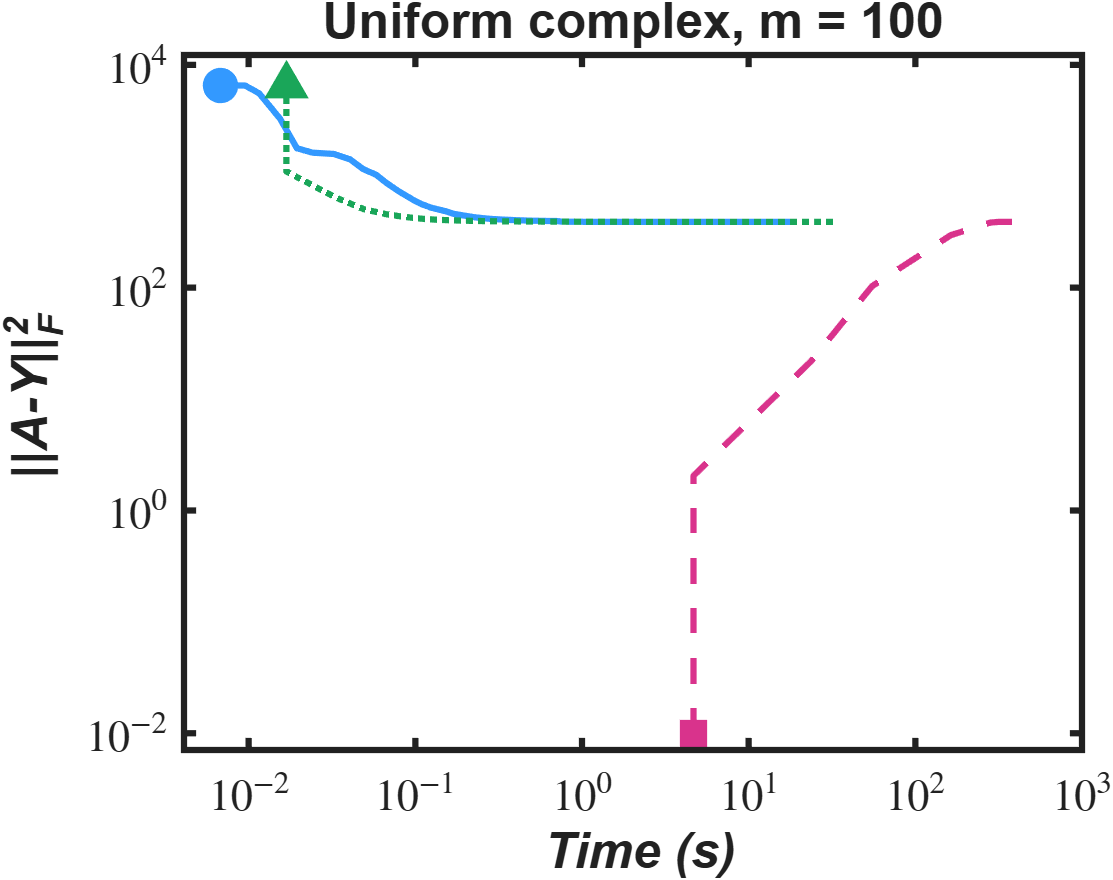}
\end{minipage}\hfill
\begin{minipage}[t]{0.10\textwidth}\vspace*{0pt}
    \phantom{\includegraphics[width=\textwidth]{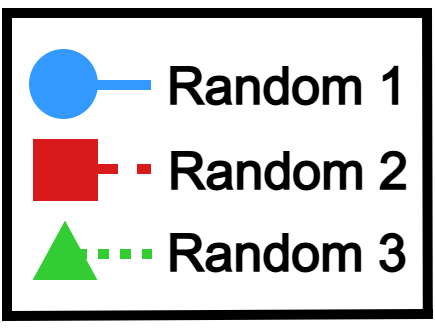}}
\end{minipage}   \\[5pt]
\begin{minipage}[t]{0.45\textwidth}\vspace*{0pt}
    \includegraphics[width=\textwidth]{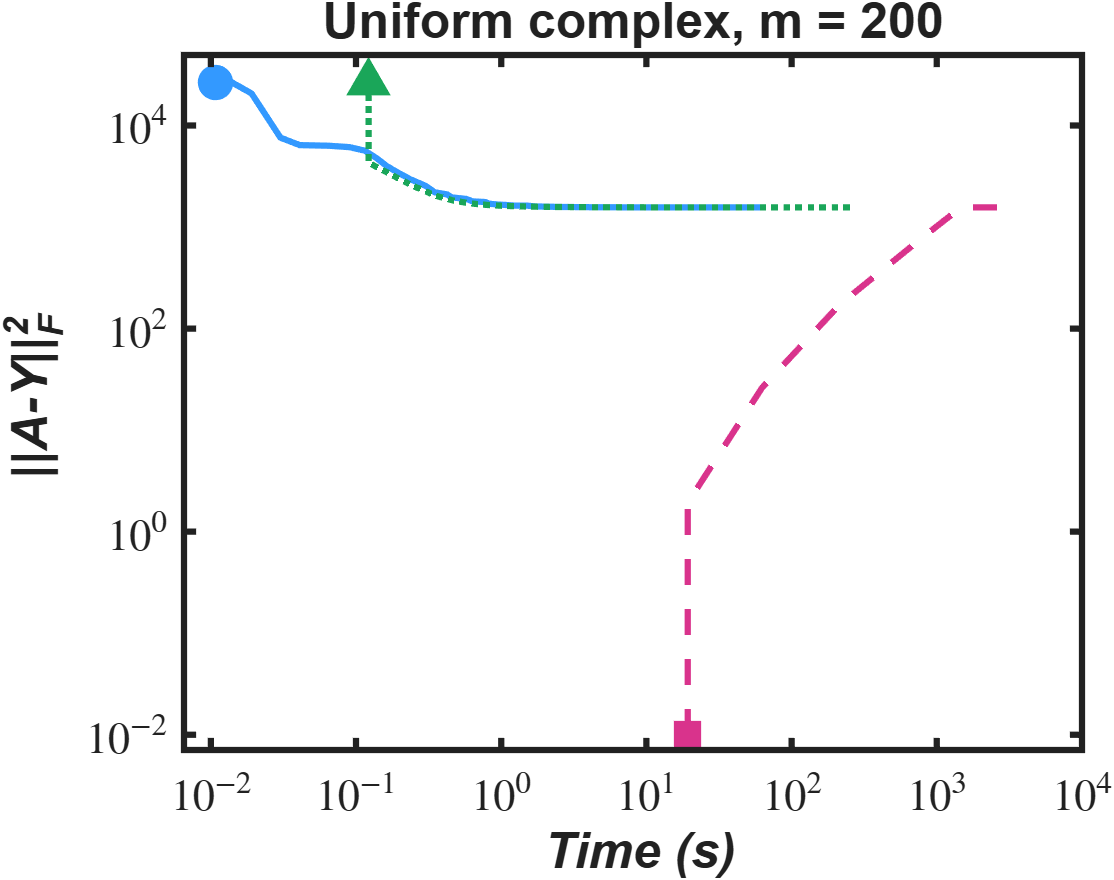}
\end{minipage}\hfill
\begin{minipage}[t]{0.45\textwidth}\vspace*{0pt}
    \includegraphics[width=\textwidth]{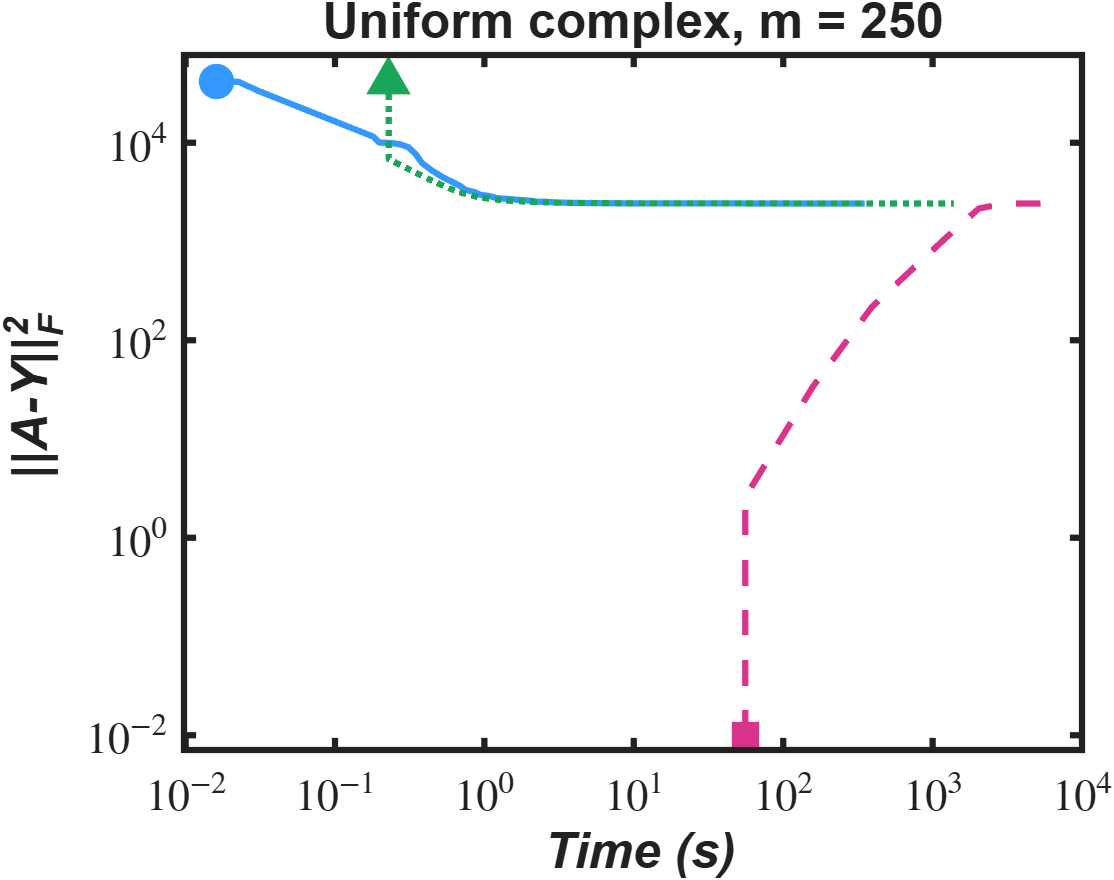}
\end{minipage}\hfill
\begin{minipage}[t]{0.10\textwidth}\vspace*{0pt}
    \phantom{\includegraphics[width=\textwidth]{Fig4f.png}}
\end{minipage}
    \caption{Residual $\|A-Y\|_F^2$ (note $A$ is not strictly in $\mathcal{N}$) vs.~time to solve, for the matrices in Table~\ref{tab:scaling}}\label{fig:complexScale}
\end{figure}

The stopping tolerance of Ruhe's algorithm could be adjusted so that it
converges in a comparable amount of time to ours on these examples, since the residual curves of Ruhe's algorithm in
Figure~\ref{fig:complexScale} appear to flatten at a similar point on the plots as
ours---suggesting that earlier termination may sacrifice little accuracy.  Regardless, a clear trend emerges on these examples: the time required for both Ruhe's method and Guglielmi and Scalone's to terminate scales at an evidently greater rate than ours as matrix size increases.

We summarize the computational cost of each method based on algorithmic structure and empirical evidence:

\begin{itemize}
    \item The computational cost of our method is dominated by a constant number of $m \times m$ matrix products---in forming the Euclidean gradient~\eqref{eq:euclidean_grad}, applying the Hessian~\eqref{eq:hess}, and retracting onto $\mathrm{U}(m)$---each of which costs $O(m^3)$. The associated worst-case bound on the number of iterates is not necessarily independent of the dimension according to \cite[Theorem~6.20]{boumal} (assuming the necessary assumptions are met), as it depends on the objective's initial optimality gap and the Lipschitz constant of its gradient, both of which may vary with $m$ through the construction of $Y$. In Table~\ref{tab:scaling}, a log--log fit of the iteration count against $m$ gives an exponent of approximately $0.7$, so that we expect our method to scale between $O(m^3)$ and $O(m^4)$ on these examples.

    \item Guglielmi and Scalone's algorithm is computationally dominated by its inner gradient system; the outer iteration contributes only scalar work, since its derivative reuses the gradient already formed by the inner flow. The authors state, ``the main cost of an integration step of [its inner system] is given by the spectral decomposition of the matrix'' $Y + \varepsilon E$, an $O(m^3)$ operation, and as ``this computation has to be done several times, the method has a computational complexity depending on the number of such decompositions'' \cite[Section 5]{scalone}. Note that in Figures \ref{fig:complexLit} and \ref{fig:complexScale} we plot the initial point of their residual curve only after a full inner flow has completed, which is why as $m$ grows this point appears at a substantially later time than those of the other methods. In practice, the number of inner iterations is substantial; for example, in Table~\ref{tab:scaling}, we observe that for $m = 250$ their method performs $68833$ inner iterations against $30$ outer iterations.

    \item Ruhe's algorithm proceeds in sweeps, each applying a plane rotation to all $\binom{m}{2}$ off-diagonal pairs. Since a rotation modifies only two rows and two columns, it costs $O(m)$, so a full sweep costs $O(m^3)$. In Table~\ref{tab:scaling} the sweep count is somewhat noisy but increases with $m$---from $57$ at ${m=10}$ to $6516$ at $m=250$---and does so at an empirically faster rate than the iteration count of our method. Ruhe's algorithm is likely fastest on small matrices because each sweep consists of inexpensive row and column updates, whereas each of our iterations forms several dense $m \times m$ products and, consequently, has a larger constant associated with its $O(m^3)$ cost. As $m$ grows, however, the increasing gap between Ruhe's algorithm's sweep count and our solver's iteration count outweighs the effect of this low sweep cost---causing our method to outperform Ruhe's in wall-clock time on these examples.
\end{itemize}

\section{Reduction of the Real NPP}\label{sec:realRed}

Let $X,Y\in \M_{m\times n}(\R)$, and let $A\in \mathcal{N}_{\R}$. We now shift gears to address the Real NPP. The following lemma offers a useful characterization of real normal matrices, whose complex eigenvalues come in pairs.

\begin{lemma}[Theorem 2.5.8, {\cite{horn}}]\label{lem:Horn}
If $A\in \mathcal{N}_{\R}$ then there exists a real orthogonal $Q\in \mathrm{O}(m)$ such that $Q^{\mathsf{T}}AQ$ is a real quasidiagonal matrix
\begin{equation}\label{eq:realSum}
    \Delta := Q^{\mathsf{T}}AQ=A_1\oplus \cdots \oplus A_t \in \M_m(\R),\text{ each } A_i \text{ is } 1\times 1 \text{ or } 2\times 2
\end{equation}
with the following properties: 
\begin{enumerate}[label=(\alph*)]
\item The $1\times 1$ direct summands in \eqref{eq:realSum} display all real eigenvalues of $A$, and each $2\times 2$ direct summand in \eqref{eq:realSum} has the special form
\[\begin{bmatrix}
    a & b\\
    -b & a
\end{bmatrix},\]which is normal and has eigenvalues $a\pm ib$, and $b>0$.

\item The direct summands in \eqref{eq:realSum}, which may appear in any prescribed order, are completely determined by the eigenvalues of $A$, and their eigenvalues are collectively the eigenvalues of $A$.
\end{enumerate}
\end{lemma}

Following Lemma~\ref{lem:Horn}, write $A=Q\Delta Q^{\mathsf{T}}$, and let the block diagonal $\Delta$ consist of $\lfloor m/2 \rfloor$ blocks of the form
\[
\Delta_i \in \Psi \cup \Omega, \quad \text{for } i = 1, \ldots, \lfloor m/2\rfloor,
\]
where
\[
\Psi := \left\{ \begin{bmatrix} a & b \\ -b & a \end{bmatrix} \right\},~
\Omega := \left\{ \begin{bmatrix} \lambda' & 0 \\ 0 & \lambda'' \end{bmatrix} \right\} \subseteq \M_2(\R),
\]
and a singleton block $\begin{bmatrix}
\lambda_m
\end{bmatrix}\in \M_1(\R)$, if $m$ is odd. Note that $\lambda'$ and $\lambda''$ may be distinct. 
\begin{rmk}
Lemma~\ref{lem:Horn} ensures $b>0$, which can always be attained by performing a similarity via the matrix $\begin{bmatrix}1 & 0 \\ 0 & -1\end{bmatrix}$. For our definition of $\Psi$ we drop this normalization; the sign of $b$ is absorbed into $Q$, so the set of
attainable residuals in \eqref{eq:real_block_decomp} is unchanged.
\end{rmk}

Then
\begin{equation}\label{eq:real_block_decomp}
\begin{aligned}
\|AX-Y\|_F^2 &= \|\Delta Q^{\mathsf{T}}X - Q^{\mathsf{T}}Y\|_F^2 \\
&= \sum_{i=1}^{\lfloor m/2\rfloor} \left\|\Delta_i \begin{bmatrix} (Q^{\mathsf{T}}X)_{2i-1} \\ (Q^{\mathsf{T}}X)_{2i} \end{bmatrix} - \begin{bmatrix} (Q^{\mathsf{T}}Y)_{2i-1} \\ (Q^{\mathsf{T}}Y)_{2i} \end{bmatrix}\right\|_F^2 \\
&~~~+ \|\lambda_m (Q^{\mathsf{T}}X)_m - (Q^{\mathsf{T}}Y)_m\|_F^2,
\end{aligned}
\end{equation}
with the final term present only if $m$ is odd. 
Since the terms in~\eqref{eq:real_block_decomp} are decoupled across $i$, we
minimize over each block independently. Before we do so, we redefine
\begin{align*}
\phi_i := \|(Q\T X)_i\|_F^2, \quad \text{and}\quad \gamma_i := (Q \T Y)_i(Q\T X)_i\T
\end{align*}
to be over the reals for the remainder of this paper. We also introduce a few new terms to shorten the following lemma. Let
\begin{align*}
\alpha_i &:=  (Q\T X)_{2i-1}(Q\T Y)_{2i-1}\T + (Q\T X)_{2i}(Q\T Y)_{2i}\T,\\
\beta_i&:= (Q\T X)_{2i}(Q\T Y)_{2i-1}\T - (Q\T X)_{2i-1}(Q\T Y)_{2i}\T.
\end{align*}
Further define the residual of the $i$th block term
\begin{equation}\label{eq:rho}\rho(\Delta_i):= \left\|\Delta_i \begin{bmatrix} (Q^{\mathsf{T}}X)_{2i-1} \\ (Q^{\mathsf{T}}X)_{2i} \end{bmatrix} - \begin{bmatrix} (Q^{\mathsf{T}}Y)_{2i-1} \\ (Q^{\mathsf{T}}Y)_{2i} \end{bmatrix}\right\|_F^2.\end{equation}

Thus, we arrive at:

\begin{lemma}\label{lem:optDelt}
  For fixed $Q \in \mathrm{O}(m)$, a global minimizer of~\eqref{eq:real_block_decomp}
  over the real block diagonal $\Delta$ (with blocks in $\Psi \cup \Omega$) is given by
  \[\Delta^\star_Q = \begin{cases}
  \mathrm{diag}\left(\Delta_1^\star,\ldots,\Delta_{\lfloor m/2\rfloor}^\star\right), & \text{if } m \text{ is even},\\[8pt]
   \mathrm{diag}\left(\Delta_1^\star,\ldots,\Delta_{\lfloor m/2\rfloor}^\star, \lambda_m^\star\right), & \text{if } m \text{ is odd}.
  \end{cases} \]
    
  For each
  $i = 1,\ldots,\lfloor m/2\rfloor$,
  \begin{equation}\label{eq:optBlock}
      \Delta_i^\star = \begin{cases}
      \begin{bmatrix} a_i^\star & b_i^\star \\ -b_i^\star & a_i^\star \end{bmatrix}, & \text{if }\rho\left(\begin{bmatrix} a_i^\star & b_i^\star \\ -b_i^\star & a_i^\star \end{bmatrix}\right) < \rho\left(\begin{bmatrix} \lambda_{2i-1}^\star & 0 \\ 0 & \lambda_{2i}^\star \end{bmatrix}\right),\\[16pt]
      \begin{bmatrix} \lambda_{2i-1}^\star & 0 \\ 0 & \lambda_{2i}^\star \end{bmatrix} & \text{if } \rho\left(\begin{bmatrix} a_i^\star & b_i^\star \\ -b_i^\star & a_i^\star \end{bmatrix}\right)\geq \rho\left(\begin{bmatrix} \lambda_{2i-1}^\star & 0 \\ 0 & \lambda_{2i}^\star \end{bmatrix}\right),
      \end{cases}
  \end{equation}
  where
  \begin{equation}\label{eq:optRot}
       (a_i^\star,  b_i^\star) =
      \begin{cases}
          \left(\dfrac{\alpha_i}{\phi_{2i-1}+\phi_{2i}},\
                \dfrac{\beta_i}{\phi_{2i-1}+\phi_{2i}}\right) & \text{if } \phi_{2i-1}+\phi_{2i} \neq 0, \\[8pt]
          (0,0) & \text{if } \phi_{2i-1}+\phi_{2i} = 0,
      \end{cases}
  \end{equation}
  and for $j \in \{2i-1,\, 2i\}$,
  \begin{equation}\label{eq:optDiag}
      \lambda_j^\star =
      \begin{cases}
          \dfrac{\gamma_j}{\phi_j} & \text{if } (Q\T X)_j \neq 0, \\[8pt]
          0 & \text{if } (Q\T X)_j = 0.
      \end{cases}
  \end{equation}

  If $m$ is odd,
  \begin{equation}\label{eq:optLast}
      \lambda_m^\star =
      \begin{cases}
          \dfrac{\gamma_m}{\phi_m} & \text{if } (Q\T X)_m \neq 0, \\[8pt]
          0 & \text{if } (Q\T X)_m = 0.
      \end{cases}
  \end{equation}
\end{lemma}

\begin{proof}The case~\eqref{eq:optLast} is identical to Lemma~\ref{lem:optD} restricted to the reals. For the block form, $\Delta_i$ is in the union of two sets of matrices. Thus, if 
    \[ \Psi_i^\star:=\underset{\Delta_i \in \Psi}{\argmin}\ \rho(\Delta_i) = \begin{bmatrix} a_i^\star & b_i^\star \\ -b_i^\star & a_i^\star \end{bmatrix} 
\quad\text{and}\quad\Omega_i^\star:=
\underset{\Delta_i \in \Omega}{\argmin}\ \rho(\Delta_i) = \begin{bmatrix} \lambda_{2i-1}^\star & 0 \\ 0 & \lambda_{2i}^\star \end{bmatrix},\]
    then
\[\Delta_i^\star=\underset{\Delta_i \in {\Psi \cup \Omega}}{\argmin}\rho(\Delta_i)=\underset{\Delta_i \in \{\Psi_i^\star, \Omega_i^\star\}}{\argmin}\rho(\Delta_i),\]
with \eqref{eq:optBlock} selecting the diagonal form in case of a tie. $\Omega_i^\star$ follows exactly as in Lemma~\ref{lem:optD}, since the optimal diagonal values are decoupled within each block.  

It remains to be shown then that $\Psi_i^\star$ is as in \eqref{eq:optRot}. Hence, we allow $\Delta_i\in \Psi$ and compute

\begin{equation}\label{eq:realResid}
    \begin{aligned}\rho\left(\Delta_i\right)&=\|a_i(Q\T X)_{2i-1} + b_i(Q\T X)_{2i} - (Q\T Y)_{2i-1}\|_F^2 \\
&~~~+ \|-b_i(Q\T X)_{2i-1} + a_i(Q\T X)_{2i} - (Q\T Y)_{2i}\|_F^2.\end{aligned}
    \end{equation}
Expanding each term:
\begin{align*}
&\|a_i(Q\T X)_{2i-1} + b_i(Q\T X)_{2i} - (Q\T Y)_{2i-1}\|_F^2 \\
&\quad= a_i^2\|(Q\T X)_{2i-1}\|_F^2 + b_i^2\|(Q\T X)_{2i}\|_F^2 
+ 2a_ib_i(Q\T X)_{2i-1}(Q\T X)_{2i}\T \\
&\qquad - 2a_i(Q\T X)_{2i-1}(Q\T Y)_{2i-1}\T 
- 2b_i(Q\T X)_{2i}(Q\T Y)_{2i-1}\T 
+ \|(Q\T Y)_{2i-1}\|_F^2, \\[6pt]
&\|-b_i(Q\T X)_{2i-1} + a_i(Q\T X)_{2i} - (Q\T Y)_{2i}\|_F^2 \\
&\quad= b_i^2\|(Q\T X)_{2i-1}\|_F^2 + a_i^2\|(Q\T X)_{2i}\|_F^2 
- 2a_ib_i(Q\T X)_{2i-1}(Q\T X)_{2i}\T \\
&\qquad - 2a_i(Q\T X)_{2i}(Q\T Y)_{2i}\T 
+ 2b_i(Q\T X)_{2i-1}(Q\T Y)_{2i}\T 
+ \|(Q\T Y)_{2i}\|_F^2.
\end{align*}
The terms involving $(Q\T X)_{2i-1}(Q\T X)_{2i}\T$ cancel upon summing, and we obtain:
\begin{align*}
\rho\left(\Delta_i\right)&=(a_i^2+b_i^2)\bigl(\|(Q\T X)_{2i-1}\|_F^2+\|(Q\T X)_{2i}\|_F^2\bigr) \\
&\quad - 2a_i\bigl((Q\T X)_{2i-1}(Q\T Y)_{2i-1}\T + (Q\T X)_{2i}(Q\T Y)_{2i}\T\bigr) \\
&\quad - 2b_i\bigl((Q\T X)_{2i}(Q\T Y)_{2i-1}\T - (Q\T X)_{2i-1}(Q\T Y)_{2i}\T\bigr) \\
&\quad + \|(Q\T Y)_{2i-1}\|_F^2 + \|(Q\T Y)_{2i}\|_F^2,
\end{align*}
which further reduces to:
\begin{equation}\label{eq:block_expanded}
\begin{aligned}
(\phi_{2i-1}+\phi_{2i})(a_i^2+b_i^2)
- 2a_i\alpha_i- 2b_i\beta_i + \|(Q\T Y)_{2i-1}\|_F^2 + \|(Q\T Y)_{2i}\|_F^2. 
\end{aligned}
\end{equation}

If $\phi_{2i-1}+\phi_{2i} = 0$, then $(Q\T X)_{2i-1} = (Q\T X)_{2i} = 0$, and \eqref{eq:block_expanded} is independent of $(a_i, b_i)$; set $a_i^\star = b_i^\star = 0$. 
Otherwise, $0<\phi_{2i-1}+\phi_{2i}$, and we complete the square to obtain:
\begin{align*}
      \rho\left(\Delta_i\right)&=(\phi_{2i-1}+\phi_{2i})\left(\left(a_i-\dfrac{\alpha_i}{\phi_{2i-1}+\phi_{2i}}\right)^2+\left(b_i-\dfrac{\beta_i}{\phi_{2i-1}+\phi_{2i}}\right)^2\right) \\
     & ~~~-  \dfrac{\alpha_i^2+\beta_i^2}{\phi_{2i-1}+\phi_{2i}}+ \|(Q\T Y)_{2i-1}\|_F^2 + \|(Q\T Y)_{2i}\|_F^2.
\end{align*}
   Since the first term is nonnegative and the remaining terms are independent of $a_i$ and $b_i$, the above is minimized if and only if $(a_i,b_i)=(a_i^\star,b_i^\star)$ as in \eqref{eq:optRot}.
 \end{proof}

 As in the complex case, we write $\Delta^\star_Q$ for the assembled optimal quasidiagonal to emphasize its dependence on $Q$, while the $i$th block is written $\Delta_i^\star$ and the scalars $\phi_i$, $\gamma_i$, $\alpha_i$, $\beta_i$, $a_i^\star$, $b_i^\star$, and $\lambda_j^\star$ carry the same dependence implicitly. We can now state our reduced objective:

\begin{theorem}\label{thm:realreduced}
  The Real Normal Procrustes Problem is equivalent to determining
  \begin{equation}\label{eq:Qminimize}
      Q^\star \in \underset{Q \in \mathrm{O}(m)}{\argmax}\, g(Q),
  \end{equation}
  where
  \begin{equation}\label{eq:gQ}
    g(Q) = \sum_{i=1}^{\lfloor m/2\rfloor}
      \max\!\left\{
        \frac{\alpha_i^2 + \beta_i^2}{\phi_{2i-1}+\phi_{2i}},\;
        \frac{\gamma_{2i-1}^2}{\phi_{2i-1}} + \frac{\gamma_{2i}^2}{\phi_{2i}}
      \right\}
      + \frac{\gamma_m^2}{\phi_m}=\|\Delta^\star_Q Q\T X\|_F^2,
  \end{equation}
  with the convention that any term with a vanishing denominator equals zero, and the final
  term present only if $m$ is odd. An optimal $A^\star = Q^\star \Delta^\star_{Q^\star} (Q^\star)^{\mathsf T}$
  is real and normal, where $\Delta^\star_{Q^\star}$ is given by Lemma~\ref{lem:optDelt}.
\end{theorem}
\begin{proof}
  Substituting~\eqref{eq:optBlock} into the $i$th block term of~\eqref{eq:real_block_decomp}, 
  the residual for the $i$th block yields:
\begin{align*}
       \left\| \begin{bmatrix} (Q^{\mathsf{T}}Y)_{2i-1} \\ (Q^{\mathsf{T}}Y)_{2i} \end{bmatrix}\right\|_F^2-\max\!\left\{
        \frac{\alpha_i^2 + \beta_i^2}{\phi_{2i-1}+\phi_{2i}},\;
        \frac{\gamma_{2i-1}^2}{\phi_{2i-1}} + \frac{\gamma_{2i}^2}{\phi_{2i}}
      \right\},
\end{align*}
  and if $m$ is odd, \eqref{eq:optLast} yields:
  \begin{align*}
       \left\| (Q^{\mathsf{T}}Y)_{m}\right\|_F^2-\frac{\gamma_m^2}{\phi_m}.
\end{align*}
  
Hence, upon summing the above terms, we arrive at the final result:
  \[
  \sum_{i=1}^{\lfloor m/2\rfloor}\left\| \begin{bmatrix} (Q^{\mathsf{T}}Y)_{2i-1} \\ (Q^{\mathsf{T}}Y)_{2i} \end{bmatrix}\right\|_F^2
  \;+\; \|(Q\T Y)_m\|_F^2
  \;-\; g(Q)
  = \|Q\T Y\|_F^2 - g(Q) = \|Y\|_F^2 - g(Q),
  \]
  where the second equality holds by orthogonal invariance of the Frobenius norm. Since $\|Y\|_F^2$ is
  constant, minimizing over $\mathrm{O}(m)$ is equivalent to determining~\eqref{eq:Qminimize}.
\end{proof}

\begin{example}
    Observe that $g$, as in the complex version, is also not geodesically concave in general. In fact, we can use the same matrices from Example~\ref{ex:geodesicComplex} to show that along the same geodesic $c(t)$,
  \begin{align*}
      g(c(t)) = \max\!\left\{
        \frac{\alpha_1^2 + \beta_1^2}{\phi_1+\phi_2},\;
        \frac{\gamma_1^2}{\phi_1} + \frac{\gamma_2^2}{\phi_2}
      \right\}
      &= \max\!\left\{\tfrac12,\; \cos^4\!\left(\tfrac{t\pi}{2}\right) + \sin^4\!\left(\tfrac{t\pi}{2}\right)\right\}\\
      &= \cos^4\!\left(\tfrac{t\pi}{2}\right) + \sin^4\!\left(\tfrac{t\pi}{2}\right),
  \end{align*}
  since $\cos^4\!\left(\tfrac{t\pi}{2}\right) + \sin^4\!\left(\tfrac{t\pi}{2}\right) \geq \tfrac12$, so the rest of Example~\ref{ex:geodesicComplex} follows as before.
\end{example}

\section{Computing the Gradient and Hessian: Real Case}\label{sec:realGradient}

We continue in a manner analogous to Section~\ref{sec:complexGradient}. The discussion of the $\{\phi_i = 0\}$ locus there applies verbatim to $g$, restricted to the reals; this argument also applies when the $\Psi$ branch is the active form of some block $\Delta_i^\star$ and the denominator vanishes. The real objective, however, carries a second source of non-smoothness with no complex analogue. Since each block term of $g$ is defined in~\eqref{eq:gQ} through a pointwise maximum of two smooth expressions, $g$ need not be differentiable where the two are equal, i.e.\ on the \emph{tie locus}
\[
\mathcal{T} := \left\{ Q \in \mathrm{O}(m) \;\middle|\; \frac{\alpha_i^2 + \beta_i^2}{\phi_{2i-1}+\phi_{2i}} = \frac{\gamma_{2i-1}^2}{\phi_{2i-1}} + \frac{\gamma_{2i}^2}{\phi_{2i}} \ \text{ for some } i \right\},
\]
across which the active form of some block $\Delta_i^\star$ may switch between branches of~\eqref{eq:optBlock}. Off $\mathcal{T}$, and provided all relevant denominators are nonzero, exactly one branch attains each maximum and the active form of every block is locally constant, so $g$ is smooth and we may differentiate it (and similarly obtain its Hessian) branchwise on this open set.

For each block, the difference of the two branch expressions is real-analytic in $Q$ on each connected component of the set in $\mathrm{O}(m)$ where the relevant denominators are nonzero. If this difference vanishes identically on such a component, the two branches coincide and $g$ is smooth on that block, so that block contributes no non-differentiability. Otherwise, its zero set has measure zero \cite{mityagin}. Hence the set on which $g$ fails to be smooth---together with the set where relevant denominators vanish---is contained in a set of measure zero; a random initialization avoids it with probability $1$. Furthermore, $g$ is a finite sum of pointwise maxima of pairs of smooth functions, hence locally Lipschitz on any region where the relevant denominators are nonzero, and therefore differentiable almost everywhere in such regions by Rademacher's Theorem \cite[Theorem 3.2]{evans}. 

Viewing $\mathrm{O}(m)$ as a Riemannian submanifold of the Euclidean space\linebreak ${\R^{m^2}\cong \M_m(\R)}$, equipped with the inner product ${\langle A, B\rangle = \operatorname{tr}(A^{\mathsf{T}}B)}$, the Riemannian gradient $\mathrm{grad}_Q\,g$ is obtained by projecting the Euclidean gradient $\nabla_Q g$ onto the tangent space $T_Q\mathrm{O}(m)$.

We proceed by computing $\nabla_Q g$. As in Section~\ref{sec:complexGradient}, we take $W\in\M_m(\R)$ for the Euclidean gradient and Hessian, and $W\in T_Q\mathrm{O}(m)$ for the Riemannian gradient and Hessian. Let $\epsilon\in \R$. We redefine $\delta(h) := \frac{d}{d\epsilon}h(Q+\epsilon W)\big|_{\epsilon=0}$. We thus compute:
\[
\delta(\phi_j) = 2\,(Q\T X)_j (W\T X)_j\T, \qquad
\delta(\gamma_j) = (W\T Y)_j (Q\T X)_j\T + (Q\T Y)_j (W\T X)_j\T,\]
\begin{align*}
\delta(\alpha_i) &= (W\T X)_{2i-1}(Q\T Y)_{2i-1}\T + (Q\T X)_{2i-1}(W\T Y)_{2i-1}\T\\
                 &~~~
                 + (W\T X)_{2i}(Q\T Y)_{2i}\T + (Q\T X)_{2i}(W\T Y)_{2i}\T,\\
\delta(\beta_i) &= (W\T X)_{2i}(Q\T Y)_{2i-1}\T + (Q\T X)_{2i}(W\T Y)_{2i-1}\T\\
                 &~~~
                 - (W\T X)_{2i-1}(Q\T Y)_{2i}\T - (Q\T X)_{2i-1}(W\T Y)_{2i}\T.
\end{align*}
Differentiating $g$ branchwise, the $i$th block term of~\eqref{eq:gQ} then satisfies
\begin{align*}
\delta\!\left(\frac{\alpha_i^2+\beta_i^2}{\phi_{2i-1}+\phi_{2i}}\right)
&= \frac{2\alpha_i\,\delta(\alpha_i) + 2\beta_i\,\delta(\beta_i)}{\phi_{2i-1}+\phi_{2i}}
- \frac{\alpha_i^2+\beta_i^2}{\bigl(\phi_{2i-1}+\phi_{2i}\bigr)^2}\bigl(\delta(\phi_{2i-1})+\delta(\phi_{2i})\bigr)\\
&= 2a_i^\star\,\delta(\alpha_i) + 2b_i^\star\,\delta(\beta_i)
- \bigl((a_i^\star)^2+(b_i^\star)^2\bigr)\bigl(\delta(\phi_{2i-1})+\delta(\phi_{2i})\bigr)
\end{align*}
when the first form is active, and
\begin{align*}
\delta\!\left(\frac{\gamma_{2i-1}^2}{\phi_{2i-1}}+\frac{\gamma_{2i}^2}{\phi_{2i}}\right)
&= \sum_{j\in\{2i-1,\,2i\}} \left(\frac{2\gamma_j\,\delta(\gamma_j)}{\phi_j} - \frac{\gamma_j^2}{\phi_j^2}\,\delta(\phi_j)\right)\\
&= \sum_{j\in\{2i-1,\,2i\}} \left(2\lambda_j^\star\,\delta(\gamma_j) - (\lambda_j^\star)^2\,\delta(\phi_j)\right)
\end{align*}
when the second is active---with the singleton term (present if $m$ is odd) obeying the latter with $j=m$. Substituting the derivatives computed above and grouping by the factor carrying $W$, each block contribution takes the common form:
\begin{align*}
    \delta(g)
    &= 2\sum_{i} \Bigg( \underbrace{\operatorname{tr}\!\left[\Delta_i^{\star\T} \begin{bmatrix} (Q\T Y)_{2i-1} \\ (Q\T Y)_{2i} \end{bmatrix} \begin{bmatrix} (W\T X)_{2i-1} \\ (W\T X)_{2i} \end{bmatrix}\T\right]}_{\text{(A)}}\\
     &~~~+ \underbrace{\operatorname{tr}\!\left[\Delta_i^\star \begin{bmatrix} (Q\T X)_{2i-1} \\ (Q\T X)_{2i} \end{bmatrix} \begin{bmatrix} (W\T Y)_{2i-1} \\ (W\T Y)_{2i} \end{bmatrix}\T\right]}_{\text{(B)}}\\
     &~~~- \underbrace{\operatorname{tr}\!\left[\Delta_i^{\star\T}\Delta_i^\star \begin{bmatrix} (Q\T X)_{2i-1} \\ (Q\T X)_{2i} \end{bmatrix} \begin{bmatrix} (W\T X)_{2i-1} \\ (W\T X)_{2i} \end{bmatrix}\T\right]}_{\text{(C)}}\Bigg),
\end{align*}
where the sum runs over the blocks of $\Delta^\star_Q$ (with $\Delta_i^\star = \lambda_m^\star$ and the stacked rows replaced by the single row $m$ for the singleton, if present). Since $\Delta^\star_Q$ is block diagonal, via the same trace identities in Section~\ref{sec:complexGradient}, we sum over the blocks to obtain:
\begin{align*}
&\text{(A):}\quad \sum_i \operatorname{tr}\!\left[\Delta_i^{\star\T} \begin{bmatrix}(Q\T Y)_{2i-1}\\(Q\T Y)_{2i}\end{bmatrix}\begin{bmatrix}(W\T X)_{2i-1}\\(W\T X)_{2i}\end{bmatrix}\T\right]
= \operatorname{tr}\bigl[\Delta_Q^{\star\T} Q\T YX\T W\bigr]\\
&\phantom{\text{(A):}\quad \sum_i \operatorname{tr}\!\left[\Delta_i^{\star\T} \begin{bmatrix}(Q\T Y)_{2i-1}\\(Q\T Y)_{2i}\end{bmatrix}\begin{bmatrix}(W\T X)_{2i-1}\\(W\T X)_{2i}\end{bmatrix}\T\right]}
= \operatorname{tr}\bigl[(XY\T Q\Delta^\star_Q)\T W\bigr],\\
&\text{(B):}\quad \sum_i \operatorname{tr}\!\left[\Delta_i^\star \begin{bmatrix}(Q\T X)_{2i-1}\\(Q\T X)_{2i}\end{bmatrix}\begin{bmatrix}(W\T Y)_{2i-1}\\(W\T Y)_{2i}\end{bmatrix}\T\right]
= \operatorname{tr}\bigl[\Delta^\star_Q Q\T XY\T W\bigr]\\
&\phantom{\text{(B):}\quad \sum_i \operatorname{tr}\!\left[\Delta_i^\star \begin{bmatrix}(Q\T X)_{2i-1}\\(Q\T X)_{2i}\end{bmatrix}\begin{bmatrix}(W\T Y)_{2i-1}\\(W\T Y)_{2i}\end{bmatrix}\T\right]}
= \operatorname{tr}\bigl[\big(YX\T Q(\Delta^\star_Q)\T\big)\T W\bigr],\\
&\text{(C):}\quad \sum_i \operatorname{tr}\!\left[\Delta_i^{\star\T}\Delta_i^\star \begin{bmatrix}(Q\T X)_{2i-1}\\(Q\T X)_{2i}\end{bmatrix}\begin{bmatrix}(W\T X)_{2i-1}\\(W\T X)_{2i}\end{bmatrix}\T\right]
= \operatorname{tr}\bigl[\Delta_Q^{\star\T}\Delta^\star_Q Q\T XX\T W\bigr]\\
&\phantom{\text{(C):}\quad \sum_i \operatorname{tr}\!\left[\Delta_i^{\star\T}\Delta_i^\star \begin{bmatrix}(Q\T X)_{2i-1}\\(Q\T X)_{2i}\end{bmatrix}\begin{bmatrix}(W\T X)_{2i-1}\\(W\T X)_{2i}\end{bmatrix}\T\right]}
= \operatorname{tr}\bigl[(XX\T Q\Delta_Q^{\star\T}\Delta^\star_Q)\T W\bigr],
\end{align*}
yielding the Euclidean gradient
\begin{equation}\label{eq:gradRealEuc}
\nabla_Q\, g = 2\left(XY\T Q\Delta^\star_Q + YX\T Q\Delta_Q^{\star\T} - XX\T Q\Delta_Q^{\star\T}\Delta^\star_Q\right).
\end{equation}

We now project $\nabla_Q\, g$ onto the tangent space $T_Q\mathrm{O}(m)$ to obtain the Riemannian gradient. Analogously to the complex case, since the Lie algebra of $\mathrm{O}(m)$ is the space of skew-symmetric matrices, $T_Q\mathrm{O}(m) = \{QS \mid S\T = -S\}$. Projecting $\nabla_Q\, g$ onto $T_Q\mathrm{O}(m)$ via the skew-symmetric part operator $\operatorname{rskew}(Z) = \tfrac{1}{2}(Z - Z\T)$ yields the Riemannian gradient
\begin{equation}\label{eq:gradReal}
\mathrm{grad}_Q\, g = Q\,\operatorname{rskew}\!\left(Q\T\, \nabla_Q\, g\right).
\end{equation}

As in the complex case, the Riemannian Hessian follows by evaluating $\delta(\nabla_Q\, g)$ and projecting onto $T_Q\mathrm{O}(m)$, via \cite[Equation (7.39)]{boumal}:
\begin{equation}\label{eq:hessReal}
\mathrm{Hess}_Q\, g[W] = Q\,\operatorname{rskew}\!\left(Q\T H_Q\, g[W] - Q\T W\operatorname{sym}(Q\T \nabla_Q\, g)\right),
\end{equation}
where $\operatorname{sym}(Z) = \tfrac{1}{2}(Z + Z\T)$, and $H_Q\, g[W] = \delta(\nabla_Q\, g)$ is the Euclidean Hessian. Applying the product rule to~\eqref{eq:gradRealEuc},
\begin{equation}\label{eq:deltaGradReal}
\begin{aligned}
H_Q\, g[W] = 2\Bigl(&XY\T W\Delta^\star_Q + XY\T Q\,\delta(\Delta^\star_Q)
+ YX\T W\Delta_Q^{\star\T} + YX\T Q\,\delta(\Delta^\star_Q)\T\\
&- XX\T W\Delta_Q^{\star\T}\Delta^\star_Q
- XX\T Q\bigl(\delta(\Delta^\star_Q)\T\Delta^\star_Q + \Delta_Q^{\star\T}\delta(\Delta^\star_Q)\bigr)\Bigr),
\end{aligned}
\end{equation}
where $\delta(\Delta^\star_Q)$ is the block diagonal obtained by differentiating each block of $\Delta^\star_Q$ entrywise according to its active form,
\[
\delta(\Delta_i^\star) = \begin{bmatrix} \delta(a_i^\star) & \delta(b_i^\star) \\ -\delta(b_i^\star) & \delta(a_i^\star) \end{bmatrix}
\quad\text{or}\quad
\delta(\Delta_i^\star) = \begin{bmatrix} \delta(\lambda_{2i-1}^\star) & 0 \\ 0 & \delta(\lambda_{2i}^\star) \end{bmatrix}
\]
(and $\delta(\lambda_m^\star)$ for the singleton, if present). By the quotient rule applied to \eqref{eq:optRot} and \eqref{eq:optDiag},
\[
\delta(a_i^\star) = \frac{\delta(\alpha_i) - a_i^\star\bigl(\delta(\phi_{2i-1})+\delta(\phi_{2i})\bigr)}{\phi_{2i-1}+\phi_{2i}},
\qquad
\delta(b_i^\star) = \frac{\delta(\beta_i) - b_i^\star\bigl(\delta(\phi_{2i-1})+\delta(\phi_{2i})\bigr)}{\phi_{2i-1}+\phi_{2i}},
\]
\[
\delta(\lambda_j^\star) = \frac{\delta(\gamma_j) - \lambda_j^\star\,\delta(\phi_j)}{\phi_j}.
\]

\section{Numerical Experiments: Real Case}\label{sec:realNum}
We proceed similarly to Section~\ref{sec:complexNum}, but we do not implement a subsection akin to Subsection~\ref{ssec:complexPerform}, since performance differences on Gaussian and uniform distributions were negligible, and performance trends (like residual spread across starts and data scaling) are still apparent in Subsection~\ref{ssec:realKnown}, which provides an additional benefit of constructing a known global minimum. Our code implements Manopt's \texttt{stiefelfactory}---allowing us to work over $\mathrm{O}(m)$. Note that \texttt{stiefelfactory}'s default retraction is QR-decomposition based and only first order, so we set \texttt{M.retr=M.retr\_polar} (the polar retraction, which is second order). Further note that $\mathrm{O}(m)$ is disconnected, so a solver starting at $Q\in \mathrm{SO}(m)$ never reaches the other connected component $\mathrm{O}(m) \setminus \mathrm{SO}(m)$. However, every objective value attainable on one component of $\mathrm{O}(m)$ is attained on the other; if $Q=Q'\operatorname{diag}(-1,1,\ldots,1)$ with $Q'\in \mathrm{O}(m) \setminus \mathrm{SO}(m)$, then by Lemma~\ref{lem:optDelt}, this sign change leaves the objective value unchanged, since we may adjust the sign of $b_1^\star$, if $\Psi$ is the active branch.

\subsection{Known Global Optimum Recovery}\label{ssec:realKnown}~\\[8pt]
\indent As in Subsection~\ref{ssec:complexKnown}, we generate a real, uniformly distributed random $X$ via MATLAB's \texttt{rand(m,n)}, and we generate $A^\star = Q^\star \Delta^\star (Q^\star)\T$ by drawing $Q^\star$ from the Haar measure on $\mathrm{O}(m)$ and drawing the quasidiagonal $\Delta^\star$ (independent of $Q^\star$) blockwise: each $2\times2$ block is, with probability $0.5$ either in $\Psi$ or $\Omega$---with all of $a,b,\lambda_1,\lambda_2$ drawn uniformly from $[-1,1]$ (we do not test odd $m$ in our examples, but our code allows for this). We then set $Y:=A^\star X$ and observe whether $A^\star$, or an equally optimal matrix, is recovered. We obtain the following results (residual and gradient norm plots are again available via our \href{https://github.com/kbierly/NPP}{GitHub repository}, as is the output of \texttt{checkgradient} and \texttt{checkhessian}):
\begin{table}[H]
\centering
\caption{Global optimum recovery for the real NPP with uniform $X$, over two random starts. \textit{Rel.\ $A^\star$} is the relative difference ${\|A_{\mathrm{rec}} - A^\star\|_F}$ as a percentage of $\|A^\star\|_F$}
\resizebox{\textwidth}{!}{%
\begin{tabular}{cc|cc|cc|cc|cc}
\hline
& & \multicolumn{2}{c|}{Init Res} & \multicolumn{2}{c|}{Final Res} & \multicolumn{2}{c|}{Time (s)} & \multicolumn{2}{c}{Rel.\ $A^\star$ (\%)} \\
$m$ & $n$ & Start 1 & Start 2 & Start 1 & Start 2 & Start 1 & Start 2 & Start 1 & Start 2 \\
\hline
10 &  2 &    0.6 &    0.2 & 9.9e-14 & 3.7e-02 &  0.23 & 0.04 & 21407.3 &  334.0 \\
10 &  5 &    3.2 &    1.8 & 5.0e-12 & 7.3e-12 &  0.38 & 0.22 &   83.4 &  127.4 \\
10 &  8 &    3.5 &    4.5 & 5.7e-02 & 1.8e-02 &  0.29 & 0.11 &   89.4 &   80.9 \\
10 & 10 &    8.4 &   17.6 & 4.8e-13 & 1.7e-15 &  0.13 & 0.09 &    0.0 &    0.0 \\
20 &  4 &    2.7 &    2.1 & 1.5e-18 & 2.7e-21 &  0.22 & 0.11 & 1244.4 &  147.9 \\
20 & 10 &    5.9 &    4.8 & 2.5e-09 & 1.2e-08 & 11.24 & 21.15 &   81.3 &   54.4 \\
20 & 16 &   21.7 &   28.5 & 5.5e-03 & 5.2e-02 &  0.99 & 0.61 &    6.9 &   46.9 \\
20 & 20 &   18.8 &   24.3 & 4.7e-12 & 8.0e-12 &  1.83 & 2.19 &    0.0 &    0.0 \\
50 & 10 &   48.3 &   42.7 & 2.0e-19 & 2.2e-16 &  0.38 & 0.31 &  130.1 &  136.4 \\
50 & 25 &  117.1 &  121.7 & 9.1e-10 & 3.6e-09 & 245.96 & 482.47 &   91.3 &  103.3 \\
50 & 40 &  138.0 &  121.8 & 1.0e-01 & 2.6e-02 & 71.87 & 94.35 &    9.8 &    5.2 \\
50 & 50 &  194.8 &  170.1 & 1.7e-01 & 7.7e-01 & 38.52 & 28.88 &    8.8 &   13.9 \\
\hline
\end{tabular}%
}
\label{tab:real_known_min}
\end{table}

In contrast to the complex case, where every recovered $A_{\mathrm{rec}}$ produced a residual of less than $10^{-10}$, Table~\ref{tab:real_known_min} shows that only $45.8\%$ of the trials we ran achieved this threshold, and $62.5\%$ achieved a residual less than $10^{-5}$. Hence, our solver quite frequently terminated at a suboptimal $A_{\mathrm{rec}}$. We observe a few trends that carry over from the complex case:
\begin{itemize}
\item For $n$ near $m/2$, solve time remains high.
\item For the trials that obtained a residual of less than $10^{-10}$, when $n$ is small Rel.\ $A^\star$ is high (reflecting a large degree of freedom), whereas the instances with large $n$ are more proximal to $A^\star$.
\item Solve time unsurprisingly increases as matrix size increases.
\end{itemize}

One potential reason why suboptimal results were often recovered in the real case---which is indeed one of the more interesting findings of this section---is a perhaps increased nonconcavity on the landscape of the objective function \eqref{eq:gQ}, i.e., a greater frequency in suboptimal stationary points that appear far from the global optimum. In the complex case, $D^\star_U$ is diagonal, so the objective~\eqref{eq:fU} is invariant under permutations of the coordinates. In the real case, this invariance is lost due to the block structure of $\Delta^\star_Q$, as we demonstrate in the following example.

\begin{example}\label{ex:realIssues}

Let $X = I_4$, $Q^\star=I_4$, 
\[
\Delta^\star =
\begin{bmatrix}
2 & 1 & 0 & 0\\
-1 & 2 & 0 & 0\\
0 & 0 & 0 & 3\\
0 & 0 & -3 & 0
\end{bmatrix}, \quad A^\star = Q^\star \Delta^\star (Q^\star)\T= \Delta^\star,
\]
and $Y=A^\star X=\Delta^\star$. Then let
\[P=\begin{bmatrix}
1 & 0 & 0 & 0\\
0 & 0 & 1 & 0 \\
0 & 1 & 0 & 0 \\
0 & 0 & 0 & 1
\end{bmatrix}\in \mathrm{O}(4),\]
such that 
\[
P\T A^\star P =
\begin{bmatrix}
2 & 0 & 1 & 0\\
0 & 0 & 0 & 3\\
-1 & 0 & 2 & 0\\
0 & -3 & 0 & 0
\end{bmatrix},
\]
and by Lemma~\ref{lem:optDelt}, we have
\[
\Delta^\star_{P} =
\begin{bmatrix}
2 & 0 & 0 & 0\\
0 & 0 & 0 & 0\\
0 & 0 & 2 & 0\\
0 & 0 & 0 & 0
\end{bmatrix}.
\]
Furthermore,
\begin{equation}
    \begin{aligned}
\nabla_{P}\, g &= 2\left(XY\T P\Delta^\star_{P} + YX\T P\Delta^{\star\T}_{P} - XX\T P\Delta_{P}^{\star\T}\Delta_{P}^\star\right)\\
&=2\left((\Delta^\star)\T P\Delta^\star_{P} + \Delta^\star P\Delta^{\star\T}_{P} -  P\Delta_{P}^{\star\T}\Delta_{P}^\star\right)\\
&= \begin{bmatrix}
    8 & 0 & 0 & 0\\
    0 & 0 & 8 & 0\\
    0 & 0 & 0 & 0 \\
    0 & 0 & 0 & 0
\end{bmatrix},
    \end{aligned}
\end{equation}
then 
\begin{equation}
    \begin{aligned}
\mathrm{grad}_P\, g&=  \frac{1}{2} P (P\T\, \nabla_P\, g - (P\T\, \nabla_P\, g)\T)\\
&= \frac{1}{2} P (0)=0.
    \end{aligned}
\end{equation}

Thus, the Riemannian gradient vanishes at $P$. To determine whether a second-order method can escape $P$, we compute $\operatorname{Hess}_P g[W]$. Let $W = PS \in T_P\mathrm{O}(4)$, where $S\T = -S$.

Since $X = I_4$, we have $\phi_i = 1$ for all $i$, so $\delta(\phi_i) = 0$, and 
\begin{align*}\delta(\lambda_i^\star) = \delta(\gamma_i) = (W\T Y P + P\T Y W)_{i,i}&=(-SP\T Y P + P\T Y PS)_{i,i}\\
&=(-S(P\T Y P + P\T Y\T P))_{i,i}\\
&=\left(-S\left(\begin{bmatrix}
4 & 0 & 0 & 0\\
0 & 0 & 0 & 0\\
0 & 0 & 4 & 0\\
0 & 0 & 0 & 0
\end{bmatrix}\right)\right)_{i,i},\end{align*}
and since $S_{i,i} = 0$, every such sum vanishes. Both blocks of $\Delta^\star_P$ are in the $\Omega$ branch, so $\delta(\Delta^\star_P) = \operatorname{diag}\bigl(\delta(\lambda_1^\star),\ldots,\delta(\lambda_4^\star)\bigr) = 0$.

Substituting $X = I_4$ and $\delta(\Delta^\star_P) = 0$ into \eqref{eq:deltaGradReal} and using $(\Delta^\star_P)^2=2 \Delta^\star_P$,
\[
H_P\, g[W] = 2\left((Y + Y\T) W \Delta^\star_P - 2W \Delta^\star_P\right).
\]
Furthermore,
\[
P\T\, \nabla_P\, g = \begin{bmatrix}
8 & 0 & 0 & 0\\
0 & 0 & 0 & 0\\
0 & 0 & 8 & 0\\
0 & 0 & 0 & 0
\end{bmatrix} = 4\Delta^\star_P,
\]
which is symmetric, so $\operatorname{sym}(P\T \nabla_P\, g) = 4\Delta^\star_P$. Since $P\T(Y+Y\T)P = 2\Delta^\star_P$ and $P\T W = S$, equation \eqref{eq:hessReal} gives
\begin{align*}
\operatorname{Hess}_P\, g[W]
&= P\,\operatorname{rskew}\!\left(P\T H_P\, g[W] - P\T W\operatorname{sym}(P\T \nabla_P\, g)\right)\\
&= P\operatorname{rskew}\!\left(4\Delta^\star_P S \Delta^\star_P - 8S\Delta^\star_P\right).
\end{align*}
A direct computation yields (where $s_{i,j}$ for $i,j\in\{1,2,3,4\}$ denotes an entry in $S$):
\[
\operatorname{Hess}_P\, g[W] = P\operatorname{rskew}\!\left(4\Delta^\star_P S \Delta^\star_P - 8S\Delta^\star_P\right)
= -8\begin{bmatrix}
0 & s_{1,2} & 0 & s_{1,4}\\
0 & -s_{2,3} & 0 & s_{3,4}\\
-s_{1,2} & 0 & s_{2,3} & 0\\
-s_{1,4} & 0 & -s_{3,4} & 0
\end{bmatrix}.
\]
The entries $s_{1,3}$ and $s_{2,4}$ are annihilated, and the remaining four are scaled by $-8$. Hence $\operatorname{Hess}_P\, g$ has eigenvalues $-8$ (multiplicity four) and $0$ (multiplicity two).

Thus, at $P$ the Riemannian gradient vanishes and the Riemannian Hessian is negative
semidefinite, yet
\[\|P \Delta_P^\star P\T X - Y \|_F^2=20> \|A^\star X-Y\|_F^2=0.\]
Hence $P$ is a suboptimal critical point at which our second-order model would predict no
ascent, so a random start initialized nearby may terminate at $P$. The permutation
invariance of the complex objective precludes this particular phenomenon.
\end{example}

If $m$ is even, this lack of permutation invariance admits $(m-1)!!$ coordinate pairings that may introduce suboptimal stationary points as in Example~\ref{ex:realIssues}; if $m$ is odd, then there are $m!!$ such pairings \cite{callan}. We do not attempt to quantify how many of these points actually give rise to suboptimal critical points where our solver would terminate. Running more random starts or introducing an improved initialization strategy may increase the likelihood of avoiding such suboptimal stationary points, but we leave an analysis of the efficacy of these methods to future work. 
\subsection{Comparison with Guglielmi and Scalone's Algorithm over the Real CNP}\label{ssec:realComp}~\\[8pt]
\indent In this subsection we again compare our method with Guglielmi and Scalone's, but this time we do so for the Real CNP. As previously discussed, Guglielmi and Scalone's method can be adapted to work over both $\M_m(\C)$ and $\M_m(\R)$, while Ruhe's is restricted to the former. 

As in Subsection~\ref{ssec:complexComp}, we can update our objective function and its Riemannian gradient and Hessian for the Real CNP. Again, since $X=I_m$, $\phi_i = 1$ for all $i=1,\ldots, m$. Set 
\[C:=Q\T Y Q,\]we thus have:
\[
\gamma_i = C_{i,i}, \quad \alpha_i= C_{2i-1,2i-1} + C_{2i,2i},
\]
\[\beta_i= C_{2i-1,2i} - C_{2i,2i-1},
\]
such that 
\begin{equation}
       (a_i^\star,  b_i^\star) =
          \left(\dfrac{C_{2i-1,2i-1} + C_{2i,2i}}{2},\
                \dfrac{C_{2i-1,2i} - C_{2i,2i-1}}{2}\right),
  \end{equation}
  and for $j \in \{2i-1,\, 2i\}$ (or when $m$ is odd, and $j=m$),
  \begin{equation}
      \lambda_j^\star =\gamma_j=C_{j,j}.
  \end{equation}

We can thus update $\Delta^\star_Q$ by substituting $\lambda_j^\star$, $a_i^\star$, and $b_i^\star$ into \eqref{eq:optBlock}.
The objective~\eqref{eq:gQ} reduces to
\begin{equation}
\begin{aligned}
   g(Q) = \|\Delta^\star_Q\|_F^2
   = \sum_{i=1}^{\lfloor m/2\rfloor}
      \max\Bigg\{\,
        &\dfrac{(C_{2i-1,2i-1}+C_{2i,2i})^2 + (C_{2i-1,2i}-C_{2i,2i-1})^2}{2},\\
        &\;\; C_{2i-1,2i-1}^2 + C_{2i,2i}^2
      \,\Bigg\} + C_{m,m}^2,
\end{aligned}
\end{equation}
where the final term is present only when $m$ is odd. Substituting $X = I_m$ into the Euclidean gradient~\eqref{eq:gradRealEuc} yields
\begin{equation}\label{eq:gradRealCNP}
   \nabla_Q \,g = 2\bigl(Y\T Q\Delta^\star_Q + YQ(\Delta^\star_Q)\T - Q\Delta_Q^{\star\T}\Delta^\star_Q\bigr),
\end{equation}
and into the Euclidean Hessian~\eqref{eq:deltaGradReal} yields
\begin{align*}
   H_Q\, g[W] = 2\bigl(
   &Y\T W\Delta^\star_Q + Y\T Q\,\delta(\Delta^\star_Q) \\
   &+ YW(\Delta^\star_Q)\T + YQ\,\delta(\Delta^\star_Q)\T \\
   &- W\,\Delta_Q^{\star\T}\Delta^\star_Q - Q\,\delta(\Delta_Q^{\star\T}\Delta^\star_Q)\bigr),
\end{align*}
where, since $\delta(\phi_j) = 0$, we obtain
\[
\delta(a_i^\star) = \frac{\delta(C)_{2i-1,2i-1} + \delta(C)_{2i,2i}}{2},
\qquad
\delta(b_i^\star) = \frac{\delta(C)_{2i-1,2i} - \delta(C)_{2i,2i-1}}{2},
\]
\[
\delta(\lambda_j^\star) = \delta(C)_{j,j},
\]
from which $\delta(\Delta^\star_Q)$ is assembled by substituting into the active branch of \eqref{eq:optBlock}. Lastly, we evaluate $\delta(C) = W\T Y Q + Q\T Y W$.

We test Guglielmi and Scalone's algorithm against ours (with a random start on $\mathrm{O}(m)$) for several familiar examples from Subsection \ref{ssec:complexComp} adapted to be real---namely $B_7$, $J_7$, $T_7$, $F_{12}$, and \cite[Example 1]{scalone}. We also test these two methods on random, uniformly distributed $Y\in \M_m(\R)$ and ${m\in \{10,20,50,100\}}$. We record the results here:

\begin{table}[H]
\centering
\caption{Performance on the benchmark matrices of \cite{ruhe} and \cite{scalone}---adapted for the reals.
Line styles as in Figure~\ref{fig:realLit}. Columns as in Table~\ref{tab:literature}}
\label{tab:realliterature}
\begin{tabular}{c|c|c|c|c|c}
\hline
\rule[-1.2ex]{0pt}{4.2ex}Matrix & Method & Final Residual & $\|AA\T-A\T A\|_F$ & Iterations & Time (s) \\
\hline
\multirow{2}{*}{Ex.\ 1}
  & \lsnpp\ NPP  & 0.709 & 1.0e-15 & 3  & 0.019 \\
  & \lsgs\ G\&S  & 0.709 & 1.6e-07 & 9\,/\,862 & 0.031 \\
\hline
\multirow{2}{*}{$B_7$}
  & \lsnpp\ NPP  & 20.794 & 5.2e-15 & 13 & 0.046 \\
  & \lsgs\ G\&S  & 20.073 & 1.4e-06 & 20\,/\,1398 & 0.068 \\
\hline
\multirow{2}{*}{$J_7$}
  & \lsnpp\ NPP  & 1.041 & 1.3e-15 & 20 & 0.035 \\
  & \lsgs\ G\&S  & 0.857 & 2.7e-07 & 21\,/\,1216 & 0.078 \\
\hline
\multirow{2}{*}{$T_7$}
  & \lsnpp\ NPP  & 7.211 & 3.1e-15 & 13 & 0.041 \\
  & \lsgs\ G\&S  & 6.918 & 3.8e-07 & 12\,/\,6448 & 0.337 \\
\hline
\multirow{2}{*}{$F_{12}$}
  & \lsnpp\ NPP  & 509.922 & 3.1e-13 & 31 & 0.167 \\
  & \lsgs\ G\&S  & 582.681 & 3.5e-05 & 29\,/\,6548 & 0.779 \\
\hline
\end{tabular}
\end{table}

From Table~\ref{tab:realliterature}, we observe a few noticeable differences from the complex case. In the complex case, we found strong agreement across all three methods on every instance, but in the real case, the final residuals differ fairly significantly. Guglielmi and Scalone's method achieved a lower final residual than our approach on three instances, and ours found a more optimal real normal matrix on the $F_{12}$ instance. In each instance where a difference in final residual occurred, these residuals disagree by greater than 3\% but less than 20\% of the larger residual. These differences again demonstrate the difficulty posed by the increased nonconcavity of our objective for the Real CNP. Unsurprisingly, under the current stopping criteria for both methods, our results continue to produce normal matrices that satisfy $\|AA\T - A\T A\|_F\approx 0$ more closely than Guglielmi and Scalone's. We also find that our method terminated in less wall-clock time than Guglielmi and Scalone's on each instance.

Lastly, both methods recover the same closest real normal matrix for \cite[Example 1]{scalone}. However, this result again differs from what Guglielmi and Scalone reported, as in the complex case. Guglielmi and Scalone recover the matrix $X_r$ via their own algorithm, and they report that $\|X_r - Y\|_F^2 \approx 1.1381$, which is greater than our final recovered residual of $\|A_{\mathrm{rec}} - Y\|_F^2 \approx 0.709$. Specifically, for
\[Y=\begin{bmatrix}
    0.3 & -1 & 0\\
    1 & 0.5 & -0.3\\
    0 & -1 & 0 
\end{bmatrix},
\]
both methods in our implementation obtain (after rounding):
\[
A_{\mathrm{rec}}=\begin{bmatrix}0.5055 &  -0.8704  &  0.4862\\
    0.9895 &   0.4729  & -0.0064\\
   -0.1220  & -0.4707  & -0.1784
\end{bmatrix}.
\]

\begin{figure}[H]
    \centering
    \begin{minipage}[t]{0.45\textwidth}\vspace*{0pt}
    \includegraphics[width=\textwidth]{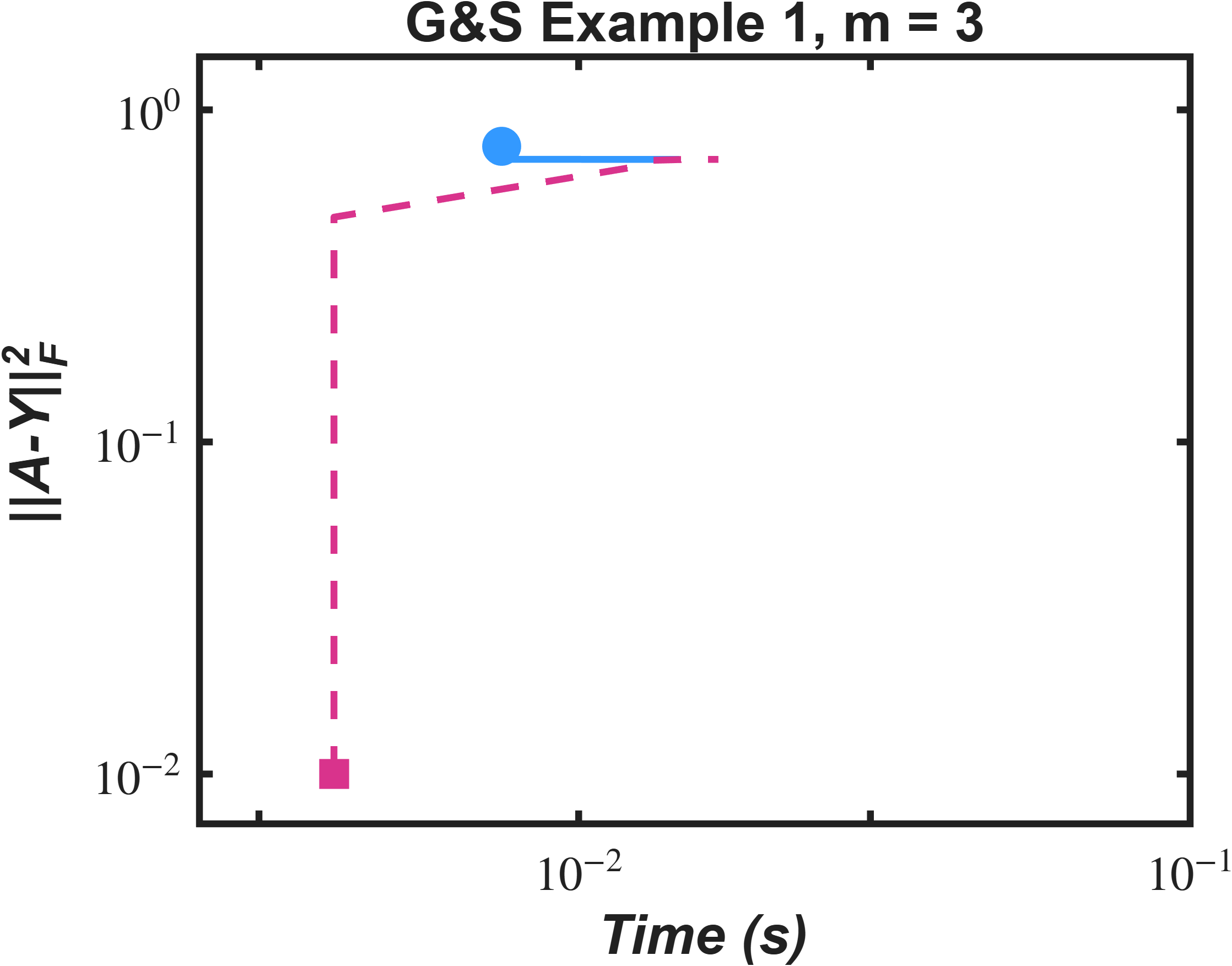}
\end{minipage}\hfill
\begin{minipage}[t]{0.45\textwidth}\vspace*{0pt}
    \includegraphics[width=\textwidth]{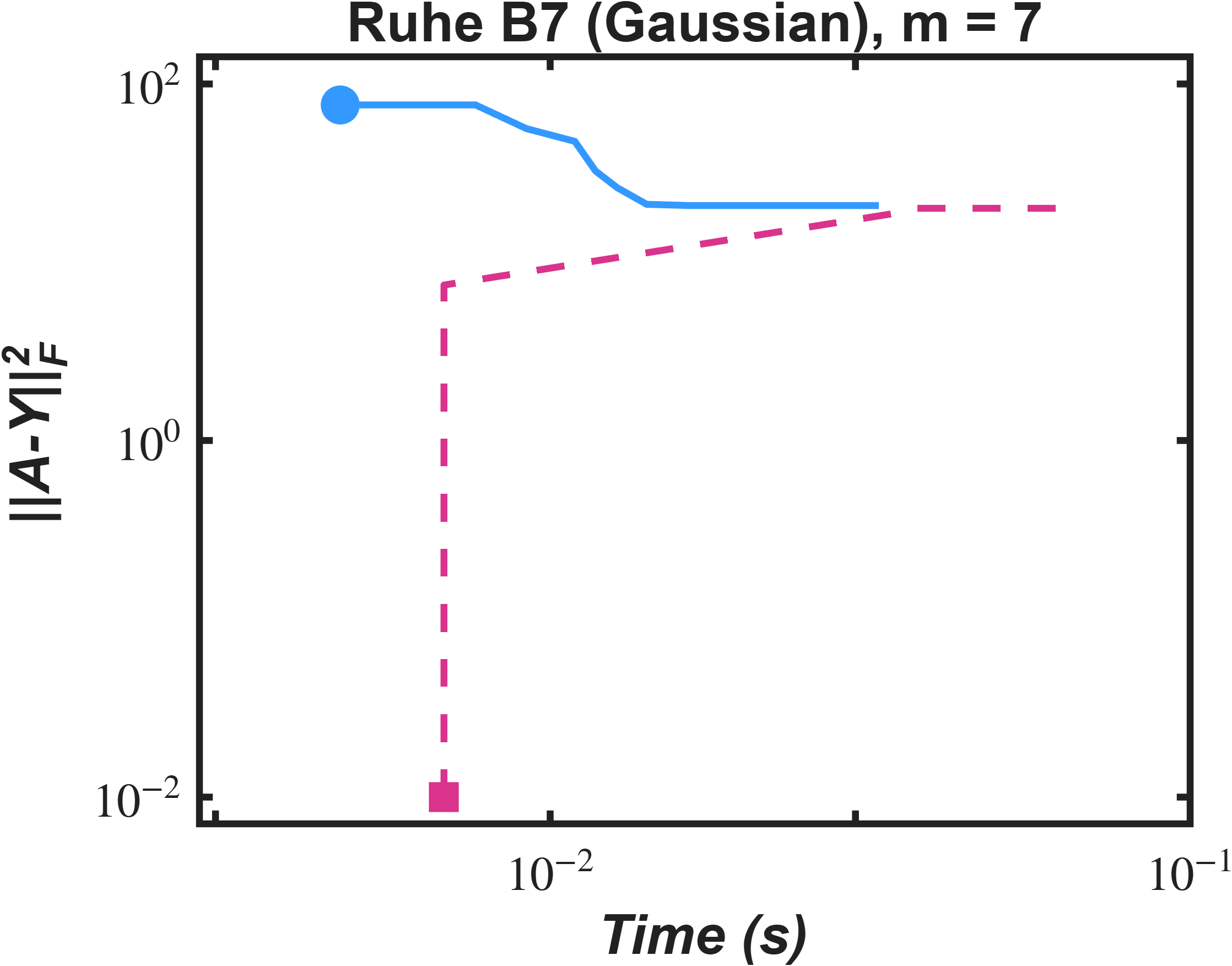}
\end{minipage}\hfill
\begin{minipage}[t]{0.10\textwidth}\vspace*{6.5pt}
    \includegraphics[width=\textwidth]{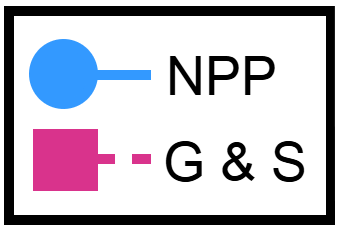}
\end{minipage}\\[5pt]
\begin{minipage}[t]{0.45\textwidth}\vspace*{0pt}
    \includegraphics[width=\textwidth]{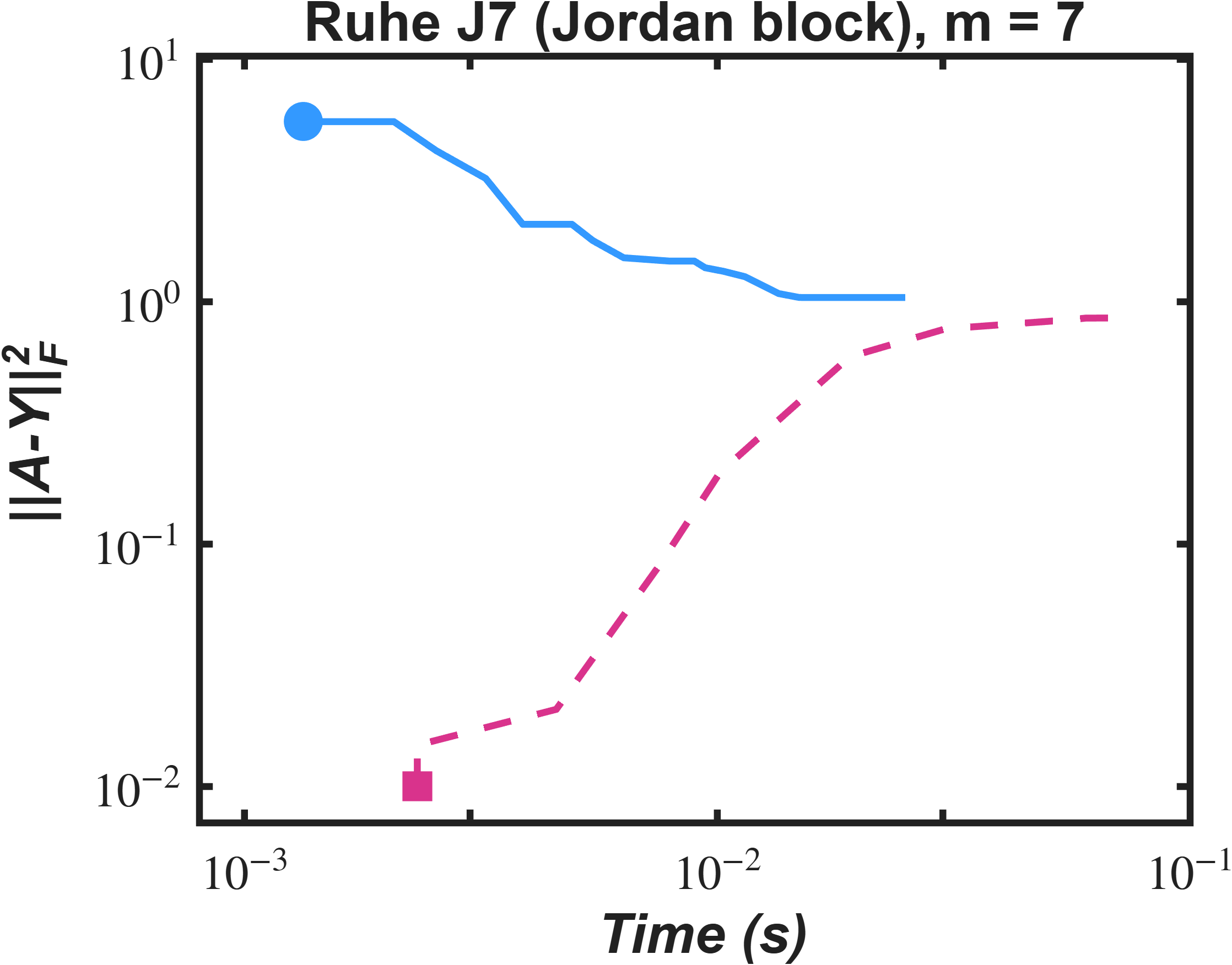}
\end{minipage}\hfill
\begin{minipage}[t]{0.45\textwidth}\vspace*{0pt}
    \includegraphics[width=\textwidth]{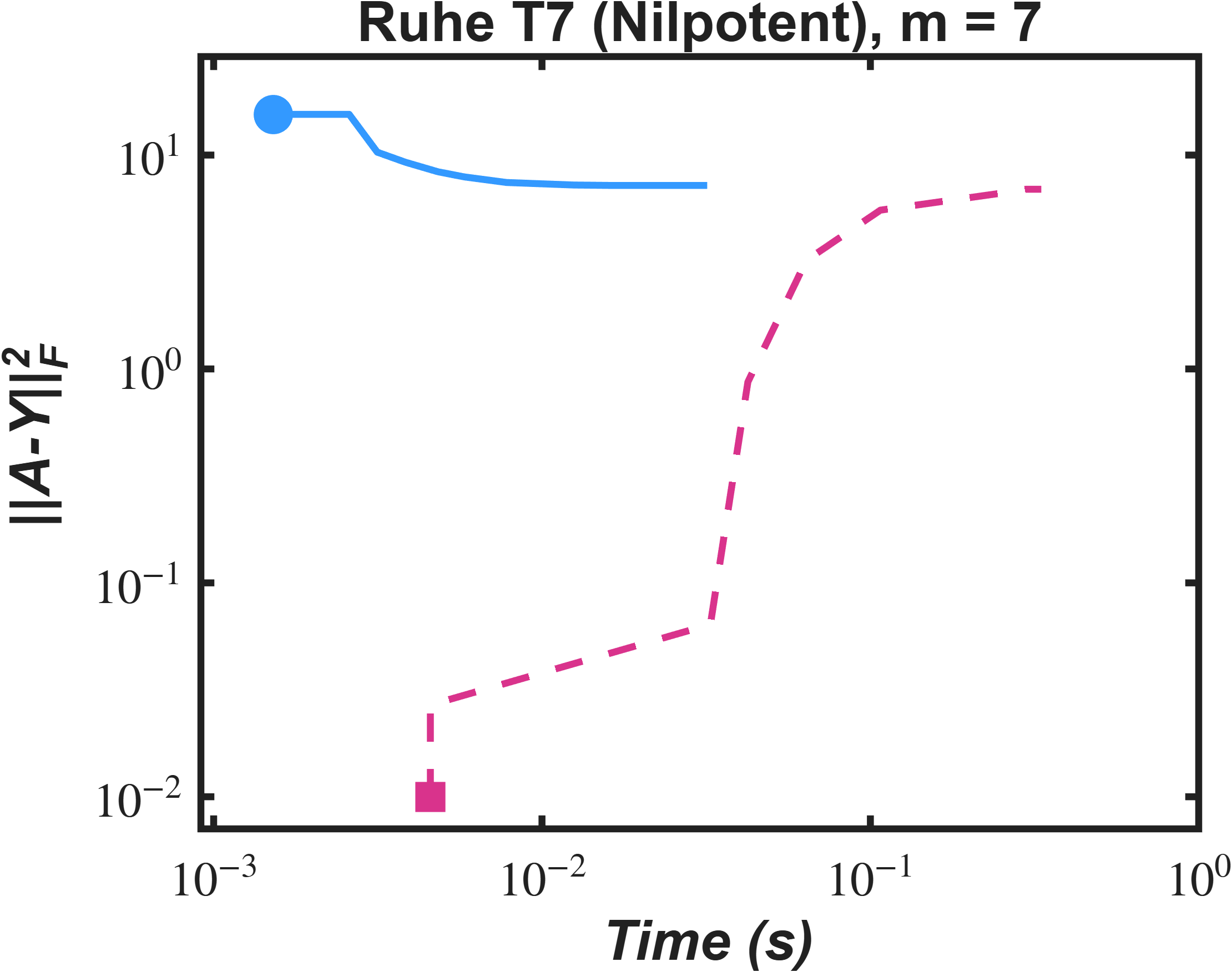}
\end{minipage}\hfill
\begin{minipage}[t]{0.10\textwidth}\vspace*{0pt}
    \phantom{\includegraphics[width=\textwidth]{Fig5c.png}}
\end{minipage}\\[5pt]
\begin{minipage}[t]{0.45\textwidth}\vspace*{0pt}
    \includegraphics[width=\textwidth]{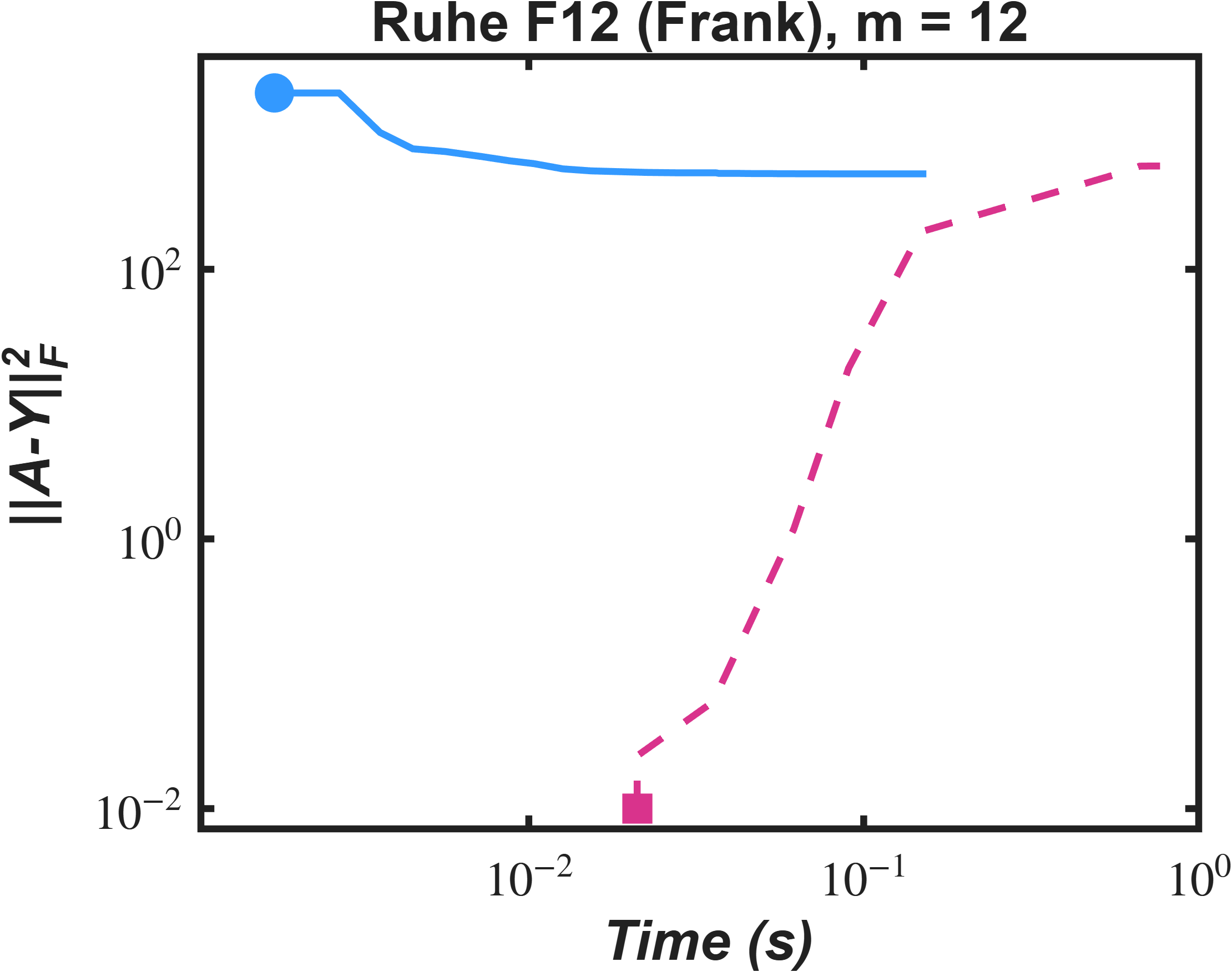}
\end{minipage}
    \caption{Residual $\|A-Y\|_F^2$ (note $A$ is not strictly in $\mathcal{N}_{\R}$) vs.~time to solve, for the matrices in Table~\ref{tab:realliterature}}\label{fig:realLit}
\end{figure}
The plots in Figure~\ref{fig:realLit} do not meaningfully exhibit any new information that Table~\ref{tab:realliterature} does not already provide. The curve associated with our method in these plots, as in Figure~\ref{fig:realScale}, is strictly in $\mathcal{N}_{\R}$ while Guglielmi and Scalone's curve is in $\M_{m}(\R)$. 
\newpage

\begin{table}[H]
\centering
\caption{Scaling on random $Y$ with uniformly distributed real entries}
\label{tab:realscaling}
\begin{tabular}{c|c|c|c|c|c}
\hline
\rule[-1.2ex]{0pt}{4.2ex}$m$ & Method & Final Residual & $\|AA\T-A\T A\|_F$ & Iterations & Time (s) \\
\hline
\multirow{2}{*}{$10$}
  & \lsnpp\ NPP  & 2.433 & 3.2e-15 & 18 & 0.055 \\
  & \lsgs\ G\&S  & 2.433 & 3.1e-07 & 12\,/\,1610 & 0.134 \\
\hline
\multirow{2}{*}{$20$}
  & \lsnpp\ NPP  & 8.473 & 2.1e-14 & 32 & 0.241 \\
  & \lsgs\ G\&S  & 7.082 & 2.3e-06 & 30\,/\,5564 & 1.467 \\
\hline
\multirow{2}{*}{$50$}
  & \lsnpp\ NPP  & 55.806 & 1.9e-13 & 55 & 2.058 \\
  & \lsgs\ G\&S  & 50.360 & 6.3e-06 & 40\,/\,29153 & 48.668 \\
\hline
\multirow{2}{*}{$100$}
  & \lsnpp\ NPP  & 216.913 & 1.0e-12 & 57 & 9.794 \\
  & \lsgs\ G\&S  & 197.108 & 2.0e-05 & 35\,/\,32015 & 234.099 \\
\hline
\end{tabular}
\end{table}

We observe similar trends in Table~\ref{tab:realscaling} to Table~\ref{tab:realliterature}, such as the noticeable difference in final residuals between methods for the $m\in \{20,50,100\}$ cases. In each such instance, residuals again differ by less than 20\% of the larger residual. Our method's advantage in our normality measure $\|AA\T -A\T A\|_F$ and in speed is again apparent. We do not test larger matrices, because the effect of scaling matrix sizes is evident from these four matrix sizes alone, and Guglielmi and Scalone's algorithm can take a considerable amount of time to terminate as matrix sizes grow much larger.

While Guglielmi and Scalone's method obtained lower residuals than our approach on $6/9$ of the tested instances on the Real CNP, it is inconclusive to say whether their method is generally superior to ours at determining closer real normal matrices. Perhaps an improved initialization strategy for our method would improve final residuals; on a highly nonconcave landscape, a random start may be detrimental, whereas a start which is more proximal to $Y$ (such as Guglielmi and Scalone's, albeit their algorithm functions completely differently) may be beneficial. Furthermore, the amount of time saved via our method would allow one to run potentially multiple starts (from different initial $Q_0$) before Guglielmi and Scalone's method terminates; this strategy would increase the likelihood of finding a more optimal real normal matrix. 

\begin{figure}[H]
    \centering
    \begin{minipage}[t]{0.45\textwidth}\vspace*{0pt}
    \includegraphics[width=\textwidth]{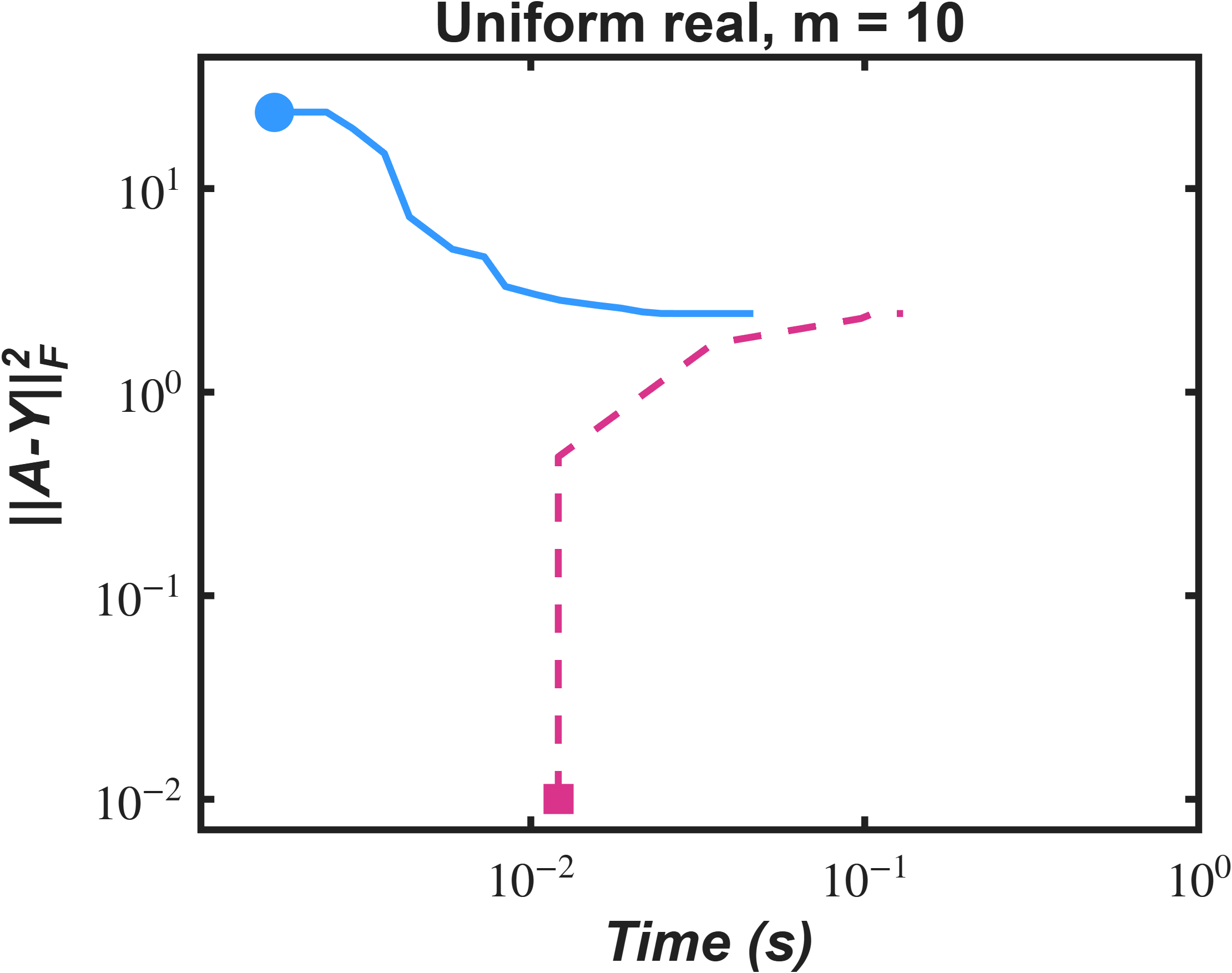}
\end{minipage}\hfill
\begin{minipage}[t]{0.45\textwidth}\vspace*{0pt}
    \includegraphics[width=\textwidth]{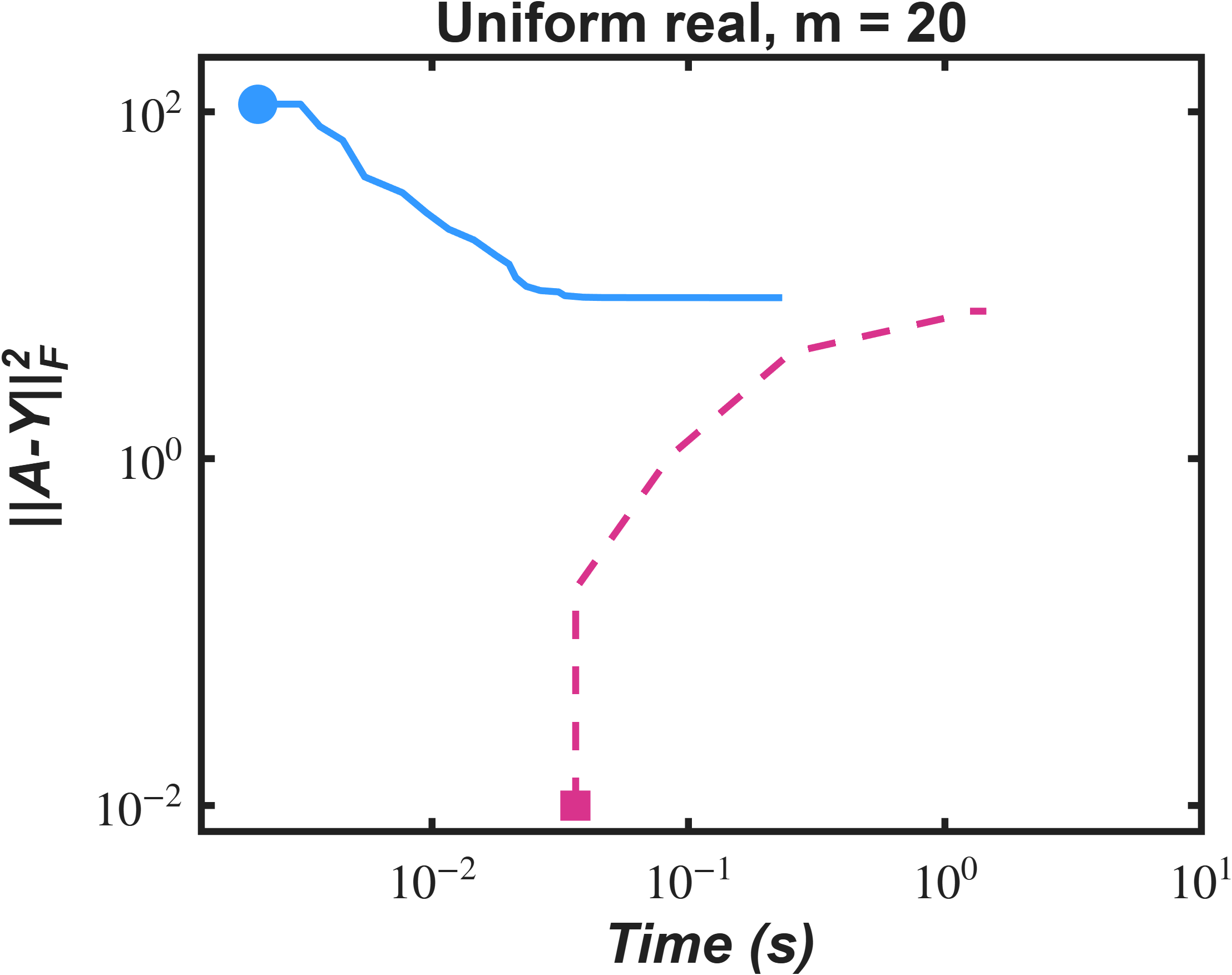}
\end{minipage}\hfill
\begin{minipage}[t]{0.10\textwidth}\vspace*{6.5pt}
    \includegraphics[width=\textwidth]{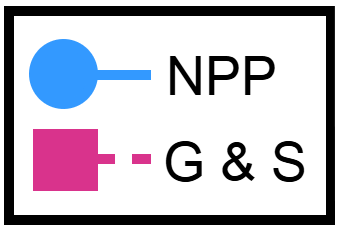}
\end{minipage}
\end{figure}
\newpage
\begin{figure}
\centering
\begin{minipage}[t]{0.45\textwidth}\vspace*{0pt}
    \includegraphics[width=\textwidth]{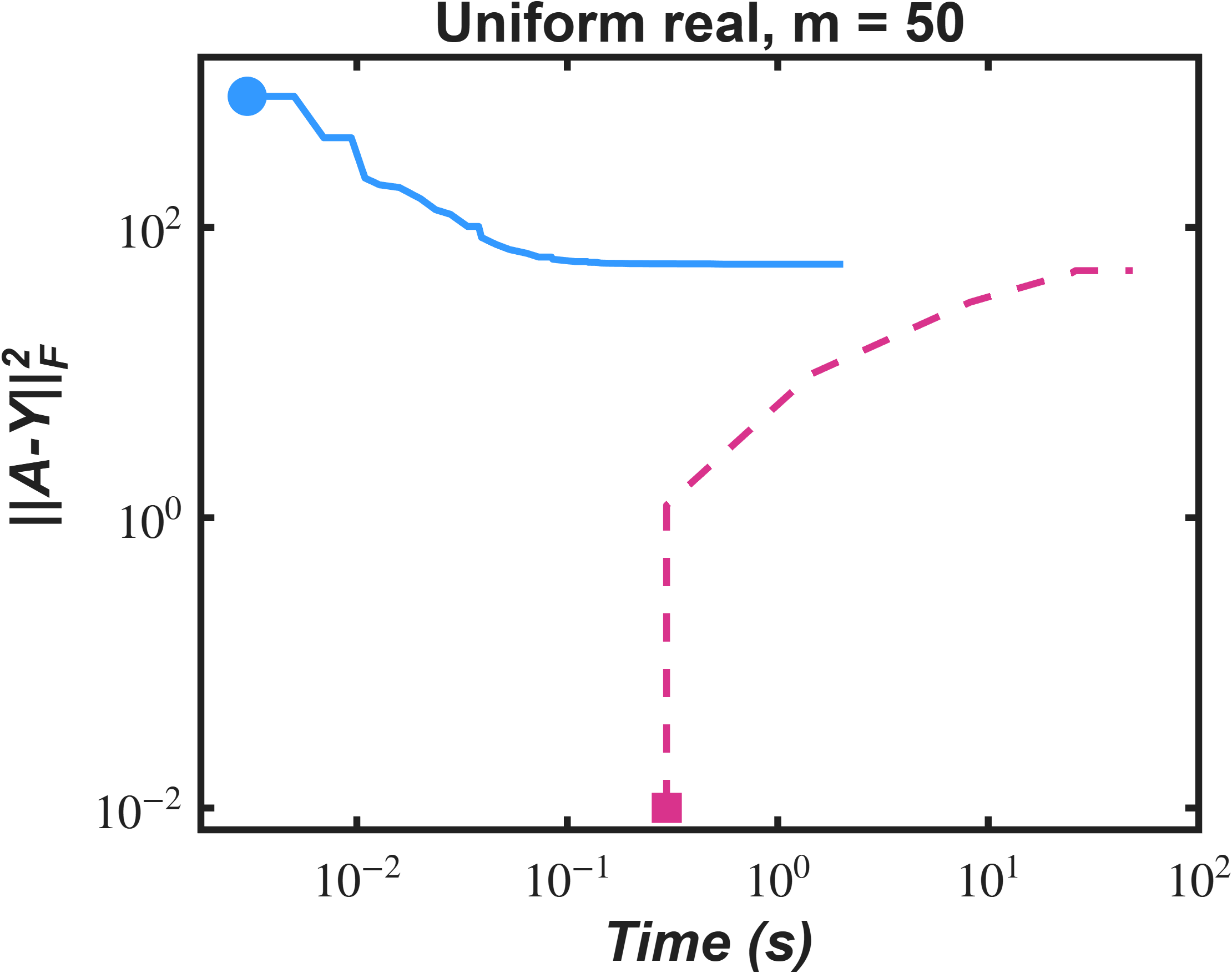}
\end{minipage}\hfill
\begin{minipage}[t]{0.45\textwidth}\vspace*{0pt}
    \includegraphics[width=\textwidth]{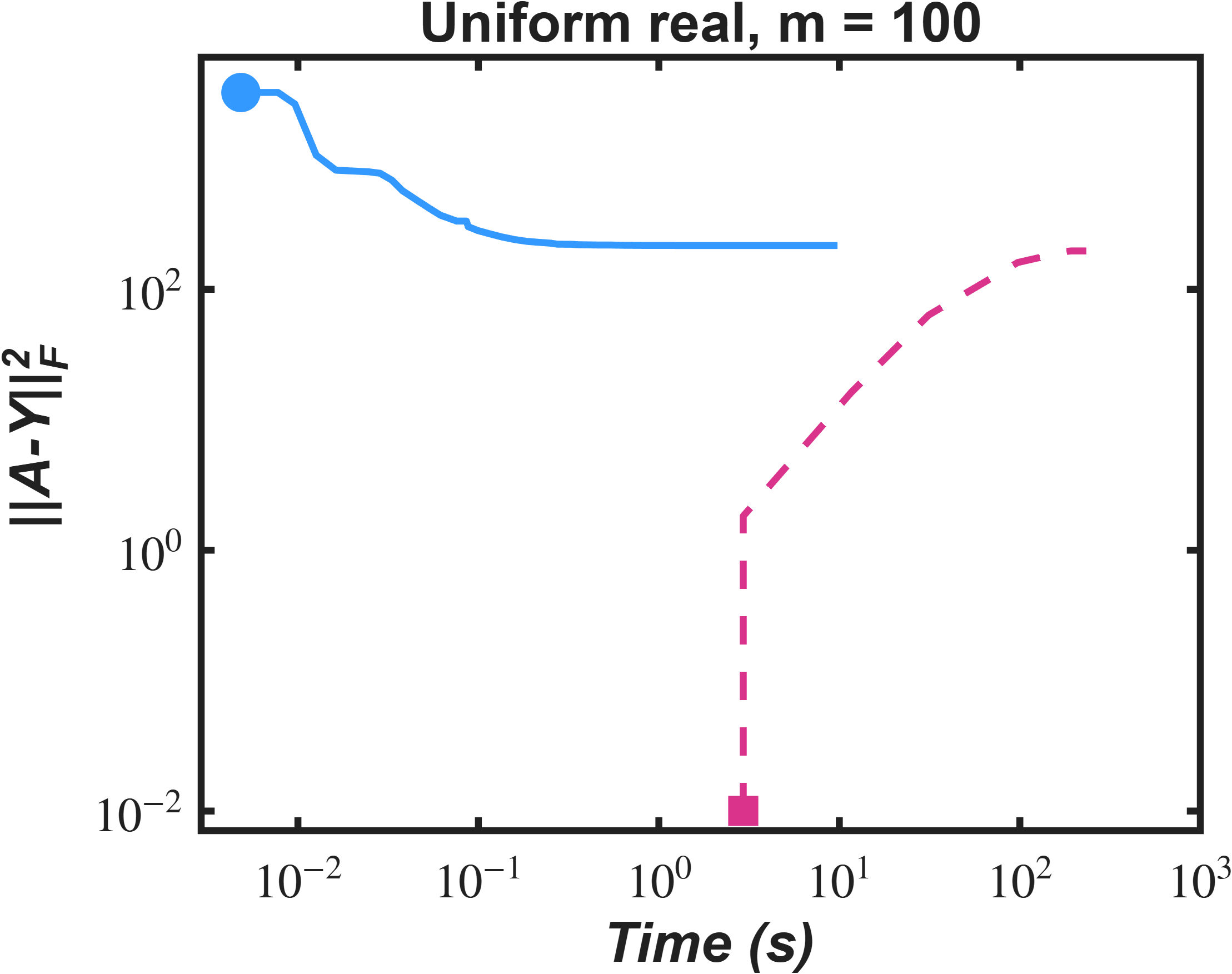}
\end{minipage}\hfill
\begin{minipage}[t]{0.10\textwidth}\vspace*{0pt}
    \phantom{\includegraphics[width=\textwidth]{Fig6c.png}}
\end{minipage}    \caption{Residual $\|A-Y\|_F^2$ (note $A$ is not strictly in $\mathcal{N}_{\R}$) vs.~time to solve, for the matrices in Table~\ref{tab:realscaling}}\label{fig:realScale}
\end{figure}

Figure~\ref{fig:realScale} visualizes our findings in Table~\ref{tab:realscaling} and again does not provide any new information worth noting.

\section{Conclusion}
In this paper, we reduce the Normal Procrustes Problem, in both its complex and real forms, to a Riemannian optimization problem over a matrix manifold---$\mathrm{U}(m)$ in the complex case and $\mathrm{O}(m)$ in the real case. We derive the Riemannian gradient and Hessian of each reduced objective and implement them to run a second-order trust-region solver. To the best of our knowledge, this is the first method addressing the NPP in this generality; setting $X = I_m$ recovers a new Riemannian approach to the well-studied Closest Normal Matrix Problem and Real Closest Normal Matrix Problem.

Our results on the complex NPP emphasize our solver's consistency across starts and accuracy in finding global optima---demonstrating our method's ability to address the Normal Targeting Problem. Our method also matches the final residuals of both Ruhe's and Guglielmi and Scalone's algorithms on the CNP while scaling more favorably in wall-clock time.

The real case introduced a couple of hiccups for our solver, notably through the suboptimal final residuals obtained for both the Real NPP and Real CNP---likely due to an increased nonconcavity on the landscape of the real objective function. Further work could be done to address this issue, such as an improved initialization strategy (that still avoids degeneracies) or multiple starts. 

\section*{Statements and Declarations}

\subsection*{Funding}~\\[8pt]
\indent The author received no specific funding for this work.

\subsection*{Competing Interests}~\\[8pt]
\indent The author declares no competing interests.

\subsection*{Code Availability}~\\[8pt] 
\indent The code for the paper is freely available at \url{https://github.com/kbierly/NPP}. Note that Guglielmi and Scalone's implementation is not included by request.

\subsection*{Data Availability}~\\[8pt]
\indent All data reported in this article (i.e. our numerical findings and plots) are available at the above repository.

\printbibliography

@article{ltp,
title = {The linear targeting problem},
journal = {Linear Algebra and its Applications},
volume = {720},
pages = {91-108},
year = {2025},
issn = {0024-3795},
doi = {10.1016/j.laa.2025.04.017},
author = {Kyle Bierly and Stephan Ramon Garcia and Roger A. Horn}
}

@book{evans,
  author    = {Evans, Lawrence C. and Gariepy, Ronald F.},
  title     = {Measure Theory and Fine Properties of Functions},
  series    = {Textbooks in Mathematics},
  edition   = {Revised},
  publisher = {CRC Press},
  address   = {Boca Raton, FL},
  year      = {2015},
  isbn      = {978-1-4822-4238-6},
  doi       = {10.1201/b18333}
}

@article{manopt,
  author  = {Boumal, Nicolas and Mishra, Bamdev and Absil, P.-A. and Sepulchre, Rodolphe},
  title   = {Manopt, a {MATLAB} Toolbox for Optimization on Manifolds},
  journal = {Journal of Machine Learning Research},
  volume  = {15},
  pages   = {1455--1459},
  year    = {2014}
}

@article{rtr,
  author  = {Absil, P.-A. and Baker, C. G. and Gallivan, K. A.},
  title   = {Trust-Region Methods on {R}iemannian Manifolds},
  journal = {Foundations of Computational Mathematics},
  volume  = {7},
  number  = {3},
  pages   = {303--330},
  year    = {2007},
  doi     = {10.1007/s10208-005-0179-9}
}

@misc{zhukov,
  author       = {Zhukov, Matvei},
  title        = {Computing Nearest Normal Matrix Using {R}iemannian Optimization},
  howpublished = {Poster, Workshop on Numerical Linear Algebra,
                  Foundations of Computational Mathematics 2026,
                  Vienna, Austria},
  month        = jul,
  year         = {2026}
}

@article{noferini,
  title={Nearest $\Omega$-stable matrix via Riemannian optimization},
  author={Noferini, Vanni and Poloni, Federico},
  journal={Numerische Mathematik},
  volume={148},
  number={4},
  pages={817--851},
  year={2021},
  publisher={Springer},
  doi={10.1007/s00211-021-01217-4}
}

@book{gower,
    author = {Gower, John C. and Dijksterhuis, Garmt B.},
    title = {Procrustes Problems},
    publisher = {Oxford University Press},
    year = {2004},
    month = {01},
    isbn = {9780198510581},
    doi = {10.1093/acprof:oso/9780198510581.001.0001},
}

@Article{ruhe,
  author = {Ruhe, Axel},
  title = {Closest normal matrix finally found!},
  journal = {BIT Numerical Mathematics},
  year = {1987},
  volume = {27},
  pages = {585--598},
  doi = {10.1007/BF01937278},
}

@article{gillis,
title = {A semi-analytical approach for the positive semidefinite Procrustes problem},
journal = {Linear Algebra and its Applications},
volume = {540},
pages = {112-137},
year = {2018},
issn = {0024-3795},
doi = {10.1016/j.laa.2017.11.023},
author = {Nicolas Gillis and Punit Sharma}
}

@article{scalone,
author = {Guglielmi, Nicola and Scalone, Carmela},
title = {Computing the closest real normal matrix and normal completion},
year = {2019},
issue_date = {Dec 2019},
volume = {45},
number = {5–6},
issn = {1019-7168},
doi = {10.1007/s10444-019-09717-6},
journal = {Advances in Computational Mathematics},
month = dec,
pages = {2867–2891}
}

@inproceedings{higham,
  author =        {Nicholas J. Higham},
  booktitle =     {Applications of Matrix Theory},
  editor =        {M. J. C. Gover and S. Barnett},
  pages =         {1-27},
  publisher =     {Oxford University Press},
  title =         {Matrix Nearness Problems and Applications},
  year =          {1989},
  url =  {https://nhigham.com/wp-content/uploads/2023/10/high89n.pdf}
}

@book{horn,
  address = {Cambridge; New York},
  author = {Horn, Roger A. and Johnson, Charles R.},
  edition = {2nd},
  isbn = {9780521839402},
  publisher = {Cambridge University Press},
  title = {Matrix Analysis},
  year = 2013
}

@article{jiang,
title = {On normal extensions of submatrices},
journal = {Linear Algebra and its Applications},
volume = {370},
pages = {301-314},
year = {2003},
doi = {10.1016/S0024-3795(03)00414-2},
author = {Chung-Chou Jiang and Kung-Hwang Kuo}
}

@article{hurley,
author={Hurley, John R. and Cattell, Raymond B.},
year={1962},
month=apr,
title={The Procrustes Program: Producing Direct Rotation to Test a Hypothesized Factor Structure},
journal={Behavioral Science},
volume={7},
number={2},
pages={258},
issn={0005-7940},
url={https://www.proquest.com/scholarly-journals/procrustes-program-producing-direct-rotation-test/docview/1301270437/se-2},
}

@article{kovac,
title = {Finding the closest normal structured matrix},
journal = {Linear Algebra and its Applications},
volume = {617},
pages = {49-77},
year = {2021},
issn = {0024-3795},
doi = {10.1016/j.laa.2021.01.013},
author = {Erna {Begović Kovač}}
}

@book{boumal, place={Cambridge}, title={An Introduction to Optimization on Smooth Manifolds}, publisher={Cambridge University Press}, author={Boumal, Nicolas}, year={2023}}

@article{gabriel,
  author  = {Gabriel, Richard},
  title   = {The Normal {$\Delta H$}-Matrices with Connection to Some {Jacobi}-like Methods},
  journal = {Linear Algebra and its Applications},
  volume  = {91},
  pages   = {181--194},
  year    = {1987}
}

@article{bahar,
author = {Arslan, Bahar and Noferini, Vanni and Tisseur, Fran\c{c}oise},
title = {The Structured Condition Number of a Differentiable Map between Matrix Manifolds, with Applications},
journal = {SIAM Journal on Matrix Analysis and Applications},
volume = {40},
number = {2},
pages = {774-799},
year = {2019},
doi = {10.1137/17M1148943},
}

@article{henrici,
  author  = {Henrici, Peter},
  title   = {Bounds for iterates, inverses, spectral variation and fields of values of non-normal matrices},
  journal = {Numerische Mathematik},
  year    = {1962},
  volume  = {4},
  number  = {1},
  pages   = {24--40},
  doi     = {10.1007/BF01386294},
  issn    = {0945-3245}
}

@article{chu,
author = {Chu, Moody T.},
title = {Least Squares Approximation by Real Normal Matrices with Specified Spectrum},
journal = {SIAM Journal on Matrix Analysis and Applications},
volume = {12},
number = {1},
pages = {115-127},
year = {1991},
doi = {10.1137/0612009}
}

@misc{noferini2,
  title  = {Closest Normal Matrix Found Again Using Riemannian Optimization},
  author = {Vanni Noferini and Matvei Zhukov},
  year   = {2026},
  note   = {arXiv:2608.28545 [math.NA]},
}

@mastersthesis{barl,
  author  = {Barl, Filip},
  title   = {Higher-Rank Numerical Range of Almost Normal Matrices},
  school  = {Jagiellonian University},
  address = {Krak\'{o}w, Poland},
  year    = {2014},
  month   = jun,
  note    = {Supervisor: K.~\.{Z}yczkowski},
  url     = {https://chaos.if.uj.edu.pl/~karol/prace/Barl14.pdf}
}

@article{mityagin,
  author  = {Mityagin, Boris S.},
  title   = {The Zero Set of a Real Analytic Function},
  journal = {Mathematical Notes},
  volume  = {107},
  number  = {3},
  pages   = {529--530},
  year    = {2020},
  doi     = {10.1134/S0001434620030189}
}

@misc{callan,
  title  = {A combinatorial survey of identities for the double factorial},
  author = {Callan, David},
  year   = {2009},
  note   = {arXiv:0906.1317 [math.CO]}
}

\end{document}